\documentclass[11pt,oneside,english]{amsart}
\usepackage[latin9]{inputenc}
\usepackage{geometry}
\usepackage{color}
\usepackage{float}
\usepackage{mathrsfs}
\usepackage{amsthm}
\usepackage{amsbsy}
\usepackage{amssymb}
\usepackage{stackrel}
\usepackage{esint}

\makeatletter
\numberwithin{equation}{section}
\numberwithin{figure}{section}
\theoremstyle{plain}
\newtheorem{thm}{\protect\theoremname}
  \theoremstyle{plain}
  \newtheorem{prop}[thm]{\protect\propositionname}
  \theoremstyle{plain}
  \newtheorem{cor}[thm]{\protect\corollaryname}
  \theoremstyle{plain}
  \newtheorem{lem}[thm]{\protect\lemmaname}
  \theoremstyle{remark}
  \newtheorem{rem}[thm]{\protect\remarkname}

\usepackage[percent]{overpic}
\usepackage{graphicx}
\usepackage{epstopdf}

\makeatother

\usepackage{babel}
  \providecommand{\corollaryname}{Corollary}
  \providecommand{\lemmaname}{Lemma}
  \providecommand{\propositionname}{Proposition}
  \providecommand{\remarkname}{Remark}
\providecommand{\theoremname}{Theorem}

\begin{document}
\global\long\def\Frf{\Lambda_{Z}}
\global\long\def\FrfF{\Lambda_{F}}
\global\long\def\FrfFt{\Lambda_{F,2}}
\global\long\def\Cun{D}
\global\long\def\Fstar{F_{*}}
\global\long\def\Indicator{\mathbf{1}}
\global\long\def\ftail{\eta_{0}}
\global\long\def\ftailbeta{\eta_{0,j}}

\title{The geometry of the Gibbs measure of pure spherical spin glasses}

\author{Eliran Subag}
\begin{abstract}
We analyze the statics for pure $p$-spin spherical spin glass models
with $p\geq3$, at low enough temperature. With $F_{N,\beta}$ denoting
the free energy, we compute the second order (logarithmic) term of
$NF_{N,\beta}$ and prove that, for an appropriate centering $c_{N,\beta}$,
$NF_{N,\beta}-c_{N,\beta}$ is a tight sequence. We establish the
absence of temperature chaos and analyze the transition rate to disorder
chaos of the Gibbs measure and ground state. Those results follow
from the following geometric picture we prove for the Gibbs measure,
of interest by itself: asymptotically, the measure splits into infinitesimal
spherical `bands' centered at deep minima, playing the role of so-called
`pure states'. For the pure models, the latter makes precise the so-called
picture of `many valleys separated by high mountains' and significant
parts of the TAP analysis from the physics literature.
\end{abstract}

\maketitle

\section{Introduction}

The Hamiltonian of the pure $p$-spin spherical spin glass model is
given by 
\begin{equation}
H_{N}\left(\boldsymbol{\sigma}\right)=H_{N,p}\left(\boldsymbol{\sigma}\right)=\frac{1}{N^{\left(p-1\right)/2}}\sum_{i_{1},...,i_{p}=1}^{N}J_{i_{1},...,i_{p}}\sigma_{i_{1}}\cdots\sigma_{i_{p}},\quad\boldsymbol{\sigma}\in\mathbb{S}^{N-1},\label{eq:Hamiltonian}
\end{equation}
where $\boldsymbol{\sigma}=\left(\sigma_{1},...,\sigma_{N}\right)$
, $\mathbb{S}^{N-1}\triangleq\{\boldsymbol{\sigma}\in\mathbb{R}^{N}:\left\Vert \boldsymbol{\sigma}\right\Vert _{2}=\sqrt{N}\}$,
and $(J_{i_{1},...,i_{p}})$ are i.i.d standard normal variables.
Unless said otherwise, in the sequel we will always assume that $p\geq3$.
In terms of the overlap function $R\left(\boldsymbol{\sigma},\boldsymbol{\sigma}'\right)$,
the covariance of $H_{N}\left(\boldsymbol{\sigma}\right)$ is expressed
by 
\[
\mathbb{E}\left\{ H_{N}(\boldsymbol{\sigma})H_{N}(\boldsymbol{\sigma}')\right\} =N\left(R\left(\boldsymbol{\sigma},\boldsymbol{\sigma}'\right)\right)^{p}\mbox{\,\,\,\ with\,\,\,\,}R\left(\boldsymbol{\sigma},\boldsymbol{\sigma}'\right)=\frac{\boldsymbol{\sigma}\cdot\boldsymbol{\sigma}'}{\left\Vert \boldsymbol{\sigma}\right\Vert \left\Vert \boldsymbol{\sigma}'\right\Vert }.
\]

In this paper we carry out a rather thorough analysis of the statics
for pure $p$-spin spherical models with $p\geq3$, at low enough
temperature. As is well known, their free energy is given by the spherical
version of the Parisi formula discovered by Crisanti and Sommers \cite{pSPSG},
proved by Talagrand \cite{Talag} and extended by Chen \cite{Chen}.
We shall compute the free energy by a different method, improve the
latter by computing a logarithmic second order term, and prove that
the fluctuations of the free energy are tight (Theorem \ref{thm:free-energy}).
We will further prove the absence of temperature chaos (Theorem \ref{thm:temp_chaos})
and analyze the transition rate to disorder chaos of the Gibbs measure
and ground state (see Section \ref{sec:discussion}). 

Those results will follow from the following geometric picture for
the Gibbs measure, of interest by itself: asymptotically, the measure
splits into infinitesimal spherical `bands' centered at deep minima,
playing the role of so-called `pure states'. On those bands (the restriction
of) $H_{N}\left(\boldsymbol{\sigma}\right)$ will be interpreted as
a replica symmetric mixed spherical model, closely related to the
$2$-spin model at an (effective) high temperature. We note that the
disorder $(J_{i_{1},...,i_{p}})$ determines the locations of the
bands through those of the minima, and the radius of the bands is
determined by the temperature. This relates systems at different temperatures
or with perturbed disorder, a fact that will be crucial for us when
we investigate chaos phenomena.

We point out that the genericity of mixed $p$-spin models%
\footnote{\label{fn:generic}A mixed model $H_{N}(\boldsymbol{\sigma})=\sum_{p\geq2}\gamma_{p}H_{N,p}(\boldsymbol{\sigma})$,
either with spherical or Ising spins, is generic if and only if $\sum p^{-1}\Indicator\left(\gamma_{p}\neq0\right)=\infty$.%
} (see \cite{PanchenkoBook}) and the assumption that interactions
are even have been found to be very useful properties and were essential
ingredients in several recent works. The pure models we consider are
`as far as can be' from being generic and we do not assume $p$ is
even. In view of the very recent proof by Panchenko \cite{PanchenkoChaos}
that generic mixed models with even interactions exhibit temperature
chaos,%
\footnote{His proof dealt with Ising spin models, but his method is expected
to apply to spherical models as well.%
} the absence of chaos in pure models expresses a significant difference
between the pure and generic models (and is somewhat surprising).
On the other hand, the spherical pure models are known to exhibit
1-step replica symmetry breaking (RSB) in the low temperature phase
\cite[Proposition 2.2]{multipleoverlap}, and are simpler in this
respect than general mixed models.

For measurable $B\subset\mathbb{S}^{N-1}$, define the \emph{relative
partition function }or\emph{ relative mass }and the Gibbs measure,
respectively, by
\begin{equation}
Z_{N,\beta}\left(B\right)=\int_{B}\exp\left\{ -\beta H_{N}\left(\boldsymbol{\sigma}\right)\right\} d\mu_{N}\left(\boldsymbol{\sigma}\right)\mbox{\,\,\,\ and\,\,\,\,\,}G_{N,\beta}\left(B\right)=\frac{Z_{N,\beta}\left(B\right)}{Z_{N,\beta}\left(\mathbb{S}^{N-1}\right)},\label{eq:rel partition}
\end{equation}
where $\mu_{N}$ is the uniform probability measure on $\mathbb{S}^{N-1}$.
We also denote by $Z_{N,\beta}\triangleq Z_{N,\beta}\left(\mathbb{S}^{N-1}\right)$
the usual partition function. For a given point $\boldsymbol{\sigma}_{0}\in\mathbb{S}^{N-1}$
and overlaps $-1\leq q\leq q'\leq1$, define the spherical band
\begin{equation}
{\rm Band}\left(\boldsymbol{\sigma}_{0},q,q'\right)\triangleq\left\{ \boldsymbol{\sigma}\in\mathbb{S}^{N-1}:q\leq R\left(\boldsymbol{\sigma},\boldsymbol{\sigma}_{0}\right)\leq q'\right\} .\label{eq:band}
\end{equation}
A point $\boldsymbol{\sigma}_{0}$ is a critical point if $\nabla H_{N}\left(\boldsymbol{\sigma}_{0}\right)=0$
with respect to the standard differential structure on the sphere.
For odd $p$, let $\boldsymbol{\sigma}_{0}^{i}$, $i=1,2,..$, be
an enumeration of the critical points of $H_{N}\left(\boldsymbol{\sigma}\right)$
ordered so that $H_{N}(\boldsymbol{\sigma}_{0}^{i})\leq H_{N}(\boldsymbol{\sigma}_{0}^{i+1})$.
When $p$ is even, for any critical point $\boldsymbol{\sigma}_{0}$,
$-\boldsymbol{\sigma}_{0}$ is also a critical point with the same
critical value. In this case, let $\boldsymbol{\sigma}_{0}^{i}$,
$i=\pm1,\pm2,..$, be an enumeration such that $\boldsymbol{\sigma}_{0}^{-i}=-\boldsymbol{\sigma}_{0}^{i}$
and $H_{N}(\boldsymbol{\sigma}_{0}^{i})$ increases for $i\geq1$.%
\footnote{We note that though we work with critical points, by Corollary \ref{cor:CrtMin}
we could have replaced everywhere in the results $\boldsymbol{\sigma}_{0}^{i}$
by the corresponding enumeration of local minima instead of general
critical points.%
} In Section \ref{sec:overlapsVals} we will define the overlap value
$q_{*}:=q_{*}\left(\beta\right)$, see (\ref{eq:39}). We use it to
define
\begin{equation}
{\rm Band}_{i}\left(\epsilon\right):={\rm Band}_{i,\beta}\left(\epsilon\right)={\rm Band}\left(\boldsymbol{\sigma}_{0}^{i},q_{*}-\epsilon,q_{*}+\epsilon\right).\label{eq:B_sigma}
\end{equation}
We define the conditional measure of $G_{N,\beta}$ given ${\rm Band}_{i}\left(cN^{-1/2}\right)$,
\[
G_{N,\beta}^{c,i}\left(\cdot\right)=G_{N,\beta}\left(\,\cdot\,\cap{\rm Band}_{i}\left(cN^{-1/2}\right)\right)/G_{N,\beta}\left({\rm Band}_{i}\left(cN^{-1/2}\right)\right).
\]
Let $G_{N,\beta}^{c,i}\otimes G_{N,\beta}^{c,j}\left\{ \left(\boldsymbol{\sigma},\boldsymbol{\sigma}^{\prime}\right)\in\cdot\right\} $
denote the product measure of $G_{N,\beta}^{c,i}$ and $G_{N,\beta}^{c,j}$.
For odd $p$ define $[k]_{p}=\left\{ 1,...,k\right\} $ and for even
$p$ define $[k]_{p}=\left\{ \pm1,...,\pm k\right\} $. By an abuse
of notation, we will simply write $[k]$ in the sequel. 
\begin{thm}
\label{thm:Geometry}(Geometry of the Gibbs measure) For large enough
$\beta$ we have
\begin{enumerate}
\item \label{enu:Geometry1}Asymptotic support: 
\begin{equation}
\lim_{k,c\to\infty}\liminf_{N\to\infty}\mathbb{E}\left\{ G_{N,\beta}\left(\cup_{i\in[k]}{\rm Band}_{i}\left(cN^{-1/2}\right)\right)\right\} =1.\label{eq:thm1_1}
\end{equation}

\item \label{enu:Geometry2}States are pure: for any $i$ and $c>0$, for
even $p$, 
\begin{equation}
\lim_{\rho\to\infty}\limsup_{N\to\infty}\mathbb{P}\left\{ G_{N,\beta}^{c,i}\otimes G_{N,\beta}^{c,\pm i}\left\{ \left|R\left(\boldsymbol{\sigma},\boldsymbol{\sigma}^{\prime}\right)-\left(\pm q_{*}^{2}\right)\right|>\rho N^{-1/2}\right\} \right\} =0.\label{eq:2601}
\end{equation}
For odd $p$, the same holds with the $\pm$ signs removed.
\item \label{enu:Geometry3}Orthogonality of states: for any $i\neq\pm j$,
$c,\,\delta>0$
\begin{equation}
\lim_{N\to\infty}\mathbb{P}\left\{ G_{N,\beta}^{c,i}\otimes G_{N,\beta}^{c,j}\left\{ \left|R\left(\boldsymbol{\sigma},\boldsymbol{\sigma}^{\prime}\right)\right|>\delta\right\} \right\} =0.\label{eq:1402}
\end{equation}

\end{enumerate}
\end{thm}
The decomposition of Theorem \ref{thm:Geometry} is closely related
to the works of Talagrand \cite{TalagrandPstates} and Jagannath \cite{JagannathApxUlt}
who proved certain abstract `pure states' decompositions, assuming
that the Ghirlanda-Guerra identities are satisfied in a limiting sense.
The critical points and values of the Hamiltonian $H_{N}\left(\boldsymbol{\sigma}\right)$
have been recently investigated by Auffinger, Ben Arous and {\v{C}}ern{\'y}
\cite{A-BA-C}, the author \cite{2nd}, and Zeitouni and the author
\cite{pspinext}; see Section \ref{sub:crtpts}. In particular, in\textcolor{red}{{}
}\cite{pspinext} the law of the (deep) critical values $H_{N}(\boldsymbol{\sigma}_{0}^{i})$
was described in terms of a limiting point process (see Theorem \ref{thm:ext proc}),
which complements Theorem \ref{thm:Geometry}. Further, a key in proving
(\ref{eq:1402}) is that the deep critical points are either antipodal
or approximately orthogonal as vectors in the Euclidean space, see
Corollary \ref{cor:orth} which builds on \cite{2nd}.

In the physics literature, pure states are often described by the
so-called picture of `many valleys separated by high mountains' (see
e.g. \cite{ParisiOrderPar}). Our results (see Corollary \ref{cor:orth}
and (\ref{eq:08038})) indeed allow one to interpret the neighborhoods
of (exponentially many) critical points corresponding to critical
values deeper than a certain fraction of the global minimum of $H_{N}(\boldsymbol{\sigma})$
as valleys. Therefore, Theorem \ref{thm:Geometry} validates the prediction
of multiple valleys, at least in the current setting, and further
identifies the valleys around the deepest critical points as the relevant
ones and bands as the relevant regions inside them. 

The Thouless-Anderson-Palmer (TAP) approach \cite{TAP} suggests that
the pure states are related to the solutions of the so-called TAP
equations, and that by correctly attributing mass to states at a given
energy  and estimating how many states there are at any energy --
i.e., computing the so-called TAP free energy and complexity, respectively
-- one can calculate the free energy. Kurchan, Parisi and Virasoro
\cite{KurchanParisiVirasoro} and Crisanti and Sommers \cite{CrisantiSommersTAPpspin}
carried out the TAP analysis of pure spherical models (their analysis
is not rigorous, and neither is claimed to be). Interestingly, for
pure spherical spin glasses, TAP solutions are nothing but the critical
points of the Hamiltonian (see \cite[Section 6]{A-BA-C}).%
\footnote{A rigorous computation of the annealed TAP complexity was performed
by Auffinger, Ben Arous and {\v{C}}ern{\'y} \cite{A-BA-C}; the fact
that it is valid quenched for low energies follows from \cite{2nd}.%
} One may therefore wonder whether the mass of bands we compute coincides
with the TAP free energy of \cite{KurchanParisiVirasoro,CrisantiSommersTAPpspin}.
As we shall see, this is indeed the case, at least in the relevant
range of overlaps; see Remark \ref{rem:18}.

As part of our investigation of the weights and structure of the Gibbs
measure on the thin bands of (\ref{eq:thm1_1}) we will study the
conditional law of the restriction of the Hamiltonian $H_{N}\left(\boldsymbol{\sigma}\right)$
to the sub-sphere $\left\{ \boldsymbol{\sigma}:R\left(\boldsymbol{\sigma},\boldsymbol{\sigma}_{0}\right)=q_{*}\right\} $
of co-dimension $1$, conditional on $\nabla H_{N}\left(\boldsymbol{\sigma}_{0}\right)=0$
and $H_{N}\left(\boldsymbol{\sigma}_{0}\right)=u$ for some level
$u\in\mathbb{R}$. The latter, we shall see, is identical in distribution
to a certain mixed spherical model involving $k$-spin interactions
with $2\leq k\leq p$ only (see Corollary \ref{cor:conditional laws}),
shifted by a constant. The Onsager reaction term added in \cite{KurchanParisiVirasoro,CrisantiSommersTAPpspin}
to what they call the `naive' free energy will arise in our calculation
as the free energy of this mixture. Moreover, we will see that the
fluctuations of the free energy of the original pure $p$-spin model
on the sphere are intimately related to those of the $2$-spin component
of the mixture. The convergence of the free energy for the spherical
$2$-spin model proved by Baik and Lee \cite{BaikLee} (see Theorem
\ref{thm:BaikLee}) will be crucial to analyzing the fluctuations
of the latter. Apart from the results stated in Section \ref{sec:earlierworks}
(which include the convergence result of \cite{BaikLee} and results
on the critical points from \cite{A-BA-C,2nd,pspinext}), our analysis
is essentially self contained.

\subsection*{\label{sub:free-energy}The free energy}

The free energy is defined by $F_{N,\beta}=\frac{1}{N}\log\left(Z_{N,\beta}\right)$.
The following theorem gives the second order correction of $NF_{N,\beta}$
and shows that, appropriately centered, $NF_{N,\beta}$ is tight. 
\begin{thm}
\label{thm:free-energy}For large enough $\beta$, with $\Frf\left(E,q\right)$,
$E_{0}$ and $c_{p}$ defined by (\ref{eq:6}), (\ref{eq:E0}) and
(\ref{eq:c_p}), 
\begin{equation}
\lim_{t\to\infty}\limsup_{N\to\infty}\mathbb{P}\left\{ \left|NF_{N,\beta}-N\Frf\left(E_{0},q_{*}\right)+\frac{\beta q_{*}^{p}}{2c_{p}}\log N\right|>t\right\} =0.\label{eq:F}
\end{equation}

\end{thm}
Earlier results regarding fluctuations of the free energy of mean
field spin glass models are surveyed in Section \ref{sub:FEflucts}.
In particular, the above is the first result proving that fluctuations
are of order $O(1)$ in the low-temperature phase. Theorem \ref{thm:free-energy}
implies that $F_{N,\beta}$ converges in probability to $\Frf\left(E_{0},q_{*}\right)$.
Obviously, the latter must coincide with the expression given by the
spherical version of the Parisi formula discovered by Crisanti and
Sommers \cite{pSPSG}, proved by Talagrand \cite{Talag} and extended
by Chen \cite{Chen}. We compare $\Frf\left(E_{0},q_{*}\right)$ and
the 1-step RSB Parisi functional by a direct calculation in Section
\ref{sec:discussion}. As $\beta\to\infty$ we obtain the limiting
law of the free energy in the following proposition.
\begin{prop}
\label{prop:large beta}There exist deterministic $a_{N,\beta}$ such
that, for any $t\in\mathbb{R}$, 
\begin{equation}
\lim_{\beta\to\infty}\limsup_{N\to\infty}\left|\mathbb{P}\left\{ \frac{1}{\beta}\left(NF_{N,\beta}-a_{N,\beta}\right)\leq t\right\} -\exp\left\{ -c_{p}^{-1}e^{-c_{p}t}\right\} \right|=0.\label{eq:F-2}
\end{equation}

\end{prop}
In fact, we will prove the above with $a_{N,\beta}=\log(\mathfrak{V}_{N,\beta}(m_{N}))$
where $\mathfrak{V}_{N,\beta}\left(u\right)$ and $m_{N}$ are defined
by (\ref{eq:V(u)}) and (\ref{eq:m_N}). We finish with the following
representation for the free energy in the $N\to\infty$ limit.
\begin{cor}
\label{cor:FE_u_Hess_rep}For large enough $\beta$, there exist random
variables $Z_{N,i}$, such that each $Z_{N,i}$ is measurable with
respect to $\left(H_{N}(\boldsymbol{\sigma}_{0}^{i}),\nabla^{2}H_{N}(\boldsymbol{\sigma}_{0}^{i})\right)$,
and a sequence $k_{N}\geq1$ with $k_{N}\to\infty$, so that 
\[
\forall\epsilon>0:\,\,\,\lim_{N\to\infty}\mathbb{P}\left\{ \left|NF_{N,\beta}-\log\sum_{i\in[k_{N}]}Z_{N,i}\right|>\epsilon\right\} =0.
\]

\end{cor}
See (\ref{eq:ZNi}) for an explicit expression for $Z_{N,i}$ (as
a conditional mass of a band).

\subsection*{Absence of temperature chaos}

Chaos phenomena, discovered by Bray and Moore \cite{BrayMooreChaos}
and Fisher and Huse \cite{FisherHuseChaos}, have been studied extensively
in the physics literature; see the recent survey \cite{Rizzo2009}
by Rizzo. They refer to the situation where a small perturbation in
the parameters of the system results in a drastic change in macroscopic
observables. In particular, we say that temperature chaos occurs if
for any $\beta_{1}\neq\beta_{2}$, 
\begin{equation}
\exists q_{0}\in\left[-1,1\right],\,\forall\epsilon>0:\,\,\,\lim_{N\to\infty}\mathbb{E}\left\{ G_{N,\beta_{1}}\otimes G_{N,\beta_{2}}\left\{ \left|R\left(\boldsymbol{\sigma},\boldsymbol{\sigma}^{\prime}\right)-q_{0}\right|>\epsilon\right\} \right\} =0.\label{eq:chaos}
\end{equation}
That is, we sample at two different temperatures with the disorder
fixed. We prove that spherical pure $p$-spin models \emph{do not}
exhibit temperature chaos, verifying a prediction of Rizzo and Yoshino
\cite{RizzoYoshino}. To the best of knowledge, this is the first
example where absence of temperature chaos is proved rigorously for
mean-field spin glass models. 
\begin{thm}
\label{thm:temp_chaos}For large enough $\beta_{1}<\beta_{2}$, (\ref{eq:chaos})
does not hold.
\end{thm}
In contrast, very recently Panchenko \cite{PanchenkoChaos} proved
that generic Ising mixed even $p$-spin models exhibit temperature
chaos (see also the related works of Chen \cite{ChenChaos} and Chen
and Panchenko \cite{ChenPanchChaos}). His methods are also expected
to work for spherical models. For pure spherical $p$-spin models
with even $p\geq4$, Panchenko and Talagrand \cite{multipleoverlap}
proved that for any $\epsilon>0$, $\lim_{N\to\infty}\mathbb{E}\left\{ G_{N,\beta_{1}}\otimes G_{N,\beta_{2}}\left\{ R\left(\boldsymbol{\sigma},\boldsymbol{\sigma}^{\prime}\right)\in A(\epsilon)\right\} \right\} =1$
with $A(\epsilon)$ denoting the union of balls of radius $\epsilon$
around $0$, $q_{12}$ and $-q_{12}$, with $q_{12}=q_{*}(\beta_{1})q_{*}(\beta_{2})$,
assuming $\beta_{1}$ and $\beta_{2}$ are in the low-temperature
phase. By itself, this is not enough to conclude or rule out chaos.
Our proof shows that in the $N\to\infty$ limit, any of those three
values is charged. See Proposition \ref{prop:temp_chaos} and Remark
\ref{rem:multipleoverlap}. The proof of Theorem \ref{thm:temp_chaos}
will be based on the fact that the centers of the bands that carry
most of the mass do not change with temperature, assuming the temperature
is low enough. 

Theorem \ref{thm:Geometry} and an investigation of the critical points
of the Hamiltonian can be also used to study disorder chaos of the
Gibbs measure and the ground state. We discuss this in Section \ref{sec:discussion}.

\subsection*{Structure of the paper}

In Section \ref{sec:earlierworks} we review earlier related works.
Section \ref{sec:outline} is dedicated to an outline of the proof
of Theorem \ref{thm:Geometry}. In Section \ref{sec:Decomposition}
we develop a certain decomposition for the Hamiltonian and study conditional
laws related to it. Various overlap values of importance are defined
in Section \ref{sec:overlapsVals}. In particular, they are used to
define two ranges of overlaps for which  conditional weights of corresponding
bands around a single critical point are investigated by different
methods in Sections \ref{sec:range(q**,1)} and \ref{sec:range(q***,q**)}.
Those are used in Section \ref{sec:Bounds-on-contributions} to derive
bounds on contributions to the partition function of the corresponding
bands around all (low enough) critical points. The latter are combined
in Section \ref{sec:thm1} to prove Theorem \ref{thm:Geometry}. The
proofs of Theorem \ref{thm:free-energy}, Proposition \ref{prop:large beta}
and Corollary \ref{cor:FE_u_Hess_rep}, which deal with the free energy,
are given in Section \ref{sec:thmFE}. The proof of the absence of
temperature chaos, Theorem \ref{thm:temp_chaos}, is given in Section
\ref{sec:chaos}. Section \ref{sec:discussion} is dedicated to concluding
remarks including, in particular, a discussion about the transition
to disorder chaos.

\section*{Acknowledgments}

I am grateful to my adviser Ofer Zeitouni for many very fruitful conversations
and for carefully reading the manuscript. I would also like to thank
G\'{e}rard Ben Arous and Aukosh Jagannath for helpful discussions.
This work is supported by the Adams Fellowship Program of the Israel
Academy of Sciences and Humanities and by a US-Israel BSF grant.

\section{\label{sec:earlierworks}Earlier and related works}

The first section below surveys works related to Theorem \ref{thm:free-energy}.
The second section is devoted to recent results about the critical
points and values of the Hamiltonian $H_{N}(\boldsymbol{\sigma})$,
directly related to Theorem \ref{thm:Geometry} and frequently used
in the sequel.

\subsection{\label{sub:FEflucts}Fluctuations of the free energy}

First we state two result regarding the free energy of the $2$-spin
model which will be crucial when we study the weights of bands. The
first is obtained as a corollary from Talagrand's proof \cite{Talag}
of the spherical version of the Parisi formula.
\begin{cor}
\label{cor:2spinFE}\cite[Theorem 1.1, Proposition 2.2]{Talag} For
the pure $2$-spin spherical model, the limiting free energy is given
by $\lim_{N\to\infty}F_{N,\beta}=\mathscr{P}_{2}\left(\beta\right)$
where
\begin{equation}
\mathscr{P}_{2}\left(\beta\right)=\begin{cases}
\frac{1}{2}\beta^{2} & \,\,\mbox{if }\beta\leq1/\sqrt{2},\\
\sqrt{2}\beta-\frac{3}{4}-\frac{1}{2}\log\left(\beta\right)-\frac{1}{4}\log2 & \,\,\mbox{if }\beta>1/\sqrt{2}.
\end{cases}\label{eq:P2}
\end{equation}

\end{cor}
The following convergence was recently proved by Baik and Lee \cite{BaikLee}.%
\footnote{\label{fn:BaikLee}Below we will apply \cite[Theorem 1.2]{BaikLee}
for the $2$-spin model defined by (\ref{eq:Hamiltonian}) with $p=2$
and the inverse-temperature $\beta_{{\rm eff}}=\beta_{{\rm eff}}\left(N,q_{*}\right)$
defined in (\ref{eq:beff-1}). In the notation of \cite[Theorem 1.2]{BaikLee},
this corresponds to centered Gaussian $J_{ij}=J_{ji}$ with variance
$2$ if $i=j$ and $1$ if $i\neq j$ and inverse-temperature $\beta_{{\rm eff}}/\sqrt{2}$.
Since $\beta_{{\rm eff}}\in\left(0,1/\sqrt{2}\right)$, in our application
we will only use (\ref{eq:BaikLee-high}).%
} We denote by $\overset{d}{\to}$ convergence in distribution and
by $\mathcal{N}\left(m,\sigma^{2}\right)$ the Gaussian distribution
with mean $m$ and variance $\sigma^{2}$.
\begin{thm}
\label{thm:BaikLee}\cite[Theorem 1.2]{BaikLee} For the pure $2$-spin
spherical model, with $f=\frac{1}{4}\log\left(1-2\beta^{2}\right)$,
$\alpha=-2f$, and $TW_{1}$ denoting the GOE Tracy-Widom distribution,
\begin{align}
 & \forall\beta\in\left(0,1/\sqrt{2}\right):\,\,\, N\left(F_{N,\beta}-\mathscr{P}_{2}\left(\beta\right)\right)\stackrel[{\scriptstyle N\to\infty}]{d}{\longrightarrow}\mathcal{N}\left(f,\alpha\right),\label{eq:BaikLee-high}\\
 & \forall\beta\in\left(1/\sqrt{2},\infty\right):\,\,\,\left(\frac{1}{\sqrt{2}}\beta-\frac{1}{2}\right)^{-1}N^{2/3}\left(F_{N,\beta}-\mathscr{P}_{2}\left(\beta\right)\right)\stackrel[{\scriptstyle N\to\infty}]{d}{\longrightarrow}TW_{1}.\label{eq:BaikLee-low}
\end{align}

\end{thm}
As for the fluctuations of the free energy in other spin glass models,
we mention the following. In the high-temperature phase, Aizenman,
Lebowitz and Ruelle \cite{AizenmanLebowitzRuelle} proved the convergence
of $N(F_{N,\beta}-C_{N,\beta})$, where $C_{N,\beta}$ is an appropriate
centering, to a Gaussian variable for the Sherrington-Kirkpatrick
(SK) model. Comets and Neveu \cite{CometsNeveu} later proved a similar
result by a different approach using martingale methods. For pure
Ising $p$-spin models with even $p\geq4$, again in the high-temperature
phase, Bovier, Kurkova and L{\"o}we \cite{BovierKulkovaLowe}, proved
similar convergence for $N^{\left(p+2\right)/4}(F_{N,\beta}-C_{N,\beta})$
by adapting the method of \cite{CometsNeveu}. With a different definition
for the model, dropping the diagonal terms in (\ref{eq:Hamiltonian}),
(but still working with Ising spins, $\boldsymbol{\sigma}\in\left\{ \pm1\right\} ^{N}$
and at high-temperature) they show for $p\geq3$ the convergence of
$N^{p/2}(F_{N,\beta}-C_{N,\beta})$. At any temperature, Chatterjee
\cite{ChatterjeeDisChaos} showed that for Ising mixed even $p$-spin
models without an external field, ${\rm Var}\left(NF_{N,\beta}\right)\leq c_{\beta}N/\log N$.
For Ising mixed $p$-spin models in the presence of an external field
and at any temperature, Chen, Dey and Panchenko \cite{ChenDeyPanchenko}
showed that ${\rm Var}\left(NF_{N,\beta}\right)\leq c_{\beta}N$.
When assuming in addition that there are no odd $p$-spin interactions,
they also showed convergence to a Gaussian variable. As they remark,
their approach should also work for spherical models.

\subsection{\label{sub:crtpts}Critical points and values}

For $B\subset\mathbb{R}$ define 
\begin{equation}
\mathscr{C}_{N}\left(B\right)\triangleq\left\{ \boldsymbol{\sigma}:\nabla H_{N}\left(\boldsymbol{\sigma}\right)=0,\, H_{N}\left(\boldsymbol{\sigma}\right)\in B\right\} .\label{eq:crt}
\end{equation}
By an abuse of notation we will also write $\mathscr{C}_{N}\left(a,b\right)$
for $\mathscr{C}_{N}\left(\left(a,b\right)\right)$. In the seminal
work \cite{A-BA-C} Auffinger, Ben Arous and {\v{C}}ern{\'y} proved
the following (see also \cite{ABA2} for the mixed case). 
\begin{thm}
\label{thm:ABAC}\cite[Theorem 2.8]{A-BA-C} Assume $p\geq3$. For
any $E\in\mathbb{R}$, there exists $\Theta_{p}(E)$ (defined in (\ref{eq:Theta_p})
below) so that 
\begin{equation}
\lim_{N\to\infty}\frac{1}{N}\log\left(\mathbb{E}\left|\mathscr{C}_{N}\left(-\infty,NE\right)\right|\right)=\Theta_{p}\left(E\right).\label{eq:1st_mom}
\end{equation}

\end{thm}
Set $E_{\infty}=2\sqrt{\left(p-1\right)/p}$ and 
\begin{equation}
\mbox{let \ensuremath{E_{0}>E_{\infty}}\,\ be the unique number satisfying \ensuremath{\Theta_{p}\left(-E_{0}\right)=0}.}\label{eq:E0}
\end{equation}
Critical points of a given index were also considered in \cite{A-BA-C}
(the index of a critical point is the number of negative eigenvalues
of the Hessian at that point). By Markov's inequality one has the
following.
\begin{cor}
\label{cor:CrtMin}\cite[Theorem 2.5]{A-BA-C} For $p\geq3$, there
exists a number $E_{1}\in\left(E_{\infty},E_{0}\right)$ such that
for any $\delta>0$ there exists $c(\delta)>0$ for which 
\[
\mathbb{P}\left\{ \exists\boldsymbol{\sigma}\in\mathscr{C}_{N}\left(-\infty,-N\left(E_{1}+\delta\right)\right):\,\boldsymbol{\sigma}\mbox{ is not a local min}\right\} <e^{-c(\delta)N}.
\]

\end{cor}
By a second moment computation for $\left|\mathscr{C}_{N}\left(NB\right)\right|$
(where $NB=\{Nx:\, x\in B\}$), concentration of the number of critical
points around its mean was proved in \cite{2nd}.
\begin{thm}
\label{thm:2ndConcentration}\cite[Corollary 2]{2nd} For $p\geq3$
and $E\in\left(-E_{0},-E_{\infty}\right)$,
\[
\lim_{N\to\infty}\frac{\left|\mathscr{C}_{N}\left(-\infty,NE\right)\right|}{\mathbb{E}\left|\mathscr{C}_{N}\left(-\infty,NE\right)\right|}=1,\,\,\,\mbox{in }L^{2}.
\]

\end{thm}
Set 
\begin{align}
m_{N} & =-E_{0}N+\frac{1}{2c_{p}}\log N-K_{0},\label{eq:m_N}\\
c_{p} & =\left.\frac{d}{dx}\right|_{x=-E_{0}}\Theta_{p}\left(x\right)=E_{0}-\frac{2}{E_{\infty}^{2}}\left(E_{0}-\sqrt{E_{0}^{2}-E_{\infty}^{2}}\right),\label{eq:c_p}
\end{align}
where $K_{0}$ is given in \cite[Eq. (2.6)]{pspinext}. Zeitouni and
the author \cite{pspinext} proved the convergence of the extremal
process of critical points defined by 
\begin{equation}
\xi_{N}\triangleq\left(1+\iota_{p}\right)^{-1}\sum_{\boldsymbol{\sigma}\in\mathscr{C}_{N}\left(-\infty,\infty\right)}\delta_{H_{N}\left(\boldsymbol{\sigma}\right)-m_{N}},\label{eq:xi_N}
\end{equation}
where $\iota_{p}=\left(1+\left(-1\right)^{p}\right)/2$ (normalizing
the weights so that $\xi_{N}$ is a simple point process a.s. for
even $p$). Let $PPP\left(\mu\right)$ denote the distribution of
a Poisson point process with intensity measure $\mu$ and endow the
space of point processes with the vague topology.
\begin{thm}
\label{thm:ext proc}\cite[Theorem 1]{pspinext} For $p\geq3$,
\begin{equation}
\xi_{N}\stackrel[{\scriptstyle N\to\infty}]{d}{\to}\xi_{\infty}\sim PPP\left(e^{c_{p}x}dx\right).\label{eq:xi_infty}
\end{equation}
\end{thm}
\begin{cor}
\label{cor:min}\cite[Theorem 1, Corollary 2]{pspinext} For $p\geq3$,
$H_{N}(\boldsymbol{\sigma}_{0}^{1})=\min_{\boldsymbol{\sigma}}H_{N}(\boldsymbol{\sigma})$
converges to the negative of a Gumbel variable, namely, $\mathbb{P}\{H_{N}(\boldsymbol{\sigma}_{0}^{1})-m_{N}\geq x\}\to\exp\left\{ -c_{p}^{-1}e^{c_{p}x}\right\} $
as $N\to\infty$. Moreover, 
\begin{align*}
\forall k\geq1,\,\, & \lim_{L\to\infty}\lim_{N\to\infty}\mathbb{P}\left\{ \mathscr{C}_{N}\left(m_{N}-L,m_{N}+L\right)\supset\left\{ \boldsymbol{\sigma}_{0}^{i}:\, i\in[k]\right\} \right\} =1,\\
\forall L>0,\,\, & \lim_{k\to\infty}\lim_{N\to\infty}\mathbb{P}\left\{ \mathscr{C}_{N}\left(m_{N}-L,m_{N}+L\right)\subset\left\{ \boldsymbol{\sigma}_{0}^{i}:\, i\in[k]\right\} \right\} =1.
\end{align*}

\end{cor}
Another consequence of the second moment calculation \cite{2nd} and
the bound on the minimum is the following bound on overlaps of critical
points.%
\footnote{Corollary \ref{cor:orth} follows from \cite[Eq. (5.2), (5.3)]{pspinext}
and since it is shown in the proof of \cite[Proposition 4]{pspinext}
that \cite[Eq. (5.2)]{pspinext} is negative.%
}
\begin{cor}
\label{cor:orth}For $p\geq3$, for any $\epsilon>0$ there exist
$\delta(\epsilon),\, c(\epsilon)>0$ such that 
\begin{equation}
\mathbb{P}\left\{ \exists\boldsymbol{\sigma},\boldsymbol{\sigma}'\in\mathscr{C}_{N}\left(N\left(-E_{0}-\delta(\epsilon),-E_{0}+\delta(\epsilon)\right)\right),\,\boldsymbol{\sigma}\neq\pm\boldsymbol{\sigma}':\,\left|R\left(\boldsymbol{\sigma},\boldsymbol{\sigma}'\right)\right|\geq\epsilon\right\} <e^{-c(\epsilon)N}.\label{eq:s1}
\end{equation}
Moreover, there exists a sequence $\epsilon_{N}>0$ with $\epsilon_{N}\to0$
as $N\to\infty$ such that 
\[
\lim_{N\to\infty}\mathbb{P}\left\{ \exists\boldsymbol{\sigma},\boldsymbol{\sigma}'\in\mathscr{C}_{N}\left(-\infty,m_{N}+\sqrt{N}\right),\,\boldsymbol{\sigma}\neq\pm\boldsymbol{\sigma}':\,\left|R\left(\boldsymbol{\sigma},\boldsymbol{\sigma}'\right)\right|\geq\epsilon_{N}\right\} =0.
\]

\end{cor}

\section{\label{sec:outline}Outline of the proof of Theorem \ref{thm:Geometry}}

Schematically, we think of an overlap-depth plane by defining the
contribution of $A\times B\subset\left[-1,1\right]\times\mathbb{R}$
to $Z_{N,\beta}$ as 
\begin{equation}
{\rm Cont}_{N,\beta}\left(A\times B\right)=Z_{N,\beta}\left(\cup_{\boldsymbol{\sigma}_{0}\in\mathscr{C}_{N}\left(B\right)}\left\{ \boldsymbol{\sigma}:R\left(\boldsymbol{\sigma},\boldsymbol{\sigma}_{0}\right)\in A\right\} \right),\label{eq:cont}
\end{equation}
The basic picture we prove about (\ref{eq:cont}) is that as one `scans'
possible depths $u=EN$ for critical points $\boldsymbol{\sigma}_{0}$
and overlaps $q$ (in some range, as we explain below), the maximal
value for the contribution (\ref{eq:cont}) with $A=\left(q,q+o(1)\right)$
and $B=\left(u,u+o(1)\right)$ is obtained with high probability (w.h.p)
when $q=q_{*}$ and $u=-E_{0}N$. Of course, ${\rm Cont}_{N,\beta}\left(\cdot\right)$
is not additive, since bands can intersect each other, and exploring
the whole range of $\left[-1,1\right]\times\mathbb{R}$ therefore
makes no sense; it is additive, however, when restricted to small
enough range of overlaps around $1$ and low enough depths (clearly,
at least when we allow those ranges to depend on the disorder). The
next section explains how we restrict to a small range of overlaps
containing $q_{*}$ and range of depths near $u=-E_{0}N$.

\subsection*{Restriction to caps}

Observe that, with $B=\{\boldsymbol{\sigma}:\, H_{N}\left(\boldsymbol{\sigma}\right)\geq u\}$
being the super-level set of $H_{N}\left(\boldsymbol{\sigma}\right)$
with some level $u\in\mathbb{R}$, we have that $Z_{N,\beta}\left(B\right)\leq e^{-\beta u}$.
In Lemma \ref{lem:LSnegligible}, we shall prove that for an appropriate
choice of $u_{{\rm LS}}$ and $q_{{\rm LS}}$ (see (\ref{eq:qLS})
and (\ref{eq:uLS})) the sub-level set of $u_{{\rm LS}}+\delta N$,
with small $\delta>0$, is covered by $\cup{\rm Cap}\left(\boldsymbol{\sigma}_{0},q_{{\rm LS}}\right)$,
where the union goes over $\boldsymbol{\sigma}_{0}\in\mathscr{C}_{N}\left(-\infty,u_{{\rm LS}}+\delta N\right)$
and we define the spherical caps 
\begin{equation}
{\rm Cap}\left(\boldsymbol{\sigma}_{0},q\right)\triangleq\left\{ \boldsymbol{\sigma}\in\mathbb{S}^{N-1}:q\leq R\left(\boldsymbol{\sigma},\boldsymbol{\sigma}_{0}\right)\right\} .\label{eq:cap}
\end{equation}
We shall also prove a lower bound on the contribution coming from
overlaps roughly $q_{*}$ (see (\ref{eq:Reg*bd})) by which for the
same choice of $u_{{\rm LS}}$ the free energy is w.h.p greater than
$e^{-\beta u_{{\rm LS}}-o\left(N\right)}$. Hence, w.h.p the mass
of the complement of the caps is negligible. That is, when analyzing
contributions we may restrict to 
\[
{\rm Reg}_{{\rm LS}}\left(\delta\right)=\left(q_{{\rm LS}},1\right)\times\left(-\infty,u_{{\rm LS}}+\delta N\right).
\]
(See Figure \ref{fig:3.1} for a graphical description.) In addition,
from bounds on the global minimum of $H_{N}\left(\boldsymbol{\sigma}\right)$
(Corollary \ref{cor:min}), with probability going to $1$ as $\kappa'\to\infty$,
for large $N$ the contribution of
\[
{\rm Reg}_{{\rm GS}}\left(\kappa'\right)=\left(-1,1\right)\times\left(-\infty,m_{N}-\kappa'\right]
\]
is equal to $0$ w.h.p, where $m_{N}$ is defined in (\ref{eq:m_N}).
Combining the above with bounds on the lowest critical values of $H_{N}\left(\boldsymbol{\sigma}\right)$
(see Corollary \ref{cor:min}), we will see that (\ref{eq:thm1_1})
follows if we prove:
\begin{enumerate}
\item an appropriate lower bound on ${\rm Cont}_{N,\beta}\left({\rm Reg}_{*}\left(c,\kappa,\kappa'\right)\right)$
where 
\begin{equation}
{\rm Reg}_{*}\left(c,\kappa,\kappa'\right)=\left(q_{*}-cN^{-1/2},q_{*}+cN^{-1/2}\right)\times\left(m_{N}-\kappa',m_{N}+\kappa\right);\label{eq:Reg*}
\end{equation}

\item upper bounds on ${\rm Cont}_{N,\beta}\left(A_{i}\times B_{i}\right)$
for some collection of sets $A_{i}\times B_{i}$ covering 
\begin{align}
 & {\rm Reg}_{{\rm UB}}\left(\delta,c,\kappa,\kappa'\right)={\rm Reg}_{{\rm LS}}\left(\delta\right)\setminus\left({\rm Reg}_{*}\left(c,\kappa,\kappa'\right)\cup{\rm Reg}_{{\rm GS}}\left(\kappa'\right)\right)\label{eq:RegUB}\\
 & =\left(q_{{\rm LS}},1\right)\times\left(m_{N}-\kappa',u_{{\rm LS}}+\delta N\right)\setminus\left(q_{*}-cN^{-1/2},q_{*}+cN^{-1/2}\right)\times\left(m_{N}-\kappa',m_{N}+\kappa\right).\nonumber 
\end{align}

\end{enumerate}
In Section \ref{sub:Overlap-depth-regions} we split the latter to
several regions, dealt with separately, corresponding to critical
values near and far from $m_{N}$ and overlaps near and far from $1$,
see Figure \ref{fig:8.1}.

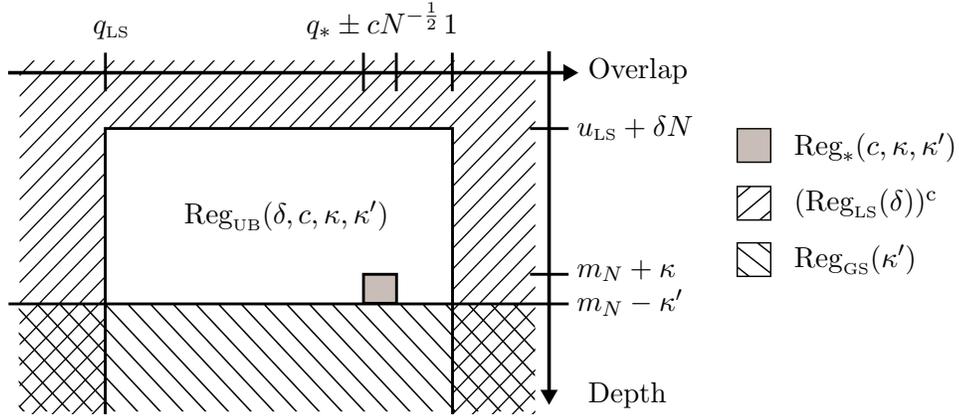
\begin{figure}[H]
~\\
~\\
~\\
%\begin{overpic}[grid,tics=10,width=0.65\textwidth]{Reg1png}
\hspace*{-1.7cm}\begin{overpic}[width=0.65\textwidth]{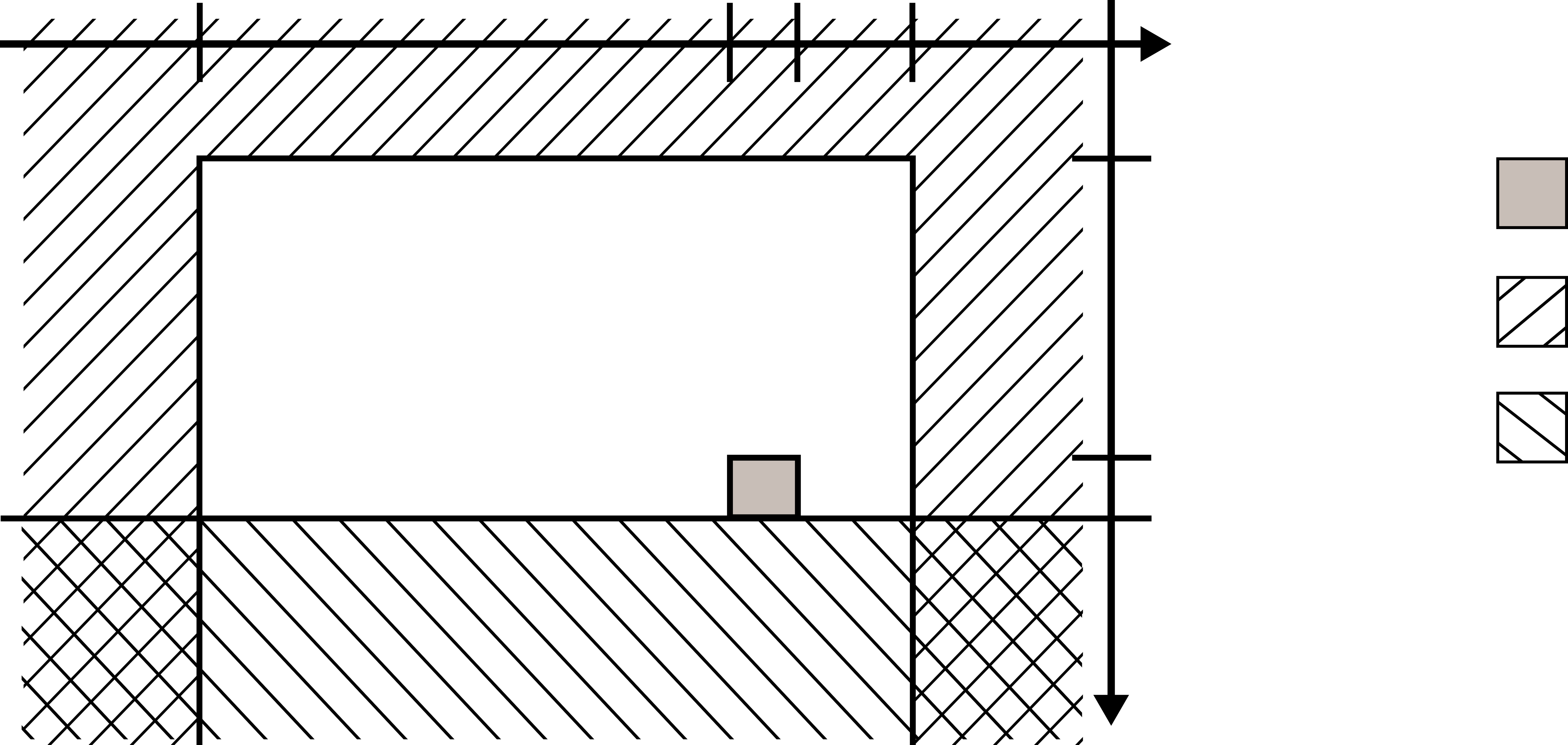}
\put (57,50) {$1$} 
\put (39,50) {$q_*\pm cN^{-\frac{1}{2}}$}
\put (11,50) {$q_{\scriptscriptstyle\rm{LS}}$}
\put (76,44) {Overlap} 
\put (74.5,36.5) {$u_{\scriptscriptstyle\rm{LS}}+\delta N$} 
\put (74.5,18) {$m_N+\kappa$} 
\put (74.5,13.5) {$m_N-\kappa'$} 
\put (103,34) {$\mbox{Reg}_*(c,\kappa,\kappa')$} 
\put (103,26.65) {$(\rm{Reg}_{\scriptscriptstyle\rm{LS}}(\delta))^c$} 
\put (103,19.30) {$\rm{Reg}_{\scriptscriptstyle\rm{GS}}(\kappa')$} 
\put (76,1.5) {Depth} 
\put (23,25) {$\mbox{Reg}_{\scriptscriptstyle\rm{UB}}(\delta,c,\kappa,\kappa')$}  
\end{overpic}\protect\caption{\label{fig:3.1}Regions of overlap and depth.}
\end{figure}

\subsection*{The Kac-Rice formula}

The basic tool we use for deriving bounds is the Kac-Rice formula;
see Appendix I. The formula will allow us to bound the expected number
of critical points $\boldsymbol{\sigma}_{0}$ that satisfy various
conditions by integrals involving the intensity of the empirical measure
of critical values and conditional probabilities that the (arbitrarily
chosen) point
\[
\hat{\mathbf{n}}\triangleq\left(0,...,0,\sqrt{N}\right),
\]
satisfies the aforementioned conditions, given that
\begin{equation}
H_{N}\left(\hat{\mathbf{n}}\right)=u\,\,\,{\rm and}\,\,\,\nabla H_{N}\left(\hat{\mathbf{n}}\right)=0.\label{eq:conditional}
\end{equation}
When combined with Markov's inequality to upper bound corresponding
probabilities, the formula can roughly be thought of as a variant
of the union bound, where the intensity accounts for the number of
events. In the sections below we explain in more detail how the formula
is used to derive bounds on the contributions ${\rm Cont}_{N,\beta}\left(A\times B\right)$.
Prior to deriving those bounds (in Section \ref{sec:Bounds-on-contributions}),
we will need to investigate the conditional probabilities mentioned
above.

\subsection*{Conditional structure around critical points}

The Kac-Rice formula allows us to transfer questions dealing with
${\rm Cont}_{N,\beta}\left(A\times B\right)$ to ones about the conditional
law of $\left\{ H_{N}\left(\boldsymbol{\sigma}\right)\right\} _{\boldsymbol{\sigma}}$
given (\ref{eq:conditional}). A useful description of the latter
will be obtained in Section \ref{sec:Decomposition} by decomposing
$H_{N}\left(\boldsymbol{\sigma}\right)$ as a sum of independent fields,
see (\ref{eq:Hbar decomposition}). For any overlap $q$, on $\left\{ \boldsymbol{\sigma}:R\left(\boldsymbol{\sigma},\hat{\mathbf{n}}\right)=q\right\} $
the $k$-th of those fields is distributed like a pure spherical $k$-spin
model multiplied by a factor that is a function of $q$, which can
be thought of as an effective temperature (see (\ref{eq:1220-2})).
The effect of conditioning on $H_{N}\left(\hat{\mathbf{n}}\right)=u$
and $\nabla H_{N}\left(\hat{\mathbf{n}}\right)=0$ is equivalent to
dropping the field corresponding to $k=1$ and setting the field corresponding
to $k=0$ to be equal to some deterministic function of $\boldsymbol{\sigma}$.
On thin bands around overlap $q$, the conditional field is roughly
a mixture of pure $k$-spin with $2\leq k\leq p$, shifted by a constant.

\subsection*{Upper bounds on masses of bands}

Denote the conditional law given (\ref{eq:conditional}) and expectation
by $\mathbb{P}_{u,0}$ and $\mathbb{E}_{u,0}$ (see Remark \ref{rem:conditional laws}).
Abbreviate, only in this subsection, $Z(q)=Z_{N,\beta}({\rm Band}(\hat{\mathbf{n}},q,q+o(1)))$.
Combining a variant of the Kac-Rice formula (see Lemmas \ref{lem:17}
and \ref{lem:18}) with a computation of $\mathbb{E}_{u,0}Z(q)$,
we will derive an upper bound on the expectation of contributions
to the partition function (\ref{eq:cont}). By Markov's inequality,
they yield upper bounds on the probabilities that the contributions
are not small compared to the mass of the bands in Theorem \ref{thm:Geometry}.
This will be good enough for overlaps sufficiently close to $1$.
Specifically, in Section \ref{sec:range(q**,1)}, we will use such
bounds for overlaps in the range $q\in(q_{**},1)$ with $q_{**}$
defined in (\ref{eq:39}) (see Figure \ref{fig:qualitative-graph}).

In general, however, what describes the typical behavior (under $\mathbb{P}_{u,0}$)
is the corresponding free energy, i.e., $\frac{1}{N}\mathbb{E}_{u,0}\log\left(Z(q)\right)$.
By Jensen's inequality, this free energy is bounded from above by
$\frac{1}{N}\log(\mathbb{E}_{u,0}Z(q))$, which for $u=-EN$ is asymptotically
equal to the expression $\Frf\left(E,q\right)$ of (\ref{eq:6}) (see
Remark \ref{rem:18}). In fact, by the methods we use in Section \ref{sec:range(q**,1)},
one can check that this bound is tight asymptotically for $q\in\left(q_{c},1\right)$,
where $q_{c}\in\left(q_{**},q_{*}\right)$ is given in (\ref{eq:qc}). 

We remark that the bound obtained from Jensen's inequality is not
good enough to bound contributions related to overlaps sufficiently
close to $q_{{\rm LS}}$. We deal with the range $\left(q_{{\rm LS}},q_{**}\right)$
in Section \ref{sec:range(q***,q**)} where we upper bound the $N\to\infty$
limit of the free energy $\frac{1}{N}\mathbb{E}_{u,0}\log\left(Z(q)\right)$
which, with $u=-NE$, we denote by $\FrfF\left(E,q\right)$ (see (\ref{eq:lambda_F})).
First, we will relate $\FrfF\left(E,q\right)$ to a free energy $\FrfFt\left(E,q\right)$
(see (\ref{eq:LambdaF2})) which in a certain sense takes into account
in a non-trivial way only the $k=2$ component in the decomposition
of $H_{N}\left(\boldsymbol{\sigma}\right)$, i.e., a pure $2$-spin.
Second, we shall bound the fluctuations (under $\mathbb{P}_{u,0}$)
of $\frac{1}{N}\log\left(Z(q)\right)$ from its mean. Those will allow
us to use the Kac-Rice formula to control the number of critical points
whose corresponding bands have exceedingly high mass. Together with
bounds on the number of critical points obtained from Theorem \ref{thm:ABAC},
this yields an upper bound on contribution related to overlaps in
the range $\left(q_{{\rm LS}},q_{**}\right)$ w.h.p. 

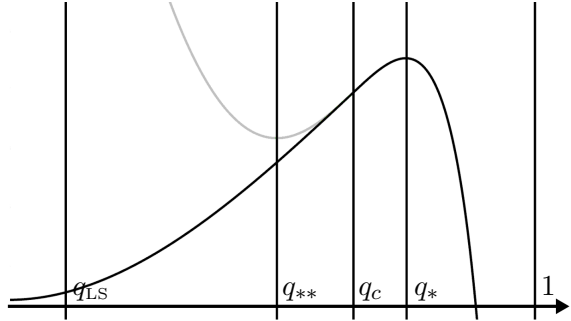
\begin{figure}[h]
\begin{overpic}[width=0.5\textwidth,tics=10]{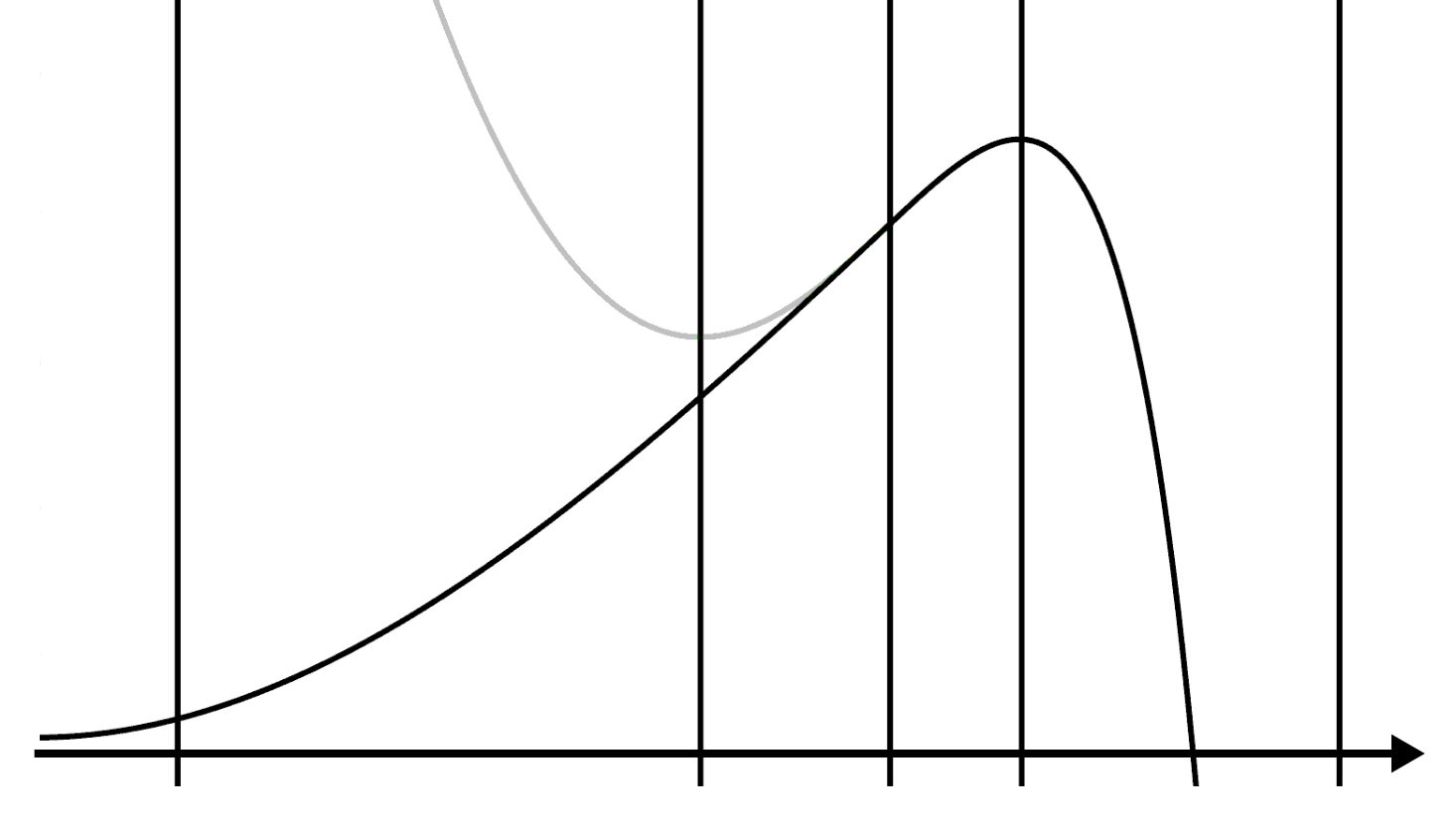}
 \put (93.2,7) {$1$} 
\put (71.5,7) {$q_*$}
\put (62,7) {$q_c$}
\put (49,7) {$q_{**}$}
\put (13.25,7) {$q_{\scriptscriptstyle\rm{LS}}$}
\end{overpic}\protect\caption{\label{fig:qualitative-graph}A qualitative graph of $\protect\FrfFt\left(E_{0},q\right)$
(black) and $\protect\Frf\left(E_{0},q\right)$ (gray), at low temperature
(on $(q_{c},1)$ both coincide). Vertical lines correspond to overlap
values.}
\end{figure}

\subsection*{Lower bound on the mass of a band of overlap $q_{*}$}

Conditional on (\ref{eq:conditional}), the behavior of the Hamiltonian
on a thin band ${\rm Band}\left(\hat{\mathbf{n}},q_{*}-o(1),q_{*}+o(1)\right)$,
is essentially described by its behavior on $\left\{ \boldsymbol{\sigma}:R\left(\boldsymbol{\sigma},\hat{\mathbf{n}}\right)=q_{*}\right\} $.
As we already mentioned, the latter is distributed like a spherical
mixed model, where the coefficients are determined by $\beta$. As
$\beta$ increases - and the band narrows, i.e., $q_{*}$ increases
to $1$ - the $2$-spin interaction becomes more dominant relative
to the other spins. This will allow us to prove (see Proposition \ref{prop:Gaussflucts}
and Lemmas \ref{lem:BaikLee} and \ref{lem:ConcentrationZ}) that
the fluctuations of the free energy attributed to the band are roughly
determined only by the $2$-spin interaction. We remark that the $2$-spin
interaction is closely related to the Hessian of the Hamiltonian at
the center of the band, which is key to Corollary \ref{cor:FE_u_Hess_rep}.
Also, this interaction will have an `effective' high temperature which,
based on the convergence result of Baik and Lee \cite{BaikLee} (Theorem
\ref{thm:BaikLee}), implies Gaussian fluctuations for the free energy.
Combined with the Kac-Rice formula this will be used to show that
w.h.p there are no bands corresponding to (\ref{eq:Reg*}) with relative
mass which is too low. On the other hand, from Theorem \ref{thm:ext proc}
we know that the probability that there are no critical points with
values in the range corresponding to (\ref{eq:Reg*}) goes to $0$
as $N\to\infty$. In Proposition \ref{prop:bounds}, we will this
to conclude that there exist bands corresponding to (\ref{eq:Reg*})
with large enough relative mass and prove a lower bound on the contribution
of (\ref{eq:Reg*}).

\subsection*{Overlap distribution}

We shall compute the first and second moments of the mass of bands
as in Theorem \ref{thm:Geometry} (see Proposition \ref{prop:means}
and Lemma \ref{lem:Lambda_F_bound}) and see that: the contribution
to the second moment coming from pairs of points whose overlap `inside
the band'%
\footnote{I.e., the overlap between the projections of the points to the orthogonal
space to the center point $\boldsymbol{\sigma}_{0}\in\mathbb{R}^{N}$.%
} is not approximately $0$ is negligible; and that the ratio of second
to first moment squared converges to a constant as $N\to\infty$.
Combined with the lower bound on the mass of corresponding bands which
we discussed in the previous subsection, this will yield (\ref{eq:2601}).
To prove (\ref{eq:1402}) from the latter, we will use Corollary \ref{cor:orth},
which states that deep critical points are either antipodal or essentially
orthogonal, and a simple deterministic geometric argument (see Lemma
\ref{lem:center}).

\section{\label{sec:Decomposition}Decomposition around (critical) points}

A crucial ingredient in our analysis is a certain decomposition of
$H_{N}\left(\boldsymbol{\sigma}\right)$ around a given point on the
sphere, which we develop in this section. We shall use $\hat{\mathbf{n}}$
as the center point, but since the Hamiltonian is invariant under
rotations of the sphere, similar results hold for an arbitrary point
on the sphere. Let $\left(E_{i}\right)_{i=1}^{N-1}=\left(E_{i}\left(\boldsymbol{\sigma}\right)\right)_{i=1}^{N-1}$
be a smooth orthonormal frame field defined over a neighborhood of
$\hat{\mathbf{n}}$ (relative to the standard Riemannian metric).
With $P_{\hat{\mathbf{n}}}:\left(x_{1},...,x_{N}\right)\mapsto\left(x_{1},...,x_{N-1}\right)$
denoting the projection from the hemisphere $\left\{ x\in\mathbb{S}^{N-1}\left(\sqrt{N}\right):\, x_{N}>0\right\} $
to $\mathbb{R}^{N-1}$, we shall assume that for any smooth function
$g:\mathbb{S}^{N-1}\left(\sqrt{N}\right)\to\mathbb{R}$,%
\footnote{The fact that such frame field exists can be seen from the following.
If we let $\left\{ \frac{\partial}{\partial x_{i}}\right\} _{i=1}^{N-1}$
be the pullback of $\left\{ \frac{d}{dx_{i}}\right\} _{i=1}^{N-1}$
by $P_{\hat{\mathbf{n}}}$, then $\left\{ \frac{\partial}{\partial x_{i}}(\hat{\mathbf{n}})\right\} _{i=1}^{N-1}$
is an orthonormal frame at the north pole. For any point in a small
neighborhood of $\hat{\mathbf{n}}$ we can define an orthonormal frame
as the parallel transport of $\left\{ \frac{\partial}{\partial x_{i}}(\hat{\mathbf{n}})\right\} _{i=1}^{N-1}$
along a geodesic from $\hat{\mathbf{n}}$ to that point. This yields
an orthonormal frame field on that neighborhood, say $E_{i}(\boldsymbol{\sigma})=\sum_{j=1}^{N-1}a_{ij}(\boldsymbol{\sigma})\frac{\partial}{\partial x_{j}}(\boldsymbol{\sigma})$,
$i=1,...,N-1$. Working with the coordinate system $P_{\hat{\mathbf{n}}}$
one can verify that at $x=0$ the Christoffel symbols $\Gamma_{ij}^{k}$
are equal to $0$, and therefore (see e.g. \cite[Eq. (2), P. 53]{DoCarmo})
the derivatives $\frac{d}{dx_{k}}a_{ij}(P_{\hat{\mathbf{n}}}^{-1}(x))$
at $x=0$ are also equal to $0$.%
}
\begin{equation}
E_{i}g\left(\hat{\mathbf{n}}\right)=\frac{d}{dx_{i}}g\circ P_{\hat{\mathbf{n}}}^{-1}\left(0\right),\,\,\, E_{i}E_{j}g\left(\hat{\mathbf{n}}\right)=\frac{d}{dx_{i}}\frac{d}{dx_{j}}g\circ P_{\hat{\mathbf{n}}}^{-1}\left(0\right).\label{eq:derivativesProjection}
\end{equation}
Define $\mathcal{F}_{k}$ as the $\sigma$-algebra generated by 
\begin{equation}
\left\{ E_{i_{1}}\cdots E_{i_{j}}H_{N}\left(\hat{\mathbf{n}}\right):1\leq i_{1}\leq\cdots\leq i_{j}\leq N-1,\,0\leq j\leq k\right\} ,\label{eq:vars}
\end{equation}
that is, by $H_{N}\left(\hat{\mathbf{n}}\right)$ and all the derivatives
(by $E_{i}$) of $H_{N}\left(\boldsymbol{\sigma}\right)$ at $\hat{\mathbf{n}}$
up to order $k$. Setting $\mathring{H}_{N}^{\hat{\mathbf{n}},k}\left(\boldsymbol{\sigma}\right)=\mathbb{E}\left[\left.H_{N}\left(\boldsymbol{\sigma}\right)\,\right|\,\mathcal{F}_{k}\right]$,
define $\bar{H}_{N}^{\hat{\mathbf{n}},k}\left(\boldsymbol{\sigma}\right)=\mathring{H}_{N}^{\hat{\mathbf{n}},k}\left(\boldsymbol{\sigma}\right)-\mathring{H}_{N}^{\hat{\mathbf{n}},k-1}\left(\boldsymbol{\sigma}\right)$
for $k>0$, and $\bar{H}_{N}^{\hat{\mathbf{n}},0}\left(\boldsymbol{\sigma}\right)=\mathring{H}_{N}^{\hat{\mathbf{n}},0}\left(\boldsymbol{\sigma}\right)$. 

Define $H_{N}^{{\rm Euc}}\left(x\right)$ as the extension of $H_{N}\left(\boldsymbol{\sigma}\right)$
to $\mathbb{R}^{N}$ defined by the formula (\ref{eq:Hamiltonian})
only with general $x\in\mathbb{R}^{N}$ instead of $\boldsymbol{\sigma}$
from the sphere. Clearly, $H_{N}\left(\hat{\mathbf{n}}\right)$ and
the derivatives of $H_{N}\left(\boldsymbol{\sigma}\right)$ at $\hat{\mathbf{n}}$
up to order $k$ are determined by $H_{N}^{{\rm Euc}}\left(\hat{\mathbf{n}}\right)$
and the Euclidean derivatives of $H_{N}^{{\rm Euc}}\left(x\right)$
at $\hat{\mathbf{n}}$ up to order $k$, and vice versa. In other
words, $\mathcal{F}_{k}$ is equal to the $\sigma$-algebra generated
by
\begin{equation}
\left\{ \frac{d}{dx_{i_{1}}}\cdots\frac{d}{dx_{i_{j}}}H_{N}^{{\rm Euc}}\left(\hat{\mathbf{n}}\right):1\leq i_{1}\leq\cdots\leq i_{j}\leq N-1,\,0\leq j\leq k\right\} .\label{eq:varEuc}
\end{equation}
The advantage of this is that the Euclidean derivatives (\ref{eq:varEuc})
are directly related to the disorder $J_{i_{1},...,i_{p}}$. 

Write 
\begin{equation}
H_{N}\left(\boldsymbol{\sigma}\right)=\frac{1}{N^{\left(p-1\right)/2}}\sum_{1\leq i_{1}\leq...\leq i_{p}\leq N}J_{i_{1},...,i_{p}}^{\prime}\sigma_{i_{1}}\cdots\sigma_{i_{p}},\quad\boldsymbol{\sigma}\in\mathbb{S}^{N-1}\left(\sqrt{N}\right),\label{eq:HN_Jprime}
\end{equation}
where $J_{\bar{i}_{1},...,\bar{i}_{p}}^{\prime}$ is defined as the
sum of all $J_{i_{1},...,i_{p}}$ in (\ref{eq:Hamiltonian}) such
that $\left\{ i_{1},...,i_{p}\right\} =\left\{ \bar{i}_{1},...,\bar{i}_{p}\right\} $
as multisets. Then for any $1\leq k\leq p$, $1\leq i_{1}\leq\cdots\leq i_{k}\leq N-1$,
\begin{equation}
\frac{d}{dx_{i_{1}}}\cdots\frac{d}{dx_{i_{k}}}H_{N}^{{\rm Euc}}\left(\hat{\mathbf{n}}\right)=\frac{\prod_{j=1}^{N-1}\left|\left\{ l\leq k:i_{l}=j\right\} \right|!}{N^{\left(k-1\right)/2}}J_{i_{1},...,i_{k},N,...,N}^{\prime},\label{eq:EucderivativeJprime}
\end{equation}
and 
\begin{equation}
H_{N}\left(\hat{\mathbf{n}}\right)=H_{N}^{{\rm Euc}}\left(\hat{\mathbf{n}}\right)=N^{1/2}J_{N,...,N}^{\prime}.\label{eq:euc1}
\end{equation}
Since $H_{N}\left(\boldsymbol{\sigma}\right)$ and $\left\{ J_{i_{1},...,i_{p}}^{\prime}\right\} $
are centered and jointly Gaussian, and $\left\{ J_{i_{1},...,i_{p}}^{\prime}\right\} $
is a set of independent variables, it follows that 
\begin{align*}
 & \mathbb{E}\left[\left.H_{N}\left(\boldsymbol{\sigma}\right)\,\right|\, J_{i_{1},...,i_{j},N,...,N}^{\prime}:1\leq i_{1}\leq\cdots\leq i_{j}\leq N-1,\,0\leq j\leq k\right]\\
 & =\sum_{j=0}^{k}\mathbb{E}\left[\left.H_{N}\left(\boldsymbol{\sigma}\right)\,\right|\, J_{i_{1},...,i_{j},N,...,N}^{\prime}:1\leq i_{1}\leq\cdots\leq i_{j}\leq N-1\right].
\end{align*}

Therefore, for $0\leq k\leq p$,
\begin{align}
\bar{H}_{N}^{\hat{\mathbf{n}},k}\left(\boldsymbol{\sigma}\right) & =\mathbb{E}\left[\left.H_{N}\left(\boldsymbol{\sigma}\right)\,\right|\, J_{i_{1},...,i_{k},N,...,N}^{\prime}:1\leq i_{1}\leq\cdots\leq i_{k}\leq N-1\right]\nonumber \\
 & =\frac{1}{N^{\left(p-1\right)/2}}\sum_{1\leq i_{1}\leq\cdots\leq i_{k}\leq N-1}J_{i_{1},...,i_{k},N,...,N}^{\prime}\sigma_{i_{1}}\cdots\sigma_{i_{k}}\sigma_{N}^{p-k}.\label{eq:1220-1}
\end{align}
The following lemma directly follows. For $0\leq k\leq p$, let $H_{N-1}^{{\rm pure\,}k}\left(\boldsymbol{\sigma}\right)$
be independent pure $k$-spin models on the sphere $\mathbb{S}^{N-2}$,
where by $0$-spin model we simply mean a constant centered Gaussian
field with variance $\sqrt{N-1}$. Also, with $\boldsymbol{\sigma}=\left(\sigma_{1},...,\sigma_{N}\right)$
define 
\begin{equation}
\tilde{\boldsymbol{\sigma}}=\sqrt{\frac{N-1}{N}}\frac{\left(\sigma_{1},...,\sigma_{N-1}\right)}{\sqrt{1-q^{2}\left(\boldsymbol{\sigma}\right)}}\in\mathbb{S}^{N-2}\,\,\,{\rm and}\,\,\, q\left(\boldsymbol{\sigma}\right)=\frac{\sigma_{N}}{\sqrt{N}}=R\left(\boldsymbol{\sigma},\hat{\mathbf{n}}\right).\label{eq:69}
\end{equation}
Lastly, define
\begin{equation}
\alpha_{k}\left(q\right)\triangleq\sqrt{\binom{p}{k}\left(1-q^{2}\right)^{k}}q^{p-k},\label{eq:alpha_k}
\end{equation}
where obviously $\sum_{k=0}^{p}\alpha_{k}^{2}\left(q\right)=\left(1-q^{2}+q^{2}\right)^{p}=1$.
Note that $\alpha_{k}\left(q\right)$ can be negative. Denote by $\overset{d}{=}$
equality in distribution.
\begin{lem}
\label{lem:HNdecomposition}The Hamiltonian decomposes as
\begin{equation}
H_{N}\left(\boldsymbol{\sigma}\right)=\sum_{k=0}^{p}\bar{H}_{N}^{\hat{\mathbf{n}},k}\left(\boldsymbol{\sigma}\right),\label{eq:Hbar decomposition}
\end{equation}
where the Gaussian fields $\left\{ \bar{H}_{N}^{\hat{\mathbf{n}},k}\left(\boldsymbol{\sigma}\right)\right\} _{\boldsymbol{\sigma}}$,
$0\leq k\leq p$, are independent of each other and
\begin{equation}
\left\{ \bar{H}_{N}^{\hat{\mathbf{n}},k}\left(\boldsymbol{\sigma}\right)\right\} _{\boldsymbol{\sigma}}\overset{d}{=}\left\{ \alpha_{k}\left(q\left(\boldsymbol{\sigma}\right)\right)\sqrt{\frac{N}{N-1}}H_{N-1}^{{\rm pure\,}k}\left(\tilde{\boldsymbol{\sigma}}\right)\right\} _{\boldsymbol{\sigma}}.\label{eq:1220-2}
\end{equation}
\end{lem}
\begin{proof}
Equation (\ref{eq:Hbar decomposition}) follows from (\ref{eq:1220-1})
and (\ref{eq:HN_Jprime}). Independence follows from that of $J_{i_{1},...,i_{p}}^{\prime}$
of each other. By simple algebra
\[
\bar{H}_{N}^{\hat{\mathbf{n}},k}\left(\boldsymbol{\sigma}\right)=\sqrt{\frac{N}{N-1}\binom{p}{k}^{-1}}\frac{\alpha_{k}\left(q\left(\boldsymbol{\sigma}\right)\right)}{\left(N-1\right)^{\left(k-1\right)/2}}\sum_{1\leq i_{1}\leq\cdots\leq i_{k}\leq N-1}J_{i_{1},...,i_{k},N,...,N}^{\prime}\tilde{\boldsymbol{\sigma}}_{i_{1}}\cdots\tilde{\boldsymbol{\sigma}}_{i_{k}},
\]
where $\tilde{\boldsymbol{\sigma}}_{i}$ are the elements of $\tilde{\boldsymbol{\sigma}}$.
Using the fact that 
\[
{\rm Var}\left(J_{i_{1},...,i_{k},N,...,N}^{\prime}\right)=\binom{p}{k}{\rm Var}\left(J_{i_{1},...,i_{k}}^{\prime}\right),
\]
(\ref{eq:1220-2}) follows from (\ref{eq:HN_Jprime}) (with $N-1$
instead of $N$ in the latter).
\end{proof}
The following lemma expresses $\bar{H}_{N}^{\hat{\mathbf{n}},k}\left(\boldsymbol{\sigma}\right)$,
$k\leq2$, in terms of the Hamiltonian $H_{N}\left(\boldsymbol{\sigma}\right)$
directly. Denote 
\[
\nabla H_{N}\left(\hat{\mathbf{n}}\right)=\left(E_{i}H_{N}\left(\hat{\mathbf{n}}\right)\right)_{i=1}^{N-1}\mbox{\,\,\, and\,\,\,\,\,\,}\nabla^{2}H_{N}\left(\hat{\mathbf{n}}\right)=\left(E_{i}E_{j}H_{N}\left(\hat{\mathbf{n}}\right)\right)_{i,j=1}^{N-1}.
\]
We define the random matrix
\begin{equation}
\mathbf{G}\left(\hat{\mathbf{n}}\right):=\mathbf{G}_{N-1}\left(\hat{\mathbf{n}}\right)\triangleq\nabla^{2}H_{N}\left(\hat{\mathbf{n}}\right)+\frac{p}{N}H_{N}\left(\hat{\mathbf{n}}\right)\mathbf{I},\label{eq:G}
\end{equation}
where $\mathbf{I}:=\mathbf{I}_{N-1}$ denotes the identity matrix
of dimension $N-1$. A random matrix $\mathbf{M}$ from the $N\times N$
Gaussian orthogonal ensemble, or for brevity an $N\times N$ GOE matrix,
is a real, symmetric matrix such that all elements are centered Gaussian
variables which, up to symmetry, are independent with variance $2/N$
on the diagonal and $1/N$ off the diagonal. 
\begin{lem}
\label{lem:Hhat_expressions}We have that
\begin{align}
\bar{H}_{N}^{\hat{\mathbf{n}},0}\left(\boldsymbol{\sigma}\right) & =q^{p}\left(\boldsymbol{\sigma}\right)\cdot H_{N}\left(\hat{\mathbf{n}}\right),\label{eq:Hnhat0}\\
\bar{H}_{N}^{\hat{\mathbf{n}},1}\left(\boldsymbol{\sigma}\right) & =q^{p-1}\left(\boldsymbol{\sigma}\right)\left(1-q^{2}\left(\boldsymbol{\sigma}\right)\right)^{1/2}\sqrt{\frac{N}{N-1}}\cdot\left\langle \nabla H_{N}\left(\hat{\mathbf{n}}\right),\tilde{\boldsymbol{\sigma}}\right\rangle ,\label{eq:Hnhat1}\\
\bar{H}_{N}^{\hat{\mathbf{n}},2}\left(\boldsymbol{\sigma}\right) & =\frac{1}{2}q^{p-2}\left(\boldsymbol{\sigma}\right)\left(1-q^{2}\left(\boldsymbol{\sigma}\right)\right)\frac{N}{N-1}\cdot\tilde{\boldsymbol{\sigma}}^{T}\mathbf{G}_{N-1}\left(\hat{\mathbf{n}}\right)\tilde{\boldsymbol{\sigma}},\label{eq:Hnhat2}
\end{align}
and $H_{N}\left(\hat{\mathbf{n}}\right)$, $\nabla H_{N}\left(\hat{\mathbf{n}}\right)$,
and $\mathbf{G}_{N-1}\left(\hat{\mathbf{n}}\right)$ are independent.
Moreover, $\nabla H_{N}\left(\hat{\mathbf{n}}\right)\sim\mathcal{N}\left(0,p\mathbf{I}_{N-1}\right)$
and with $\mathbf{M}$ being a GOE matrix of dimension $N-1$, 
\begin{equation}
\mathbf{G}_{N-1}\left(\hat{\mathbf{n}}\right)\stackrel{d}{=}\sqrt{\frac{N-1}{N}p\left(p-1\right)}\mathbf{M}.\label{eq:GMrel}
\end{equation}
\end{lem}
\begin{proof}
Since $H_{N}\left(\hat{\mathbf{n}}\right)=N^{1/2}J_{N,...,N}^{\prime}$
and $q\left(\boldsymbol{\sigma}\right)=\sigma_{N}/\sqrt{N}$, (\ref{eq:Hnhat0})
follows from (\ref{eq:1220-1}) (with $k=0$). From (\ref{eq:derivativesProjection}),
(\ref{eq:HN_Jprime}) and the fact that
\[
\left.\frac{d}{dx_{i}}\right|_{x=0}\left(x_{i_{1}}\cdots x_{i_{k}}\left(N-\sum_{i=1}^{N-1}x_{i}^{2}\right)^{\frac{p-k}{2}}\right)=\begin{cases}
N^{\frac{p-1}{2}},\,\, & k=1,\, i_{1}=i\\
0, & {\rm otherwise}
\end{cases},
\]
we obtain that
\begin{equation}
E_{i}H_{N}\left(\hat{\mathbf{n}}\right)=J_{i,N,...,N}^{\prime}\sim\mathcal{N}\left(0,p\right).\label{eq:z1}
\end{equation}
Thus, (\ref{eq:Hnhat1}) follows from (\ref{eq:1220-1}) (with $k=1$)
and some rearranging. The independence of $(J_{i,N,...,N}^{\prime})_{i}$
implies that $\nabla H_{N}\left(\hat{\mathbf{n}}\right)\sim\mathcal{N}\left(0,p\mathbf{I}_{N-1}\right)$.

Similarly, from (\ref{eq:derivativesProjection}), (\ref{eq:HN_Jprime})
and the fact that
\[
\left.\frac{d}{dx_{i}}\frac{d}{dx_{j}}\right|_{x=0}\left(x_{i_{1}}\cdots x_{i_{k}}\left(N-\sum_{i=1}^{N-1}x_{i}^{2}\right)^{\frac{p-k}{2}}\right)=\begin{cases}
\left(1+\delta_{ij}\right)N^{\frac{p-2}{2}},\,\, & k=2,\,\left\{ i_{1},i_{2}\right\} =\left\{ i,j\right\} \\
-\delta_{ij}pN^{\frac{p-2}{2}}, & k=0\\
0, & {\rm otherwise}
\end{cases},
\]
we obtain that 
\begin{align*}
E_{i}E_{j}H_{N}\left(\hat{\mathbf{n}}\right)= & N^{-1/2}\left(\left(1+\delta_{ij}\right)J_{i,j,N,...,N}^{\prime}-\delta_{ij}pJ_{N,...,N}^{\prime}\right),
\end{align*}
and thus 
\begin{equation}
\left(\mathbf{G}_{N-1}\left(\hat{\mathbf{n}}\right)\right)_{ij}=N^{-1/2}\left(1+\delta_{ij}\right)J_{i,j,N,...,N}^{\prime}.\label{eq:z2}
\end{equation}
Since $J_{i,j,N,...,N}^{\prime}\sim\mathcal{N}(0,p\left(p-1\right)/(1+\delta_{ij}))$
are independent, (\ref{eq:GMrel}) follows. Substituting this in (\ref{eq:1220-1})
(with $k=2$), after some algebra we obtain (\ref{eq:Hnhat2}). 

Lastly, the independence $H_{N}\left(\hat{\mathbf{n}}\right)$, $\nabla H_{N}\left(\hat{\mathbf{n}}\right)$,
and $\mathbf{G}_{N-1}\left(\hat{\mathbf{n}}\right)$ follows since
they are measurable w.r.t to disjoint subsets of the random variables
$\left\{ J_{i_{1},...,i_{p}}^{\prime}\right\} $.
\end{proof}
Denote
\begin{equation}
\bar{H}_{N}^{\hat{\mathbf{n}},k+}\left(\boldsymbol{\sigma}\right)=\sum_{m=k}^{p}\bar{H}_{N}^{\hat{\mathbf{n}},m}\left(\boldsymbol{\sigma}\right).\label{eq:3+}
\end{equation}

\begin{rem}
\label{rem:conditional laws}In many situations in the sequel we will
have some random variable $X$ for which $H_{N}\left(\hat{\mathbf{n}}\right)$,
$\nabla H_{N}\left(\hat{\mathbf{n}}\right)$ and $X$ have a continuous
joint density and we will need to compute or estimate the probability
that $X\in B$, for some measurable set $B$, conditional on $H_{N}\left(\hat{\mathbf{n}}\right)=u$,
$\nabla H_{N}\left(\hat{\mathbf{n}}\right)=0$. In this case, the
notation $\mathbb{P}_{u,0}\left\{ X\in B\right\} $ should be understood
as the probability under the law determined the usual conditional
density given by the ratio of the joint densities. That is, if $\varphi_{1}\left(v,w,x\right)$
and $\varphi_{2}\left(v,w\right)$ are the continuous densities of
$(H_{N}(\hat{\mathbf{n}}),\nabla H_{N}(\hat{\mathbf{n}}),X)$ and
$(H_{N}(\hat{\mathbf{n}}),\nabla H_{N}(\hat{\mathbf{n}}))$ , respectively,
then $\mathbb{P}_{u,0}\left\{ X\in B\right\} =\int_{B}\varphi_{1}\left(u,0,x\right)dx/\varphi_{2}\left(u,0\right)$.
We will also refer to $\mathbb{P}_{u,0}\left\{ X\in\cdot\right\} $
as the conditional law of $X$ under $\mathbb{P}_{u,0}\left\{ \,\cdot\,\right\} $.
We define $\mathbb{P}_{u,0,\mathbf{A}}\left\{ X\in B\right\} $ similarly,
assuming the corresponding joint density exists, as the conditional
probability given 
\begin{equation}
H_{N}\left(\hat{\mathbf{n}}\right)=u,\nabla H_{N}\left(\hat{\mathbf{n}}\right)=0{\rm \,\, and\,\,}\mathbf{G}_{N-1}\left(\hat{\mathbf{n}}\right)=\mathbf{A}.\label{eq:24}
\end{equation}
The conditional expectations corresponding to the above will denoted
by $\mathbb{E}_{u,0}\left\{ \,\cdot\,\right\} $ and $\mathbb{E}_{u,0,\mathbf{A}}\left\{ \,\cdot\,\right\} $.

The corollary below follows directly from Lemmas \ref{lem:HNdecomposition}
and \ref{lem:Hhat_expressions}.\end{rem}
\begin{cor}
\label{cor:conditional laws}The conditional law of the field $H_{N}\left(\boldsymbol{\sigma}\right)$
under $\mathbb{P}_{u,0}\left\{ \,\cdot\,\right\} $ is identical to
the (unconditional) law of
\[
uq^{p}\left(\boldsymbol{\sigma}\right)+\bar{H}_{N}^{\hat{\mathbf{n}},2+}\left(\boldsymbol{\sigma}\right).
\]
Similarly, the conditional law of the field $H_{N}\left(\boldsymbol{\sigma}\right)$
under $\mathbb{P}_{u,0,\mathbf{A}}\left\{ \,\cdot\,\right\} $ is
identical to the (unconditional) law of
\[
uq^{p}\left(\boldsymbol{\sigma}\right)+\frac{1}{2}q^{p-2}\left(\boldsymbol{\sigma}\right)\left(1-q^{2}\left(\boldsymbol{\sigma}\right)\right)\frac{N}{N-1}\cdot\tilde{\boldsymbol{\sigma}}^{T}\mathbf{A}\tilde{\boldsymbol{\sigma}}+\bar{H}_{N}^{\hat{\mathbf{n}},3+}\left(\boldsymbol{\sigma}\right).
\]

\end{cor}

\section{\label{sec:overlapsVals}Important overlap values}

Our analysis requires understanding the contribution to the partition
function $Z_{N,\beta}$ coming from different distances, or equivalently
overlaps, from critical points. In this section we define several
important overlap values that will be used to define different regions
in the overlap-depth plane of Section \ref{sec:outline} (see Figure
\ref{fig:3.1}). It will be very useful for us to investigate the
restriction of the random fields $\bar{H}_{N}^{\hat{\mathbf{n}},k}\left(\boldsymbol{\sigma}\right)$
(with $k=2$ in particular) to 
\[
\left\{ \boldsymbol{\sigma}\in\mathbb{S}^{N-1}:R\left(\boldsymbol{\sigma},\hat{\mathbf{n}}\right)=q\right\} ,
\]
which for convenience we parametrize as random fields on $\mathbb{S}^{N-2}$.
With $\theta_{q}:\mathbb{S}^{N-2}\to\mathbb{S}^{N-1}$ being the left
inverse of $\boldsymbol{\sigma}\mapsto\tilde{\boldsymbol{\sigma}}$
(see (\ref{eq:69})) given by 
\[
\theta_{q}\left(\left(\sigma_{1},...,\sigma_{N-1}\right)\right)=\sqrt{\frac{N}{N-1}\left(1-q^{2}\right)}\left(\sigma_{1},...,\sigma_{N-1},0\right)+q\hat{\mathbf{n}},
\]
for any function $h:\mathbb{S}^{N-1}\to\mathbb{R}$ define $h|_{q}:\mathbb{S}^{N-2}\to\mathbb{R}$
by 
\begin{equation}
h|_{q}\left(\boldsymbol{\sigma}\right)=h\circ\theta_{q}\left(\boldsymbol{\sigma}\right).\label{eq:22}
\end{equation}

Note that by (\ref{eq:1220-2}), $\bar{H}_{N}^{\hat{\mathbf{n}},k}|_{q}$
is a pure $k$-spin for any $q$, up to a multiplicative factor. Specifically,
we have that 
\begin{equation}
\mathbb{E}\left\{ \bar{H}_{N}^{\hat{\mathbf{n}},2}|_{q}\left(\boldsymbol{\sigma}\right)\bar{H}_{N}^{\hat{\mathbf{n}},2}|_{q}\left(\boldsymbol{\sigma}'\right)\right\} =N\left(\alpha_{2}\left(q\right)R\left(\boldsymbol{\sigma},\boldsymbol{\sigma}'\right)\right)^{2},\label{eq:Hnr2}
\end{equation}
where $\alpha_{2}\left(q\right)$ was defined in (\ref{eq:alpha_k}).
With $\beta$ fixed, we can think of the partition function corresponding
to $\bar{H}_{N}^{\hat{\mathbf{n}},2}|_{q}$ as that of the pure $2$-spin
model on $\mathbb{S}^{N-2}$ with an `effective' temperature
\begin{equation}
\beta|\alpha_{2}\left(q\right)|\sqrt{\frac{N}{N-1}}.\label{eq:beff}
\end{equation}
By Corollary \ref{cor:2spinFE}, with $\beta$ fixed, the limiting
free energy of $\bar{H}_{N}^{\hat{\mathbf{n}},2}|_{q}$ undergoes
a transition at values of $q$ that satisfy 
\begin{equation}
\alpha_{2}\left(q\right)=\binom{p}{2}^{1/2}q^{p-2}\left(1-q^{2}\right)=\frac{1}{\beta\sqrt{2}}\triangleq\chi_{2}.\label{eq:s2}
\end{equation}
Define, assuming $\beta$ is large enough so that the set below is
non-empty, 
\begin{equation}
q_{c}:=q_{c}(\beta)=\max\left\{ q\in\left(0,1\right):\alpha_{2}\left(q\right)=\chi_{2}\right\} .\label{eq:qc}
\end{equation}

In our computations we will encounter the function
\begin{equation}
\Frf\left(E,q\right)\triangleq\frac{1}{2}\log\left(1-q^{2}\right)+\beta Eq^{p}+\frac{1}{2}\beta^{2}\left(1-q^{2p}-pq^{2p-2}\left(1-q^{2}\right)\right),\label{eq:6}
\end{equation}
which is related to the free energy of bands of overlap approximately
$q$ around critical points of depth $-EN$ (see Lemma \ref{lem:Z_1st_mom}).
The critical points of $\Frf\left(E_{0},q\right)$ as a function of
$q$ are the solutions of 
\begin{equation}
-\frac{q}{1-q^{2}}+p\beta E_{0}q^{p-1}-p\left(p-1\right)\beta^{2}q^{2p-3}\left(1-q^{2}\right)=0.\label{eq:14}
\end{equation}
Viewing the latter as a quadratic equation in $\beta$ we have that
the solutions with $q\neq0$ are characterized by the relation 
\begin{equation}
\alpha_{2}\left(q\right)=\frac{1}{\sqrt{2}\beta}\left(\frac{E_{0}}{E_{\infty}}\pm\sqrt{\frac{E_{0}^{2}}{E_{\infty}^{2}}-1}\right).\label{eq:a39}
\end{equation}
From the fact that $z-\sqrt{z^{2}-1}$ decreases in $z>1$ and since
$E_{0}>E_{\infty}$,
\begin{equation}
\chi_{1}\triangleq\frac{1}{\sqrt{2}\beta}\left(\frac{E_{0}}{E_{\infty}}-\sqrt{\frac{E_{0}^{2}}{E_{\infty}^{2}}-1}\right)<\chi_{2}<\frac{1}{\sqrt{2}\beta}\left(\frac{E_{0}}{E_{\infty}}+\sqrt{\frac{E_{0}^{2}}{E_{\infty}^{2}}-1}\right)\triangleq\chi_{3}.\label{eq:7}
\end{equation}
Thus, assuming $\beta$ is sufficiently large so that the sets below
are non-empty, defining 
\begin{align}
q_{*} & :=q_{*}(\beta)=\max\left\{ q\in\left(0,1\right):\alpha_{2}\left(q\right)=\chi_{1}\right\} ,\label{eq:39}\\
q_{**} & :=q_{**}(\beta)=\max\left\{ q\in\left(0,1\right):\alpha_{2}\left(q\right)=\chi_{3}\right\} ,\nonumber 
\end{align}
which are in particular critical points of $\Frf\left(E,q\right)$
in $q$, we have that 
\[
0<q_{**}<q_{c}<q_{*}<1.
\]
(See Figure \ref{fig:qualitative-graph}.) We also have
\begin{align}
\lim_{\beta\to\infty}\frac{1}{2}\left(1-q_{*}^{2}\right)^{2}\frac{\partial^{2}}{\partial q^{2}}\Frf\left(E_{0},q_{*}\right) & =2\left(\beta\chi_{1}\right)^{2}-1\label{eq:14031}\\
 & <2\left(\beta\chi_{2}\right)^{2}-1=0.\nonumber 
\end{align}
Thus, for large enough $\beta$, $\frac{\partial^{2}}{\partial q^{2}}\Frf\left(E_{0},q_{*}\right)<0$.

With an arbitrary constant which will be fixed $C_{{\rm LS}}>\left(2p\left(E_{0}-E_{\infty}\right)\right)^{-1}$,
the last overlap value we define is 
\begin{equation}
q_{{\rm LS}}=1-C_{{\rm LS}}\frac{\log\beta}{\beta}.\label{eq:qLS}
\end{equation}
This overlap value is the one we will use to define the caps which
we restrict to, as described in the outline in Section \ref{sec:outline}.

\section{\label{sec:range(q**,1)}mass of bands under $\mathbb{P}_{u,0}$:
the range $(q_{**},1)$}

In this section we evaluate the relative partition function of bands
and caps. To shorten the notation, we will henceforth use the abbreviations
\begin{align}
{\rm Cap} & \triangleq{\rm Cap}\left(\hat{\mathbf{n}},q_{**}\right),\label{eq:BN}\\
{\rm Band}\left(\epsilon\right) & \triangleq{\rm Band}\left(\hat{\mathbf{n}},q_{*}-\epsilon,q_{*}+\epsilon\right).\nonumber 
\end{align}
The main results of this section are the two propositions below. In
Proposition \ref{prop:means} we compute, under $\mathbb{P}_{u,0}$,
the expected relative partition function of ${\rm Band}\left(cN^{-1/2}\right)$,
and show that for large $c$ it is much larger than that of ${\rm Cap}\setminus{\rm Band}\left(cN^{-1/2}\right)$.
The levels $u$ considered in the proposition are approximately equal
to $m_{N}$ (in fact for higher levels the overlap that captures most
of the mass is not $q_{*}$). However, because of the simple dependence
of the conditional law in $u$ (see Corollary \ref{cor:conditional laws})
we will be able to easily derive from Proposition \ref{prop:means}
bounds for higher levels when needed (e.g., in the proof of Lemma
\ref{eq:RegIIIbd}). As mentioned in Section \ref{sec:outline}, bounds
on the expected relative partition function will be sufficient for
our analysis of overlaps close enough to $1$. Specifically, the caps
(\ref{eq:BN}) cover the range $(q_{**},1)$. Define 
\begin{equation}
\mathfrak{V}_{N,\beta}\left(u\right)\triangleq\left(1-q_{*}^{2}\right)^{-3/2}\left|\frac{\partial^{2}}{\partial q^{2}}\Frf\left(E_{0},q_{*}\right)\right|^{-1/2}\exp\left\{ N\Frf\left(E_{0},q_{*}\right)-\left(\beta E_{0}N+\beta u\right)q_{*}^{p}\right\} .\label{eq:V(u)}
\end{equation}

\begin{prop}
\label{prop:means}For large enough $\beta$ we have the following.
Let $a_{N}=o\left(\sqrt{N}\right)$ be a sequence of positive numbers
and set $J_{N}=\left(m_{N}-a_{N},m_{N}+a_{N}\right)$. Then, 
\begin{equation}
\lim_{c\to\infty}\limsup_{N\to\infty}\sup_{u\in J_{N}}\left|\frac{\mathbb{E}_{u,0}\left\{ Z_{N,\beta}\left({\rm Band}\left(cN^{-1/2}\right)\right)\right\} }{\mathfrak{V}_{N,\beta}\left(u\right)}-1\right|=0,\label{eq:13-3}
\end{equation}
and
\begin{equation}
\lim_{c\to\infty}\limsup_{N\to\infty}\sup_{u\in J_{N}}\frac{\mathbb{E}_{u,0}\left\{ Z_{N,\beta}\left({\rm Cap}\setminus{\rm Band}\left(cN^{-1/2}\right)\right)\right\} }{\mathbb{E}_{u,0}\left\{ Z_{N,\beta}\left({\rm Band}\left(cN^{-1/2}\right)\right)\right\} }=0.\label{eq:13-2}
\end{equation}

\end{prop}
With $\chi_{1}$ given by (\ref{eq:7}), define 
\begin{equation}
C_{*}=1-2\left(\beta\chi_{1}\right)^{2}=1-\left(\frac{E_{0}}{E_{\infty}}-\sqrt{\frac{E_{0}^{2}}{E_{\infty}^{2}}-1}\right)^{2}\label{eq:Cstar}
\end{equation}
and note that by (\ref{eq:7}) and (\ref{eq:s2}), $C_{*}>0$. Define
\begin{equation}
Y_{*}\sim\mathcal{N}\left(\frac{1}{4}\log\left(C_{*}\right),-\frac{1}{2}\log\left(C_{*}\right)\right).\label{eq:48}
\end{equation}
The following proposition is key to controlling the mass of the bands
in (\ref{eq:thm1_1}).
\begin{prop}
\label{prop:Gaussflucts}For $\beta>0$ large enough we have the following.
Let $a_{N}=o\left(N\right)$ and $\epsilon_{N}=o(1)$ be sequences
of positive numbers and set $J_{N}=\left(m_{N}-a_{N},m_{N}+a_{N}\right)$.
Let $Y_{*}$ be defined as in (\ref{eq:48}). Then,
\begin{equation}
\lim_{N\to\infty}\sup_{u\in J_{N}}\left|\mathbb{P}_{u,0}\left\{ \frac{Z_{N,\beta}\left({\rm Band}\left(\epsilon_{N}\right)\right)}{\mathbb{E}_{u,0}\left\{ Z_{N,\beta}\left({\rm Band}\left(\epsilon_{N}\right)\right)\right\} }\leq t\right\} -\mathbb{P}\left\{ e^{Y_{*}}\leq t\right\} \right|=0.\label{eq:51}
\end{equation}

\end{prop}
In Section \ref{sub:moments} we prove Proposition \ref{prop:means}.
We proceed with a corresponding second moment computation in Section
\ref{sub:2nd}. In Section \ref{sub:conv_2to3+} we prove a version
of Proposition \ref{prop:Gaussflucts} where the `3-and-above' spins
in the decomposition (\ref{eq:Hbar decomposition}) are averaged out.
Lastly, we prove Proposition \ref{prop:Gaussflucts} in Section \ref{sub:pfGaussFlucts}
by essentially showing that this averaging has no effect in the scale
we work in.

\subsection{\label{sub:moments}Proof of Proposition \ref{prop:means}}

The first step in the proof of the proposition, is the computation
of expectations on exponential scale in the lemma below.
\begin{lem}
\label{lem:Z_1st_mom}For large enough $\beta$ we have the following.
Let $a_{N}>0$ be a sequence such that $a_{N}/N\to0$ and set $J_{N}=\left(m_{N}-a_{N},m_{N}+a_{N}\right)$.
Then for any $c$, $\epsilon>0$,

\begin{equation}
\limsup_{N\to\infty}\sup_{u\in J_{N}}\left|\frac{1}{N}\log\left(\mathbb{E}_{u,0}\left\{ Z_{N,\beta}\left({\rm Band}\left(cN^{-1/2}\right)\right)\right\} \right)-\Frf\left(E_{0},q_{*}\right)\right|=0,\label{eq:1812_2}
\end{equation}
and

\begin{equation}
\limsup_{N\to\infty}\sup_{u\in J_{N}}\frac{1}{N}\log\left(\mathbb{E}_{u,0}\left\{ Z_{N,\beta}\left({\rm Cap}\setminus{\rm Band}\left(\epsilon\right)\right)\right\} \right)<\Frf\left(E_{0},q_{*}\right).\label{eq:13}
\end{equation}
\end{lem}
\begin{proof}
If $\left\{ q_{i}\right\} _{i=0}^{k}$ is a finite sequence such that,
with $j<k$, 
\begin{equation}
q_{**}=q_{0}<q_{1}<\cdots<q_{j}=q_{*}-\epsilon<q_{*}+\epsilon=q_{j+1}<\cdots<q_{k}=1,\label{eq:15}
\end{equation}
then 
\begin{equation}
\mathbb{E}_{u,0}\left\{ Z_{N,\beta}\left({\rm Cap}\setminus{\rm Band}\left(\epsilon\right)\right)\right\} =\sum_{i\neq j}\mathbb{E}_{u,0}\left\{ Z_{N,\beta}\left({\rm Band}\left(\hat{\mathbf{n}},q_{i},q_{i+1}\right)\right)\right\} .\label{eq:1812_1}
\end{equation}

Using the co-area formula with the mapping $\boldsymbol{\sigma}\mapsto R\left(\boldsymbol{\sigma},\hat{\mathbf{n}}\right)$
we obtain
\begin{align}
\mathbb{E}_{u,0}\left\{ Z_{N,\beta}\left({\rm Band}\left(\hat{\mathbf{n}},q_{i},q_{i+1}\right)\right)\right\}  & =\int_{q_{i}}^{q_{i+1}}\Phi_{N,\beta,u}^{\left(1\right)}\left(q\right)dq,\label{eq:16}
\end{align}
where, using Corollary \ref{cor:conditional laws} and (\ref{eq:1220-2}),
\begin{align}
\Phi_{N,\beta,u}^{\left(1\right)}\left(q\right) & \triangleq\frac{\omega_{N-1}}{\omega_{N}}\left(1-q^{2}\right)^{\frac{N-3}{2}}\Xi_{N,\beta,u}^{\left(1\right)}\left(q\right),\label{eq:Phi1}\\
\Xi_{N,\beta,u}^{\left(1\right)}\left(q\right) & \triangleq\mathbb{E}_{u,0}\left\{ \exp\left\{ -\beta H_{N}\left(\boldsymbol{\sigma}_{q}\right)\right\} \right\} \nonumber \\
 & =\exp\left\{ -\beta uq^{p}+\frac{1}{2}\beta^{2}N\left(1-q^{2p}-pq^{2p-2}\left(1-q^{2}\right)\right)\right\} ,\nonumber 
\end{align}
with $\boldsymbol{\sigma}_{q}$ being an arbitrary point on the sphere
such that $R\left(\boldsymbol{\sigma}_{q},\hat{\mathbf{n}}\right)=q$.
Let $\delta>0$ and note that, with $\Frf\left(E,q\right)$ as defined
in (\ref{eq:6}), uniformly in $q\in\left(-1+\delta,1-\delta\right)$,
\[
\lim_{N\to\infty}\left|\frac{1}{N}\log\left(\Phi_{N,\beta,-NE_{0}}^{\left(1\right)}\left(q\right)\right)-\Frf\left(E_{0},q\right)\right|=0,
\]
and, 
\[
\limsup_{N\to\infty}\sup_{u\in J_{N}}\frac{1}{N}\left|\log\left(\vphantom{\Phi_{N,\beta,-NE_{0}}^{\left(1\right)}}\mathbb{E}_{u,0}\left\{ Z_{N,\beta}\left({\rm Band}\left(\hat{\mathbf{n}},q-s,q+s\right)\right)\right\} \right)-\log\left(\Phi_{N,\beta,-NE_{0}}^{\left(1\right)}\left(q\right)\right)\right|=o\left(s\right).
\]
For small enough $\delta>0$, 
\[
\sup_{u\in J_{N}}\frac{1}{N}\log\left(\mathbb{E}_{u,0}\left\{ Z_{N,\beta}\left({\rm Cap}\left(\hat{\mathbf{n}},1-\delta\right)\right)\right\} \right)
\]
is as negative as we wish. Hence, for small enough $\delta$, denoting
$I=(q_{**},1-\delta)\setminus[q_{*}-\epsilon,q_{*}+\epsilon]$, using
(\ref{eq:1812_1}) and refining the partition (\ref{eq:15}) if needed,
we have that the left-hand side of (\ref{eq:13}) is bounded from
above by 
\[
\limsup_{N\to\infty}\sup_{u\in J_{N}}\sup_{q\in I}\frac{1}{N}\log\left(\Phi_{N,\beta,-NE_{0}}^{\left(1\right)}\left(q\right)\right)\leq\sup_{q\in I}\Frf\left(E_{0},q\right).
\]
In a similar manner, from the representation (\ref{eq:16}) with $q_{i}=q_{*}-cN^{-1/2}$
and $q_{i+1}=q_{*}+cN^{-1/2}$ and since $\Frf\left(E_{0},q\right)$
is continuous in $q$ in a neighborhood of $q_{*}$, we have that
(\ref{eq:1812_2}) holds.

Thus, in order to finish the proof it is enough to show that $q_{*}$
is the unique maximum point of $\Frf\left(E_{0},q\right)$ in the
interval $q\in\left[q_{**},1\right)$. This follows since $q_{*}$
is the only critical point (see Section \ref{sec:overlapsVals}) in
the interior of the interval and since by (\ref{eq:14031}), $\frac{\partial^{2}}{\partial q^{2}}\Frf\left(E_{0},q_{*}\right)<0$,
for large$\beta$. \end{proof}
\begin{rem}
\label{rem:18}From (\ref{eq:16}) it also follows that for any $E>0$
and $q\in\left(-1,1\right)$, if $a_{N},\,\epsilon_{N}>0$ are sequences
such that $a_{N}=o(N)$, $\epsilon_{N}=o(N)$ and $\log\epsilon_{N}=o(N)$,
then, setting $J_{N}=\left(-NE-a_{N},-NE+a_{N}\right)$,

\[
\limsup_{N\to\infty}\sup_{u\in J_{N}}\left|\frac{1}{N}\log\left(\mathbb{E}_{u,0}\left\{ Z_{N,\beta}\left({\rm Band}\left(\hat{\mathbf{n}},q,q+\epsilon_{N}\right)\right)\right\} \right)-\Frf\left(E,q\right)\right|=0.
\]
We note that $\Frf\left(E,q\right)$ coincides (when taking into account
small differences in the definition of the model) with the TAP free
energy computed by Kurchan, Parisi and Virasoro \cite[Eq. (20)]{KurchanParisiVirasoro}
and Crisanti and Sommers \cite[Eq. (7)]{CrisantiSommersTAPpspin}
for a pure state at energy $E$ and such that the overlap inside the
state is $q^{2}$. Recall that Theorem \ref{thm:Geometry} states
that if two samples are taken from a band of overlaps approximately
$q_{*}$ as in (\ref{eq:thm1_1}), then the overlap of the samples
is generically equal to $q_{*}^{2}$. The proof of this fact generalizes
to any $q>q_{c}$ instead of $q_{*}$. The condition $q>q_{c}$ corresponds
to \cite[Eq. (25)]{KurchanParisiVirasoro} which is referred to in
\cite{KurchanParisiVirasoro} as a condition for stability with respect
to fluctuations inside a cluster.
\end{rem}
We continue with the proof of Proposition \ref{prop:means}. It follows
from Lemma \ref{lem:Z_1st_mom} that there exists some sequence $\epsilon_{N}$
such that $\epsilon_{N}\to0$ as $N\to\infty$ 
\[
\lim_{N\to\infty}\sup_{u\in J_{N}}\frac{\mathbb{E}_{u,0}\left\{ Z_{N,\beta}\left({\rm Cap}\setminus{\rm Band}\left(\epsilon_{N}\right)\right)\right\} }{\mathbb{E}_{u,0}\left\{ Z_{N,\beta}\left({\rm Band}\left(cN^{-1/2}\right)\right)\right\} }=0,
\]
for any $c>0$ (with no information provided, however, on the rate
of convergence of $\epsilon_{N}$ to $0$). In order to prove (\ref{eq:13-2}),
what remains to show is that 
\begin{equation}
\lim_{c\to\infty}\lim_{N\to\infty}\sup_{u\in J_{N}}\frac{\mathbb{E}_{u,0}\left\{ Z_{N,\beta}\left({\rm Band}\left(\epsilon_{N}\right)\setminus{\rm Band}\left(cN^{-1/2}\right)\right)\right\} }{\mathbb{E}_{u,0}\left\{ Z_{N,\beta}\left({\rm Band}\left(cN^{-1/2}\right)\right)\right\} }=0.\label{eq:32}
\end{equation}
By (\ref{eq:16}), with 
\[
A_{1}:=A_{1,N}=\left(q_{*}-\epsilon_{N},q_{*}+\epsilon_{N}\right),\,\,\, A_{2}:=A_{2,N}=\left(q_{*}-cN^{-1/2},q_{*}+cN^{-1/2}\right),
\]
the ratio of expectations in (\ref{eq:32}) is equal to
\begin{equation}
\frac{\int_{A_{1}\setminus A_{2}}\Phi_{N,\beta,u}^{\left(1\right)}\left(q\right)dq}{\int_{A_{2}}\Phi_{N,\beta,u}^{\left(1\right)}\left(q\right)dq}.\label{eq:38}
\end{equation}

Recall that in (\ref{eq:14031}) we showed that for $\beta$ large,
$\frac{\partial^{2}}{\partial q^{2}}\Frf\left(E_{0},q_{*}\right)<0$
and that $q_{*}$ was chosen such that $\frac{\partial}{\partial q}\Frf\left(E_{0},q_{*}\right)=0$
(see (\ref{eq:14})-(\ref{eq:7})). Therefore, as $q\to q_{*}$,
\[
\Frf\left(E_{0},q\right)=\Frf\left(E_{0},q_{*}\right)+\frac{1}{2}\frac{\partial^{2}}{\partial q^{2}}\Frf\left(E_{0},q_{*}\right)\left(q-q_{*}\right)^{2}+o\left(\left(q-q_{*}\right)^{2}\right).
\]
From our assumption made in Proposition \ref{prop:means} that $a_{N}=o\left(\sqrt{N}\right)$,
uniformly in $u\in J_{N}$ and $q\in q_{*}+\left(-\epsilon_{N},\epsilon_{N}\right)\cup\left(-cN^{-1/2},cN^{-1/2}\right)$,
with some $v\left(q,u,N\right)=o\left(\sqrt{N}\right)$, 
\[
\left(\beta E_{0}N+\beta u\right)q^{p}=\left(\beta E_{0}N+\beta u\right)q_{*}^{p}+v\left(q,u,N\right)\cdot\left(q-q_{*}\right),
\]
and, from (\ref{eq:Phi1}), 
\begin{align*}
\Phi_{N,\beta,u}^{\left(1\right)}\left(q\right) & =\frac{\omega_{N-1}}{\omega_{N}}\left(1-q^{2}\right)^{-3/2}\exp\left\{ N\Frf\left(E_{0},q\right)-\left(\beta E_{0}N+\beta u\right)q^{p}\right\} \\
 & =\left(1+o\left(1\right)\right)\sqrt{\frac{N}{2\pi}}\left(1-q_{*}^{2}\right)^{-3/2}\exp\left\{ N\Frf\left(E_{0},q_{*}\right)-\left(\beta E_{0}N+\beta u\right)q_{*}^{p}\right\} \\
 & \times\exp\left\{ \frac{1}{2}N\frac{\partial^{2}}{\partial q^{2}}\Frf\left(E_{0},q_{*}\right)\cdot\left(q-q_{*}\right)^{2}+N\cdot o\left(\left(q-q_{*}\right)^{2}\right)+v\left(q,u,N\right)\cdot\left(q-q_{*}\right)\right\} ,
\end{align*}
where we used the fact that $\sqrt{N}\omega_{N}/\omega_{N-1}\to\sqrt{2\pi}$
as $N\to\infty$ and the $o\left(1\right)$ term is with respect to
taking $N\to\infty$. 

By the change of variables $\sqrt{N}\left(q-q_{*}\right)\mapsto s$
we obtain that, uniformly in $u\in J_{N}$, 
\begin{align*}
 & \lim_{c\to\infty}\limsup_{N\to\infty}\frac{\int_{A_{1}\setminus A_{2}}\Phi_{N,\beta,u}^{\left(1\right)}\left(q\right)dq}{\exp\left\{ N\Frf\left(E_{0},q_{*}\right)-\left(\beta E_{0}N+\beta u\right)q_{*}^{p}\right\} }\\
 & \leq\lim_{c\to\infty}\left(1-q_{*}^{2}\right)^{-3/2}\frac{1}{\sqrt{2\pi}}\int_{\mathbb{R}\setminus\left(-c,c\right)}\exp\left\{ \frac{1}{2}\frac{\partial^{2}}{\partial q^{2}}\Frf\left(q_{*}\right)s^{2}\right\} ds=0.
\end{align*}
Similarly, uniformly in $u\in J_{N}$,
\begin{align*}
 & \lim_{c\to\infty}\lim_{N\to\infty}\frac{\int_{A_{2}}\Phi_{N,\beta,u}^{\left(1\right)}\left(q\right)dq}{\exp\left\{ N\Frf\left(E_{0},q_{*}\right)-\left(\beta E_{0}N+\beta u\right)q_{*}^{p}\right\} }\\
 & =\lim_{c\to\infty}\left(1-q_{*}^{2}\right)^{-3/2}\frac{1}{\sqrt{2\pi}}\int_{\left(-c,c\right)}\exp\left\{ \frac{1}{2}\frac{\partial^{2}}{\partial q^{2}}\Frf\left(q_{*}\right)s^{2}\right\} ds\\
 & =\left(1-q_{*}^{2}\right)^{-3/2}\left|\frac{\partial^{2}}{\partial q^{2}}\Frf\left(E_{0},q_{*}\right)\right|^{-1/2}.
\end{align*}
This proves (\ref{eq:32}) and therefore (\ref{eq:13-2}). By (\ref{eq:16}),
the last equation also proves (\ref{eq:13-3}), which completes the
proof of Proposition \ref{prop:means}.\qed

\subsection{\label{sub:2nd}A second moment calculation}

For $\boldsymbol{\sigma}$, $\boldsymbol{\sigma}'$, $\boldsymbol{\sigma}_{0}\in\mathbb{S}^{N-1}$
define the projective overlap of $\boldsymbol{\sigma}$ and $\boldsymbol{\sigma}'$
relative to $\boldsymbol{\sigma}_{0}$ by 
\begin{equation}
R_{\boldsymbol{\sigma}_{0}}\left(\boldsymbol{\sigma},\boldsymbol{\sigma}'\right)=R\left(\boldsymbol{\sigma}-R(\boldsymbol{\sigma},\boldsymbol{\sigma}_{0})\boldsymbol{\sigma}_{0},\,\boldsymbol{\sigma}'-R(\boldsymbol{\sigma}',\boldsymbol{\sigma}_{0})\boldsymbol{\sigma}_{0}\right).\label{eq:projective overlap}
\end{equation}
For any Borel set $B\subset\mathbb{S}^{N-1}$ and $I_{R}\subset\left[-1,1\right]$,
define the subset 
\begin{equation}
T_{\boldsymbol{\sigma}_{0}}\left(B;I_{R}\right)\triangleq\left\{ \left(\boldsymbol{\sigma},\boldsymbol{\sigma}'\right)\in B\times B\,:\, R_{\boldsymbol{\sigma}_{0}}\left(\boldsymbol{\sigma},\boldsymbol{\sigma}'\right)\in I_{R}\right\} .\label{eq:Tsigma}
\end{equation}
For any Borel set $B_{2}\subset\left(\mathbb{S}^{N-1}\right)^{2}$
define
\begin{equation}
\left(Z\times Z\right)_{N,\beta}\left(B_{2}\right)\triangleq\int_{B_{2}}\exp\left\{ -\beta\left(H_{N}\left(\boldsymbol{\sigma}\right)+H_{N}\left(\boldsymbol{\sigma}'\right)\right)\right\} d\mu_{N}\otimes\mu_{N}\left(\boldsymbol{\sigma},\boldsymbol{\sigma}'\right),\label{eq:ZtimesZ}
\end{equation}
and, by an abuse of notation,
\begin{equation}
\left(Z\times Z\right)_{N,\beta}\left({\rm Band}\left(\boldsymbol{\sigma}_{0},q,q'\right);I_{R}\right)\triangleq\left(Z\times Z\right)_{N,\beta}\left(T_{\boldsymbol{\sigma}_{0}}\left({\rm Band}\left(\boldsymbol{\sigma}_{0},q,q'\right);I_{R}\right)\right).\label{eq:ZtimesZ2}
\end{equation}
Note that $\left(Z\times Z\right)_{N,\beta}\left({\rm Band}\left(\boldsymbol{\sigma}_{0},q,q'\right);\left[-1,1\right]\right)=\left(Z_{N,\beta}\left({\rm Band}\left(\boldsymbol{\sigma}_{0},q,q'\right)\right)\right)^{2}$. 
\begin{lem}
\label{lem:Z_2nd_mom}For large enough $\beta$ we have the following.
Let $a_{N}>0$ be a sequence such that $a_{N}/N\to0$ and set $J_{N}=\left(m_{N}-a_{N},m_{N}+a_{N}\right)$.
Define $C_{*}$ by (\ref{eq:Cstar}). Then:
\begin{enumerate}
\item \label{enu:1_lem2nd}For any $c>0$,
\begin{equation}
\lim_{N\to\infty}\sup_{u\in J_{N}}\left|\frac{1}{N}\log\left(\mathbb{E}_{u,0}\left\{ \left(Z_{N,\beta}\left({\rm Band}\left(cN^{-1/2}\right)\right)\right)^{2}\right\} \right)-2\Frf\left(E_{0},q_{*}\right)\right|=0.\label{eq:29}
\end{equation}

\item \label{enu:2_lem2nd}For any $c>0$ and $\rho_{0}>0$,
\begin{equation}
\limsup_{N\to\infty}\sup_{u\in J_{N}}\frac{1}{N}\log\left(\mathbb{E}_{u,0}\left\{ \left(Z\times Z\right)_{N,\beta}\left({\rm Band}\left(cN^{-1/2}\right);\left[-1,1\right]\setminus\left(-\rho_{0},\rho_{0}\right)\right)\right\} \right)<2\Frf\left(E_{0}q_{*}\right).\label{eq:12}
\end{equation}

\item \label{enu:3_lem2nd}For any $c>0$,
\begin{equation}
\lim_{\rho\to\infty}\limsup_{N\to\infty}\sup_{u\in J_{N}}\frac{\mathbb{E}_{u,0}\left\{ \left(Z\times Z\right)_{N,\beta}\left({\rm Band}\left(cN^{-1/2}\right);\left[-1,1\right]\setminus\left(-\rho N^{-1/2},\rho N^{-1/2}\right)\right)\right\} }{\left(\mathbb{E}_{u,0}\left\{ Z_{N,\beta}\left({\rm Band}\left(cN^{-1/2}\right)\right)\right\} \right)^{2}}=0.\label{eq:12-4}
\end{equation}

\item \label{enu:4_lem2nd}For any $c>0$,
\begin{equation}
\lim_{N\to\infty}\sup_{u\in J_{N}}\left|\frac{\mathbb{E}_{u,0}\left\{ \left(Z_{N,\beta}\left({\rm Band}\left(cN^{-1/2}\right)\right)\right)^{2}\right\} }{\left(\mathbb{E}_{u,0}\left\{ Z_{N,\beta}\left({\rm Band}\left(cN^{-1/2}\right)\right)\right\} \right)^{2}}-\frac{1}{\sqrt{C_{*}}}\right|=0.\label{eq:29-2}
\end{equation}

\end{enumerate}
\end{lem}
\begin{proof}
Using the co-area formula with the mapping 
\begin{equation}
\left(\boldsymbol{\sigma},\boldsymbol{\sigma}'\right)\mapsto\left(q_{1},q_{2},\varrho\right)=\left(R\left(\boldsymbol{\sigma},\hat{\mathbf{n}}\right),R\left(\boldsymbol{\sigma}',\hat{\mathbf{n}}\right),R_{\hat{\mathbf{n}}}\left(\boldsymbol{\sigma},\boldsymbol{\sigma}'\right)\right),\label{eq:21}
\end{equation}
we have that, for $I_{R}=\left[-1,1\right]$ or $I_{R}=\left[-1,1\right]\setminus\left(-\rho_{0},\rho_{0}\right)$,
\begin{align}
 & \negthickspace\negthickspace\mathbb{E}_{u,0}\left\{ \left(Z\times Z\right)_{N,\beta}\left({\rm Band}\left(cN^{-1/2}\right);I_{R}\right)\right\} \label{eq:17}\\
 & =\int_{q_{*}-cN^{-1/2}}^{q_{*}+cN^{-1/2}}dq_{1}\int_{q_{*}-cN^{-1/2}}^{q_{*}+cN^{-1/2}}dq_{2}\int_{I_{R}}d\varrho\Phi_{N,\beta,u}^{\left(2\right)}\left(q_{1},q_{2},\varrho\right),\nonumber 
\end{align}
where, using Corollary \ref{cor:conditional laws} and (\ref{eq:1220-2}),

\begin{align}
\Phi_{N,\beta,u}^{\left(2\right)}\left(q_{1},q_{2},\varrho\right) & \triangleq\frac{\omega_{N-1}\omega_{N-2}}{\omega_{N}^{2}}\left(1-q_{1}^{2}\right)^{\frac{N-3}{2}}\left(1-q_{2}^{2}\right)^{\frac{N-3}{2}}\left(1-\varrho^{2}\right)^{\frac{N-4}{2}}\Xi_{N,\beta,u}^{\left(2\right)}\left(q_{1},q_{2},\varrho\right),\label{eq:Phi2}\\
\Xi_{N,\beta,u}^{\left(2\right)}\left(q_{1},q_{2},\varrho\right) & \triangleq\mathbb{E}_{u,0}\left\{ \exp\left\{ -\beta H_{N}\left(\boldsymbol{\sigma}\right)-\beta H_{N}\left(\boldsymbol{\sigma}'\right)\right\} \right\} \label{eq:Xi2}\\
 & =\Xi_{N,\beta,u}^{\left(1\right)}\left(q_{1}\right)\cdot\Xi_{N,\beta,u}^{\left(1\right)}\left(q_{2}\right)\cdot\exp\left\{ \beta^{2}N\left(\left(\left(1-q_{1}^{2}\right)^{\frac{1}{2}}\left(1-q_{2}^{2}\right)^{\frac{1}{2}}\varrho+q_{1}q_{2}\right)^{p}\right.\right.\nonumber \\
 & \left.\left.\vphantom{\left(\left(1-r_{1}^{2}\right)^{-\frac{1}{2}}\left(1-r_{2}^{2}\right)^{-\frac{1}{2}}\left\langle \theta_{1},\theta_{2}\right\rangle +r_{1}r_{2}\right)^{p}}-q_{1}^{p}q_{2}^{p}-pq_{1}^{p-1}q_{2}^{p-1}\left(1-q_{1}^{2}\right)^{\frac{1}{2}}\left(1-q_{2}^{2}\right)^{\frac{1}{2}}\varrho\right)\right\} ,\nonumber 
\end{align}
and where the relation between $\left(\boldsymbol{\sigma},\boldsymbol{\sigma}'\right)$
and $\left(q_{1},q_{2},\varrho\right)$ in the last equation is as
in (\ref{eq:21}). It follows that 
\begin{align}
\lim_{N\to\infty}\sup_{u\in J_{N}}\frac{1}{N} & \left|\log\left(\mathbb{E}_{u,0}\left\{ \left(Z\times Z\right)_{N,\beta}\left({\rm Band}\left(cN^{-1/2}\right);I_{R}\right)\right\} \right)\vphantom{\left(\Phi_{N,\beta,-NE_{0}}^{\left(2\right)}\left(r_{*},r_{*},\rho\right)\right)}\right.\label{eq:11}\\
 & \left.-\sup_{\varrho\in I_{R}}\log\left(\Phi_{N,\beta,-NE_{0}}^{\left(2\right)}\left(q_{*},q_{*},\varrho\right)\right)\right|=0.\nonumber 
\end{align}

Set 
\[
A_{0}\left(\varrho\right)=\lim_{N\to\infty}\frac{1}{N}\log\left(\Phi_{N,\beta,-NE_{0}}^{\left(2\right)}\left(q_{*},q_{*},\varrho\right)\right)
\]
and note that for any $\delta>0$ the convergence is uniform in $\varrho\in\left(-1+\delta,1-\delta\right)$.
Also, for small enough $\delta>0$, 
\[
\limsup_{N\to\infty}\sup_{\varrho\notin\left(-1+\delta,1-\delta\right)}\frac{1}{N}\log\left(\Phi_{N,\beta,-NE_{0}}^{\left(2\right)}\left(q_{*},q_{*},\varrho\right)\right)
\]
is as negative as we wish. Therefore, part (\ref{enu:2_lem2nd}) of
the lemma will follow if we show that $A_{0}\left(\varrho\right)$
attains its maximum over $\left(-1,1\right)$ uniquely at $\varrho=0$.
Since $A_{0}\left(0\right)=2\Frf\left(E_{0},q_{*}\right)$, the latter
also implies part (\ref{enu:1_lem2nd}) of the lemma.

We note that 
\[
\sup_{\rho\in I_{R}}A_{0}\left(\varrho\right)=\log\left(1-q_{*}^{2}\right)+2\beta E_{0}q_{*}^{p}+\beta^{2}\left(1-q_{*}^{2p}-pq_{*}^{2p-2}\left(1-q_{*}^{2}\right)\right)+\sup_{\varrho\in I_{R}}A_{1}\left(\varrho\right),
\]
where
\[
A_{1}\left(\varrho\right)=\frac{1}{2}\log\left(1-\varrho^{2}\right)+\beta^{2}\left(\left(\left(1-q_{*}^{2}\right)\varrho+q_{*}^{2}\right)^{p}-q_{*}^{2p}-pq_{*}^{2p-2}\left(1-q_{*}^{2}\right)\varrho\right).
\]

By expanding $\left(\left(1-q_{*}^{2}\right)\rho+q_{*}^{2}\right)^{p}$,
one can verify that for $\varrho\in\left(0,1\right)$, $A_{0}\left(\varrho\right)>A_{0}\left(-\varrho\right)$.
Also,
\begin{align}
\frac{d}{d\varrho}A_{1}\left(\varrho\right) & =-\frac{\varrho}{1-\varrho^{2}}+\beta^{2}\left(p\left(\left(1-q_{*}^{2}\right)\varrho+q_{*}^{2}\right)^{p-1}\left(1-q_{*}^{2}\right)-pq_{*}^{2p-2}\left(1-q_{*}^{2}\right)\right)\nonumber \\
 & =-\frac{\varrho}{1-\varrho^{2}}+p\beta^{2}\left(1-q_{*}^{2}\right)\sum_{k=1}^{p-1}\binom{p-1}{k}\left(\left(1-q_{*}^{2}\right)\varrho\right)^{k}q_{*}^{2\left(p-1-k\right)}\nonumber \\
 & <-\varrho C_{\beta}^{*}\left(1+o\left(1\right)\right),\label{eq:31}
\end{align}
as $\beta\to\infty$, uniformly in $\varrho$, where we relied on
the fact that as $\beta\to\infty$, $q_{*}\to1$. As noted immediately
after its definition (\ref{eq:Cstar}), $C_{\beta}^{*}$ is positive.\textcolor{red}{{}
}Hence, if $\beta$ is large enough, the derivative $\frac{d}{d\varrho}A_{1}\left(\varrho\right)$
is negative for $\rho\in\left(0,1\right)$. This implies that the
maximum of $A_{1}\left(\varrho\right)$ in $\left(-1,1\right)$ is
attained uniquely at $\varrho=0$ and proves parts (\ref{enu:1_lem2nd})
and (\ref{enu:2_lem2nd}) of the lemma.

We move on to the next parts. First note that from parts (\ref{enu:1_lem2nd})
and (\ref{enu:2_lem2nd}) of the lemma there exists some sequence
$\rho_{N}\in\left(0,1\right)$ with $\rho_{N}\to0$, as $N\to\infty$,
such that
\[
\lim_{N\to\infty}\sup_{u\in J_{N}}\frac{\mathbb{E}_{u,0}\left\{ \left(Z\times Z\right)_{N,\beta}\left({\rm Band}\left(cN^{-1/2}\right);\left[-1,1\right]\setminus\left(-\rho_{N},\rho_{N}\right)\right)\right\} }{\left(\mathbb{E}_{u,0}\left\{ Z_{N,\beta}\left({\rm Band}\left(cN^{-1/2}\right)\right)\right\} \right)^{2}}=0,
\]
for any $c>0$. In order to prove part (\ref{enu:3_lem2nd}) it is
therefore enough to show that, with $I_{R,N}:=I_{R,N}\left(\rho\right)\triangleq\left(-\rho_{N},\rho_{N}\right)\setminus\left(-\rho N^{-1/2},\rho N^{-1/2}\right)$,
\[
\lim_{\rho\to\infty}\lim_{N\to\infty}\sup_{u\in J_{N}}\frac{\mathbb{E}_{u,0}\left\{ \left(Z\times Z\right)_{N,\beta}\left({\rm Band}\left(cN^{-1/2}\right);I_{R,N}\right)\right\} }{\left(\mathbb{E}_{u,0}\left\{ Z_{N,\beta}\left({\rm Band}\left(cN^{-1/2}\right)\right)\right\} \right)^{2}}=0.
\]
By substitution of (\ref{eq:16}) and (\ref{eq:17}), this is equivalent
to
\[
\lim_{\rho\to\infty}\lim_{N\to\infty}\sup_{u\in J_{N}}\frac{\int_{q_{*}-cN^{-1/2}}^{q_{*}+cN^{-1/2}}dq_{1}\int_{q_{*}-cN^{-1/2}}^{q_{*}+cN^{-1/2}}dq_{2}\int_{I_{R,N}}d\varrho\Phi_{N,\beta,u}^{\left(2\right)}\left(q_{1},q_{2},\varrho\right)}{\int_{q_{*}-cN^{-1/2}}^{q_{*}+cN^{-1/2}}dq_{1}\int_{q_{*}-cN^{-1/2}}^{q_{*}+cN^{-1/2}}dq_{2}\Phi_{N,\beta,u}^{\left(1\right)}\left(q_{1}\right)\Phi_{N,\beta,u}^{\left(1\right)}\left(q_{2}\right)}=0.
\]
Therefore, in order to prove (\ref{eq:12-4}) it is enough to show
that
\[
\lim_{\rho\to\infty}\lim_{N\to\infty}\sup_{u\in J_{N}}\sup_{q_{i}:\left|q_{i}-q_{*}\right|\leq cN^{-1/2}}\frac{\int_{I_{R,N}}d\rho\Phi_{N,\beta,u}^{\left(2\right)}\left(q_{1},q_{2},\varrho\right)}{\Phi_{N,\beta,u}^{\left(1\right)}\left(q_{1}\right)\Phi_{N,\beta,u}^{\left(1\right)}\left(q_{2}\right)}=0,
\]
which from (\ref{eq:Phi1}), (\ref{eq:Phi2}), and the fact that $\sqrt{N}\omega_{N}/\omega_{N-1}\to\sqrt{2\pi}$
as $N\to\infty$, is equivalent to 
\begin{equation}
\lim_{\rho\to\infty}\lim_{N\to\infty}\sup_{u\in J_{N}}\sup_{q_{i}:\left|q_{i}-q_{*}\right|\leq cN^{-1/2}}\sqrt{\frac{N}{2\pi}}\int_{I_{R,N}}d\varrho\left(1-\varrho^{2}\right)^{\frac{N-4}{2}}\frac{\Xi_{N,\beta,u}^{\left(2\right)}\left(q_{1},q_{2},\varrho\right)}{\Xi_{N,\beta,u}^{\left(1\right)}\left(q_{1}\right)\Xi_{N,\beta,u}^{\left(1\right)}\left(q_{2}\right)}=0.\label{eq:30}
\end{equation}

From (\ref{eq:Xi2}), 
\begin{align}
\left(1-\varrho^{2}\right)^{\frac{N-4}{2}}\frac{\Xi_{N,\beta,u}^{\left(2\right)}\left(q_{1},q_{2},\varrho\right)}{\Xi_{N,\beta,u}^{\left(1\right)}\left(q_{1}\right)\Xi_{N,\beta,u}^{\left(1\right)}\left(q_{2}\right)} & =\left(1-\varrho^{2}\right)^{-2}\exp\left\{ -\frac{1}{2}C_{*}N\varrho^{2}+N\beta^{2}\sum_{i=2}^{p}K_{i,N}\left(q_{1},q_{2}\right)\varrho^{i}\right\} ,\label{eq:34}
\end{align}
where the coefficients $K_{i,N}:=K_{i,N}\left(q_{1},q_{2}\right)$
satisfy 
\[
\lim_{N\to\infty}K_{2,N}\left(q_{1},q_{2}\right)=0\mbox{\,\,\,\ and\,\,\,}\left|K_{i,N}\left(q_{1},q_{2}\right)\right|<K,\,\,\forall i\leq p,
\]
for some appropriate constant $K>0$, uniformly in $q_{1}$, $q_{2}\in\left(q_{*}-cN^{-1/2},q_{*}+cN^{-1/2}\right)$
and independently of $u$. Hence, by a change of variables, the left-hand
side of (\ref{eq:30}) is bounded from above by 
\[
\lim_{\rho\to\infty}\limsup_{N\to\infty}\int_{\sqrt{N}I_{R,N}}\frac{1}{\sqrt{2\pi}}e^{-\frac{1}{2}C_{*}\varrho^{2}}d\varrho\leq\lim_{\rho\to\infty}\int_{\left(-\rho,\rho\right)^{c}}\frac{1}{\sqrt{2\pi}}e^{-\frac{1}{2}C_{*}\varrho^{2}}d\varrho,
\]
where we have used the fact that $\rho_{N}\to0$ as $N\to\infty$
to neglect high powers of $\varrho$ in the limit. Part (\ref{enu:3_lem2nd})
of the lemma follows.

From (\ref{eq:12-4}), in order to prove (\ref{eq:29-2}) it is enough
to prove that, with $I_{R,N}^{\prime}:=I_{R,N}^{\prime}\left(\rho\right)\triangleq\left(-\rho N^{-1/2},\rho N^{-1/2}\right)$,
\[
\lim_{\rho\to\infty}\limsup_{N\to\infty}\sup_{u\in J_{N}}\left|\frac{\mathbb{E}_{u,0}\left\{ \left(Z\times Z\right)_{N,\beta}\left({\rm Band}\left(\epsilon_{N}\right);I_{R,N}^{\prime}\right)\right\} }{\left(\mathbb{E}_{u,0}\left\{ Z_{N,\beta}\left({\rm Band}\left(cN^{-1/2}\right)\right)\right\} \right)^{2}}-\frac{1}{\sqrt{C_{*}}}\right|=0.
\]
 Thus, by similar arguments to the above, it is sufficient to show
that
\[
\lim_{\rho\to\infty}\limsup_{N\to\infty}\sup_{u\in J_{N}}\sup_{q_{i}:\left|q_{i}-q_{*}\right|\leq cN^{-1/2}}\left|\sqrt{\frac{N}{2\pi}}\int_{I_{R,N}^{\prime}}d\varrho\left(1-\varrho^{2}\right)^{\frac{N-4}{2}}\frac{\Xi_{N,\beta,u}^{\left(2\right)}\left(q_{1},q_{2},\varrho\right)}{\Xi_{N,\beta,u}^{\left(1\right)}\left(q_{1}\right)\Xi_{N,\beta,u}^{\left(1\right)}\left(q_{2}\right)}-\frac{1}{\sqrt{C_{*}}}\right|=0.
\]
Using (\ref{eq:34}) and a change of variables, this is equivalent
to
\[
\lim_{\rho\to\infty}\left|\int_{\left(-\rho,\rho\right)}\frac{1}{\sqrt{2\pi}}e^{-\frac{1}{2}C_{*}\varrho^{2}}d\varrho-\frac{1}{\sqrt{C_{*}}}\right|=0,
\]
which completes the proof of part (\ref{enu:4_lem2nd}) of the lemma.
\end{proof}
We finish the subsection with the following corollary.
\begin{cor}
\label{cor:2nd_moment}Assume that $\epsilon_{N},\,\epsilon_{N}^{\prime}\to0$
and $N^{-1}\log\left(\epsilon_{N}^{\prime}+\epsilon_{N}\right)\to0$,
as $N\to\infty$. Then Lemma \ref{lem:Z_2nd_mom} is also true if
we replace everywhere ${\rm Band}\left(cN^{-1/2}\right)$ by ${\rm Band}\left(\hat{\mathbf{n}},q_{*}-\epsilon_{N},q_{*}+\epsilon_{N}^{\prime}\right)$.\end{cor}
\begin{proof}
The corollary follows by replacing everywhere $-cN^{-1/2}$ and $cN^{-1/2}$
by $-\epsilon_{N}$ and $\epsilon_{N}^{\prime}$, respectively, in
the proof of Lemma \ref{lem:Z_2nd_mom}.
\end{proof}

\subsection{\label{sub:conv_2to3+}Convergence with `3-and-above' spins averaged}

Recall the definition of the conditional probability $\mathbb{P}_{u,0,\mathbf{\mathbf{A}}}$
and random matrix $\mathbf{G}\left(\hat{\mathbf{n}}\right):=\mathbf{G}_{N-1}\left(\hat{\mathbf{n}}\right)$
given in Remark \ref{rem:conditional laws} and (\ref{eq:G}), respectively.
If $X$ is some random variable, viewing $\mathbb{E}_{u,0,\mathbf{\mathbf{A}}}\left\{ X\right\} $
as a (deterministic) function of $\mathbf{\mathbf{A}}$, $\mathbb{E}_{u,0,\mathbf{\mathbf{G}\left(\hat{\mathbf{n}}\right)}}\left\{ X\right\} $
is a random variable (measurable with respect to $\mathbf{G}\left(\hat{\mathbf{n}}\right)$).
In this subsection we prove the following lemma. 
\begin{lem}
\textcolor{red}{\label{lem:BaikLee}}For large enough $\beta$ we
have the following. Let $\epsilon_{N}>0$ be a sequence such that
$\epsilon_{N}\to0$, as $N\to\infty$. With $Y_{*}$ given by (\ref{eq:48}),

\[
\lim_{N\to\infty}\sup_{u\in\mathbb{R}}\left|\mathbb{P}\left\{ \frac{\mathbb{E}_{u,0,\mathbf{\mathbf{G}\left(\hat{\mathbf{n}}\right)}}\left\{ Z_{N,\beta}\left({\rm Band}\left(\epsilon_{N}\right)\right)\right\} }{\mathbb{E}_{u,0}\left\{ Z_{N,\beta}\left({\rm Band}\left(\epsilon_{N}\right)\right)\right\} }\leq t\right\} -\mathbb{P}\left\{ e^{Y_{*}}\leq t\right\} \right|=0.
\]
\end{lem}
\begin{proof}
Assume throughout the proof that $\beta$ is fixed, but large enough.
We define 
\begin{equation}
c_{{\rm eff}}\left(N,q\right)=\sqrt{\frac{1}{2}\frac{N}{N-1}}\alpha_{2}\left(q\right)\,\,\,\mbox{and}\,\,\,\beta_{{\rm eff}}\left(N,q\right)=\beta\cdot c_{{\rm eff}}\left(N,q\right).\label{eq:beff-1}
\end{equation}
Define the matrix $\mathbf{G}_{0}:=\mathbf{G}_{0,N-1}=\sqrt{\frac{N}{p\left(p-1\right)}}\mathbf{G}_{N-1}\left(\hat{\mathbf{n}}\right)$
and random field
\[
\forall\boldsymbol{\sigma}\in\mathbb{S}^{N-2}:\,\,\, H_{N-1}^{\mathbf{G}_{0}}\left(\boldsymbol{\sigma}\right)=\frac{1}{\sqrt{N-1}}\boldsymbol{\sigma}^{T}\mathbf{G}_{0}\boldsymbol{\sigma},
\]
and note that (see (\ref{eq:22}) for the definition of the restriction)
\begin{equation}
\bar{H}_{N}^{\hat{\mathbf{n}},2}|_{q}\left(\boldsymbol{\sigma}\right)=c_{{\rm eff}}\left(N,q\right)H_{N-1}^{\mathbf{G}_{0}}\left(\boldsymbol{\sigma}\right)\,\,\,\mbox{and}\,\,\, Z_{N-1,\beta}\left(\bar{H}_{N}^{\hat{\mathbf{n}},2}|_{q}\right)=Z_{N-1,\beta_{{\rm eff}}\left(N,q\right)}\left(H_{N-1}^{\mathbf{G}_{0}}\right),\label{eq:HG}
\end{equation}
where for a real function $f$ on $\mathbb{S}^{N-1}$, by abuse of
notation, $Z_{N,\beta}(f)$ denotes the corresponding partition function
\begin{equation}
Z_{N,\beta}(f)=\int\exp\{-\beta f(\boldsymbol{\sigma})\}d\mu_{N}(\boldsymbol{\sigma}).\label{eq:Z(f)}
\end{equation}

Similarly to (\ref{eq:16}), from Lemmas \ref{lem:HNdecomposition}
and \ref{lem:Hhat_expressions} we have that
\begin{align}
 & \negthickspace\negthickspace\negthickspace\frac{\mathbb{E}_{u,0,\mathbf{\mathbf{G}\left(\hat{\mathbf{n}}\right)}}\left\{ Z_{N,\beta}\left({\rm Band}\left(\epsilon_{N}\right)\right)\right\} }{\mathbb{E}_{u,0}\left\{ Z_{N,\beta}\left({\rm Band}\left(\epsilon_{N}\right)\right)\right\} }\nonumber \\
 & =\frac{\int_{q_{*}-\epsilon_{N}}^{q_{*}+\epsilon_{N}}\frac{\omega_{N-1}}{\omega_{N}}e^{-\beta uq^{p}}\left(1-q^{2}\right)^{\frac{N-3}{2}}Z_{\beta}\left(\bar{H}_{N}^{\hat{\mathbf{n}},2}|_{q}\right)\mathbb{E}\left\{ e^{-\beta\bar{H}_{N}^{\hat{\mathbf{n}},3+}\left(\boldsymbol{\sigma}_{q}\right)}\right\} dq}{\int_{q_{*}-\epsilon_{N}}^{q_{*}+\epsilon_{N}}\frac{\omega_{N-1}}{\omega_{N}}e^{-\beta uq^{p}}\left(1-q^{2}\right)^{\frac{N-3}{2}}\mathbb{E}\left\{ Z_{\beta}\left(\bar{H}_{N}^{\hat{\mathbf{n}},2}|_{q}\right)\right\} \mathbb{E}\left\{ e^{-\beta\bar{H}_{N}^{\hat{\mathbf{n}},3+}\left(\boldsymbol{\sigma}_{q}\right)}\right\} dq},\label{eq:84}
\end{align}
where $\boldsymbol{\sigma}_{q}$ is an arbitrary point such that $R\left(\boldsymbol{\sigma}_{q},\hat{\mathbf{n}}\right)=q$.
From the relation (\ref{eq:HG}), since $\beta_{{\rm eff}}\left(N,q_{*}\right)\to\frac{1}{2}\sqrt{1-C_{*}}$
and $\epsilon_{N}\to0$ as $N\to\infty$, in order to prove the lemma
it is sufficient to show that: 
\begin{enumerate}
\item $Y_{N}^{\mathbf{G}}\left(\frac{1}{2}\sqrt{1-C_{*}}\right)\overset{d}{\to}Y_{*}$
as $N\to\infty$, where we define, for any $\beta_{0}>0$, the random
variable
\begin{align*}
Y_{N}^{\mathbf{G}}\left(\beta_{0}\right) & \triangleq\log\left(Z_{N-1,\beta_{0}}\left(H_{N-1}^{\mathbf{G}_{0}}\right)\right)-\log\left(\mathbb{E}\left\{ Z_{N-1,\beta_{0}}\left(H_{N-1}^{\mathbf{G}_{0}}\right)\right\} \right)\\
 & =\log\left(Z_{N-1,\beta_{0}}\left(H_{N-1}^{\mathbf{G}_{0}}\right)\right)-\left(N-1\right)\beta_{0}^{2}.
\end{align*}

\item For any small enough $\tau>0$, the random process $\left\{ Y_{N}^{\mathbf{G}}\left(\beta_{0}\right)\right\} _{\beta_{0}\in D(\delta)}$
on the interval $D(\tau):=\frac{1}{2}\sqrt{1-C_{\beta}^{*}}+\left[-\tau,\tau\right]$
converges in distribution as $N\to\infty$ (or in fact, even just
tight in $N\geq1$) in the space of continuous functions on $D(\tau)$,
equipped with the supremum norm.
\end{enumerate}
The Gaussian random matrix $\mathbf{G}_{0}$ is an $N-1\times N-1$
symmetric matrix whose on-or-above-diagonal entries are independent
with variance $1$ off the diagonal and $2$ on the diagonal. Thus,
the first item above follows from Theorem \ref{thm:BaikLee} since
$H_{N-1}^{\mathbf{G}_{0}}\left(\boldsymbol{\sigma}\right)\overset{d}{=}\sqrt{2}H_{N-1}^{\mbox{pure}\,2}\left(\boldsymbol{\sigma}\right)$
and $\frac{1}{2}\sqrt{1-C_{*}}<\frac{1}{2}$. In order to prove the
second item we rely on certain calculations from the proof of \cite[Theorem 2.10]{BaikLee}
and a result from \cite{BaiSilverstein} concerning convergence of
linear statistics of Wigner matrices. In \cite[Eq. (2.14)]{BaikLee}
the quantity $\widehat{\gamma}(\beta_{0})$ is defined. In the sentence
following this definition, it is noted that $\widehat{\gamma}(\beta_{0})$
is a decreasing function of $\beta_{0}\in(0,\beta_{c}$) and as $\beta_{0}\nearrow\beta_{c}$,
$\widehat{\gamma}(\beta_{0})\searrow C_{+}$ where $C_{+}$ and $\beta_{c}$
are given in \cite[Condition 2.3, Definition 2.8]{BaikLee} and in
our case (of Wigner matrices) $C_{+}=2$ and $\beta_{c}=1/2$. Hence,
if $\tau$ is small enough, then 
\begin{equation}
\inf_{\beta_{0}\in D(\tau)}\widehat{\gamma}(\beta_{0})>C_{+}.\label{eq:ineq1}
\end{equation}
In the proof of \cite[Corollary 5.2]{BaikLee}, the variable $\delta$
is chosen so that $\delta\in\left(0,(\widehat{\gamma}(\beta_{0})-C_{+})/2\right)$
for some fixed $\beta_{0}$. Replacing this requirement on $\delta$
by 
\begin{equation}
\delta\in\left(0,(\inf_{\beta_{0}\in D(\tau)}\widehat{\gamma}(\beta_{0})-C_{+})/2\right),\label{eq:2803-1}
\end{equation}
and using the fact that the derivatives up to order $3$ of $\widehat{G}(z)$
(defined in \cite[Eq. (5.1)]{BaikLee}) are bounded uniformly on $[C_{+}+\delta,K]$
for any $\delta$,$K>0$, all the statements and proofs from \cite[Corollary (5.2)]{BaikLee}
until \cite[Eq. (5.33)]{BaikLee} hold with the following changes,
assuming $\tau$ is small enough. First, any equation that depends
on $\beta_{0}$ (either directly, or implicitly by depending on $\widehat{\gamma}=\widehat{\gamma}(\beta_{0})$
or $\gamma=\gamma(\beta_{0})$; where the latter of the two is defined
in \cite[Lemma 4.1]{BaikLee}) holds simultaneously for all $\beta_{0}\in D(\tau)$.
In particular, equations that hold with probability that goes to $1$
as $N\to\infty$ for fixed $\beta_{0}$ also hold simultaneously for
all $\beta_{0}\in D(\tau)$ with such probability. Second, any of
the estimates of the form $O(t_{N})$ hold uniformly in $\beta_{0}\in D(\tau)$;
that is, in any estimate where a term of the form $O(t_{N})$ appears,
we can replace it by a sequence $\alpha_{N}(\beta_{0})$ satisfying
$\sup_{\beta_{0}\in D(\tau)}|\alpha_{N}(\beta_{0})|=O(t_{N})$.

Hence, from (the modified) \cite[Eq. (5.33)]{BaikLee} (using Definition
2.13, (3.11), (A.4) and (A.7) of \cite{BaikLee}), we have that with
probability tending to $1$ as $N\to\infty$,%
\footnote{We note that, as verified with the authors of \cite{BaikLee}, in
the version currently on the arXiv, there is a small typo in \cite[Eq. (5.33)]{BaikLee}
and the $O(N^{-1+\epsilon})$ should be replaced by $O(N^{-2+\epsilon})$.%
} 
\begin{equation}
Y_{N}^{\mathbf{G}}\left(\beta_{0}\right)=-\frac{1}{2}\mathcal{N}_{\varphi}+\log\left(2\beta_{0}\right)-\frac{1}{2}\log\left(f_{2}\left(\beta_{0}\right)\right)-\frac{1}{2}\log\left(1+\frac{\mathcal{N}_{\psi}}{\left(N-1\right)f_{2}\left(\beta_{0}\right)}\right)+\alpha_{N}(\beta_{0}),\label{eq:85}
\end{equation}
where $|\alpha_{N}(\beta_{0})|=O\left(N^{-1+a}\right)$ with arbitrarily
chosen $a>0$, and where for any function $\phi\left(\beta_{0},x\right)$
with $\lambda_{i}$ denoting the eigenvalues of $\mathbf{G}_{0}/\sqrt{N-1}$,
we abbreviate $\mathcal{N}_{\phi}:=\mathcal{N}_{\phi}\left(\beta_{0},\left(\lambda_{i}\right)_{i=1}^{N-1}\right)$
where 
\begin{equation}
\mathcal{N}_{\phi}\left(\beta_{0},\left(z_{i}\right)_{i=1}^{N-1}\right)=\sum_{i=1}^{N-1}\phi\left(\beta_{0},z_{i}\right)-\left(N-1\right)\int_{-2}^{2}\phi\left(\beta_{0},x\right)d\nu\left(x\right),\label{eq:N_phi}
\end{equation}
$\nu=\nu^{*}$ is the semicircle law given in (\ref{eq:semicirc}),
and with $\hat{\gamma}=2\beta_{0}+(2\beta_{0})^{-1}$, the functions
$\varphi$, $\psi$ and $f_{2}$ are given by
\begin{align}
\varphi\left(\beta_{0},x\right) & =\log\left(\hat{\gamma}-x\right),\,\,\psi\left(\beta_{0},x\right)=\left(\hat{\gamma}-x\right)^{-2},\,\, f_{2}\left(\beta_{0}\right)=-\frac{1}{2}+\frac{\hat{\gamma}}{2\sqrt{\hat{\gamma}^{2}-4}}.\label{eq:f_2}
\end{align}
In particular, $f_{2}\left(\beta_{0}\right)$ is continuously differentiable
in a neighborhood of $\frac{1}{2}\sqrt{1-C_{*}}$, and bounded away
from $0$ uniformly in $\beta_{0}\in D(\tau)$. 

Finally, assuming $\tau$ is small enough, the convergence in distribution
as $N\to\infty$ of the processes 
\[
\left\{ \mathcal{N}_{\varphi\left(\beta_{0},\lambda_{i}\right)}\right\} _{\beta_{0}\in D(\tau)}\,\,\,\mbox{and}\,\,\,\left\{ \mathcal{N}_{\psi\left(\beta_{0},\lambda_{i}\right)}\right\} _{\beta_{0}\in D(\tau)}
\]
in the space of continuous function with the supremum norm, follows
from \cite[Example 9.3]{BaiSilverstein}. Combined with (\ref{eq:85}),
this concludes the proof.
\end{proof}

\subsection{\label{sub:pfGaussFlucts}Proof of Proposition \ref{prop:Gaussflucts}}

In view of Lemma \ref{lem:BaikLee}, Proposition \ref{prop:Gaussflucts}
will follows if we prove the following lemma.
\begin{lem}
\label{lem:ConcentrationZ}For large enough $\beta$ we have the following.
Let $a_{N},\,\epsilon_{N}>0$ be sequences such $a_{N}=o(N)$, $\epsilon_{N}=o(1)$,
and set $J_{N}=\left(m_{N}-a_{N},m_{N}+a_{N}\right)$. Then

\[
\forall\delta>0:\,\,\,\lim_{N\to\infty}\sup_{u\in J_{N}}\mathbb{P}_{u,0}\left\{ \left|\frac{Z_{N,\beta}\left({\rm Band}\left(\epsilon_{N}\right)\right)}{\mathbb{E}_{u,0,\mathbf{\mathbf{G}\left(\hat{\mathbf{n}}\right)}}\left\{ Z_{N,\beta}\left({\rm Band}\left(\epsilon_{N}\right)\right)\right\} }-1\right|\geq\delta\right\} =0.
\]
\end{lem}
\begin{proof}
Denote 
\[
X\left(u\right)=\frac{Z_{N,\beta}\left({\rm Band}\left(\epsilon_{N}\right)\right)}{\mathbb{E}_{u,0,\mathbf{\mathbf{G}\left(\hat{\mathbf{n}}\right)}}\left\{ Z_{N,\beta}\left({\rm Band}\left(\epsilon_{N}\right)\right)\right\} }.
\]
It is enough to show that there exist subsets $\mathcal{A}:=\mathcal{A}_{N,u}$
of the space of real symmetric $N-1\times N-1$ matrices such that
\begin{equation}
\lim_{N\to\infty}\inf_{u\in J_{N}}\mathbb{P}\left\{ \mathbf{G}_{N-1}\left(\hat{\mathbf{n}}\right)\in\mathcal{A}_{N,u}\right\} =1,\label{eq:53}
\end{equation}
(which is equivalent to the same with $\mathbb{P}_{u,0}$, by Lemma
\ref{lem:Hhat_expressions}) and 
\begin{equation}
\forall\delta>0:\,\,\,\lim_{N\to\infty}\sup_{u\in J_{N}}\sup_{\mathbf{A}\in\mathcal{A}_{N,u}}\mathbb{P}_{u,0,\mathbf{A}}\left\{ \left|X\left(u\right)-1\right|\geq\delta\right\} =0.\label{eq:52}
\end{equation}
Since $\mathbb{E}_{u,0,\mathbf{A}}\left\{ X\left(u\right)\right\} =1$
by definition, (\ref{eq:52}) follows by Chebyshev's inequality if
we show that
\begin{equation}
\lim_{N\to\infty}\sup_{u\in J_{N}}\sup_{\mathbf{A}\in\mathcal{A}_{N,u}}\mathbb{E}_{u,0,\mathbf{A}}\left\{ \left(X\left(u\right)\right)^{2}\right\} \leq1.\label{eq:54}
\end{equation}

Let $\rho_{N}\in\left(0,1\right)$ be an arbitrary sequence such that
$\rho_{N}\to0$ and $\rho_{N}\sqrt{N}\to\infty$ as $N\to\infty$.
Define the quantities $Q_{i}:=Q_{i,u,\mathbf{A}}$ by 
\begin{align*}
Q_{1} & \triangleq\mathbb{E}_{u,0,\mathbf{A}}\left\{ \left(Z\times Z\right)_{N,\beta}\left({\rm Band}\left(\epsilon_{N}\right);\left(-\rho_{N},\rho_{N}\right)\right)\right\} ,\\
Q_{2} & \triangleq\mathbb{E}_{u,0,\mathbf{A}}\left\{ \left(Z\times Z\right)_{N,\beta}\left({\rm Band}\left(\epsilon_{N}\right);\left(-\rho_{N},\rho_{N}\right)^{c}\right)\right\} ,\\
Q_{3} & \triangleq\int_{T_{\hat{\mathbf{n}}}\left({\rm Band}\left(\epsilon_{N}\right);\left(-\rho_{N},\rho_{N}\right)\right)}\prod_{i=1,2}\mathbb{E}_{u,0,\mathbf{A}}\left\{ e^{-\beta H_{N}\left(\boldsymbol{\sigma}_{i}\right)}\right\} d\mu_{N}\otimes\mu_{N}\left(\boldsymbol{\sigma}_{1},\boldsymbol{\sigma}_{2}\right),\\
Q_{4} & \triangleq\int_{T_{\hat{\mathbf{n}}}\left({\rm Band}\left(\epsilon_{N}\right);\left(-\rho_{N},\rho_{N}\right)^{c}\right)}\prod_{i=1,2}\mathbb{E}_{u,0,\mathbf{A}}\left\{ e^{-\beta H_{N}\left(\boldsymbol{\sigma}_{i}\right)}\right\} d\mu_{N}\otimes\mu_{N}\left(\boldsymbol{\sigma}_{1},\boldsymbol{\sigma}_{2}\right),
\end{align*}
where $\left(-\rho_{N},\rho_{N}\right)^{c}=\left[-1,1\right]\setminus\left(-\rho_{N},\rho_{N}\right)$
and $T_{\boldsymbol{\sigma}}\left(B;I\right)$ and $\left(Z\times Z\right)_{N,\beta}\left(B;I\right)$
are defined in (\ref{eq:Tsigma}) and (\ref{eq:ZtimesZ}). Since 
\[
\mathbb{E}_{u,0,\mathbf{A}}\left\{ \left(X\left(u\right)\right)^{2}\right\} =\frac{Q_{1}+Q_{2}}{Q_{3}+Q_{4}}\leq\frac{Q_{1}}{Q_{3}}+\frac{Q_{2}}{Q_{3}+Q_{4}},
\]
(\ref{eq:54}) will follow if we show that 
\begin{equation}
\lim_{N\to\infty}\sup_{u\in J_{N}}\sup_{\mathbf{A}\in\mathcal{A}_{N,u}}\frac{Q_{2}}{Q_{3}+Q_{4}}=0,\label{eq:55}
\end{equation}
and 
\begin{equation}
\limsup_{N\to\infty}\sup_{u\in J_{N}}\sup_{\mathbf{A}}\frac{Q_{1}}{Q_{3}}\leq1,\label{eq:56-1}
\end{equation}
where in the last equality the supremum in $\mathbf{A}$ is over all
real, symmetric $N-1\times N-1$ matrices. 

By part (\ref{enu:3_lem2nd}) of Lemma \ref{lem:Z_2nd_mom} and Corollary
\ref{cor:2nd_moment}, there exist sets $\mathcal{A}_{N,u}^{\left(1\right)}$
such that (\ref{eq:53}) holds with $\mathcal{A}_{N,u}^{\left(1\right)}$
instead of $\mathcal{A}_{N,u}$ and 
\[
\lim_{N\to\infty}t_{N}=0,\,\,\,\mbox{with}\,\, t_{N}=\sup_{u\in J_{N}}\sup_{\mathbf{A}\in\mathcal{A}_{N,u}^{\left(1\right)}}\frac{Q_{2}}{\left(\mathbb{E}_{u,0}\left\{ Z_{N,\beta}\left({\rm Band}\left(\epsilon_{N}\right)\right)\right\} \right)^{2}}=0.
\]
By Lemma \ref{lem:BaikLee}, since 
\[
Q_{3}+Q_{4}=\left(\mathbb{E}_{u,0,\mathbf{A}}\left\{ Z_{N,\beta}\left({\rm Band}\left(\epsilon_{N}\right)\right)\right\} \right)^{2},
\]
we have that 
\[
\lim_{N\to\infty}\inf_{u\in\mathbb{R}}\mathbb{P}\left\{ \frac{\mathbb{E}_{u,0,\mathbf{\mathbf{G}\left(\hat{\mathbf{n}}\right)}}\left\{ Z_{N,\beta}\left({\rm Band}\left(\epsilon_{N}\right)\right)\right\} }{\mathbb{E}_{u,0}\left\{ Z_{N,\beta}\left({\rm Band}\left(\epsilon_{N}\right)\right)\right\} }\geq t_{N}^{1/4}\right\} =1.
\]
In other words, there exist sets $\mathcal{A}_{N,u}^{\left(2\right)}$
such that (\ref{eq:53}) holds with $\mathcal{A}_{N,u}^{\left(2\right)}$
instead of $\mathcal{A}_{N,u}$ and 
\[
\inf_{u\in\mathbb{R}}\inf_{\mathbf{A}\in\mathcal{A}_{N,u}^{\left(2\right)}}\frac{Q_{3}+Q_{4}}{\left(\mathbb{E}_{u,0}\left\{ Z_{N,\beta}\left({\rm Band}\left(\epsilon_{N}\right)\right)\right\} \right)^{2}}\geq t_{N}^{1/2}.
\]
Setting $\mathcal{A}_{N,u}=\mathcal{A}_{N,u}^{\left(1\right)}\cap\mathcal{A}_{N,u}^{\left(2\right)}$
 we have that (\ref{eq:53}) and (\ref{eq:55}) hold.

Lastly, note that for any real, symmetric $N-1\times N-1$ matrix
$\mathbf{A}$, with all suprema taken over $(\boldsymbol{\sigma}_{1},\boldsymbol{\sigma}_{2})\in T_{\hat{\mathbf{n}}}\left({\rm Band}\left(\epsilon_{N}\right);\left(-\rho_{N},\rho_{N}\right)\right)$,
\begin{align*}
\frac{Q_{1}}{Q_{3}} & \leq\sup\frac{\mathbb{E}_{u,0,\mathbf{A}}\left\{ \prod_{i=1,2}e^{-\beta H_{N}\left(\boldsymbol{\sigma}_{i}\right)}\right\} }{\prod_{i=1,2}\mathbb{E}_{u,0,\mathbf{A}}\left\{ e^{-\beta H_{N}\left(\boldsymbol{\sigma}_{i}\right)}\right\} }\\
 & =\sup\frac{\mathbb{E}_{u,0,\mathbf{A}}\left\{ \prod_{i=1,2}e^{-\beta\bar{H}_{N}^{\hat{\mathbf{n}},3+}\left(\boldsymbol{\sigma}_{i}\right)}\right\} }{\prod_{i=1,2}\mathbb{E}_{u,0,\mathbf{A}}\left\{ e^{-\beta\bar{H}_{N}^{\hat{\mathbf{n}},3+}\left(\boldsymbol{\sigma}_{i}\right)}\right\} }\\
 & =\sup\exp\left\{ \beta^{2}{\rm Cov}_{N}^{\hat{\mathbf{n}},3+}\left(\boldsymbol{\sigma}_{1},\boldsymbol{\sigma}_{2}\right)\right\} \leq\exp\left\{ N\beta^{2}\sum_{k=3}^{p}\binom{p}{k}\rho_{N}^{k}\right\} ,
\end{align*}
where the inequality in the first line follows since $\int_{T}f(t)d\mu/\int_{T}g(t)d\mu\leq\sup_{t\in T}f(t)/g(t)$
for positive functions, the equality in the second line follows from
Lemmas \ref{lem:HNdecomposition} and \ref{lem:Hhat_expressions},
the equality in the third line follows since the law of $\bar{H}_{N}^{\hat{\mathbf{n}},3+}$
under $\mathbb{P}_{u,0,\mathbf{A}}$ is the same as its unconditional
law, and the inequality in the third line follows from
\begin{align}
{\rm Cov}_{N}^{\hat{\mathbf{n}},3+}\left(\boldsymbol{\sigma}_{1},\boldsymbol{\sigma}_{2}\right) & \triangleq\mathbb{E}\left\{ \bar{H}_{N}^{\hat{\mathbf{n}},3+}\left(\boldsymbol{\sigma}_{1}\right)\bar{H}_{N}^{\hat{\mathbf{n}},3+}\left(\boldsymbol{\sigma}_{2}\right)\right\} \label{eq:Knhat}\\
 & =N\sum_{k=3}^{p}\binom{p}{k}\left(\sqrt{1-q_{1}^{2}}\sqrt{1-q_{2}^{2}}R_{\hat{\mathbf{n}}}\left(\boldsymbol{\sigma}_{1},\boldsymbol{\sigma}_{2}\right)\right)^{k}q_{1}^{p-k}q_{2}^{p-k},\nonumber 
\end{align}
which follows, with $q_{i}=R(\boldsymbol{\sigma}_{i},\hat{\mathbf{n}})$,
from (\ref{eq:1220-2}) and (\ref{eq:Hbar decomposition}). If we
choose $\rho_{N}$ such that $\rho_{N}N^{1/3}\to0$ as $N\to\infty$,
then (\ref{eq:56-1}) holds. This completes the proof.
\end{proof}

\section{\label{sec:range(q***,q**)}mass of bands under $\mathbb{P}_{u,0}$:
the range $(q_{{\rm LS}},q_{**})$}

As pointed out in Section \ref{sec:outline}, in order to bound contributions
to the partition function related to overlap range $\left(q_{{\rm LS}},q_{**}\right)$
we derive upper bounds for the corresponding free energy of bands.
Throughout the section we shall use 
\begin{equation}
{\rm Band}={\rm Band}\left(\hat{\mathbf{n}},q_{1},q_{2}\right)\label{eq:1229-1}
\end{equation}
as an abbreviation for a band with general overlaps. Define%
\footnote{The fact that the limit in (\ref{eq:lambda_F}) exists follows from
the spherical Parisi formula \cite{Talag,Chen}. We could avoid relying
on the formula by simply replacing the limit with the supremum limit
and modify the next result appropriately.%
} 
\begin{equation}
\FrfF\left(E,q\right)=\frac{1}{2}\log\left(1-q^{2}\right)+\beta Eq^{p}+\lim_{N\to\infty}\frac{1}{N}\mathbb{E}\left\{ \log\left(Z_{N-1,\beta}\left(\bar{H}_{N}^{\hat{\mathbf{n}},2+}|_{q}\right)\right)\right\} ,\label{eq:lambda_F}
\end{equation}
where $\bar{H}_{N}^{\hat{\mathbf{n}},2+}|_{q}$, defined by (\ref{eq:3+})
and (\ref{eq:22}), is a mixed spherical models on $\mathbb{S}^{N-2}$,
and where $Z_{N,\beta}(f)$ is given by (\ref{eq:Z(f)}). From a standard
concentration argument we have the following.
\begin{lem}
\label{lem:Lambda_F_bound}Let $E>0$ be some positive number and
set $J_{N}=\left[-EN,EN\right]$. Then, for any $\beta$, 
\begin{equation}
\limsup_{N\to\infty}\sup_{u\in J_{N}}\left|\frac{1}{N}\mathbb{E}_{u,0}\left\{ \log\left(Z_{N,\beta}\left({\rm Band}\right)\right)\right\} -\sup_{q\in\left(q_{1},q_{2}\right)}\FrfF\left(-\frac{u}{N},q\right)\right|\leq0.\label{eq:15031}
\end{equation}

\end{lem}
Define 
\begin{align}
\FrfFt\left(E,q\right) & =\frac{1}{2}\log\left(1-q^{2}\right)+\beta Eq^{p}\label{eq:LambdaF2}\\
 & +\frac{1}{N}\log\mathbb{E}\left\{ Z_{N-1,\beta}\left(\bar{H}_{N}^{\hat{\mathbf{n}},3+}|_{q}\right)\right\} +\lim_{N\to\infty}\frac{1}{N}\mathbb{E}\left\{ \log\left(Z_{N-1,\beta}\left(\bar{H}_{N}^{\hat{\mathbf{n}},2}|_{q}\right)\right)\right\} .\nonumber 
\end{align}
(Where the term involving $\bar{H}_{N}^{\hat{\mathbf{n}},3+}$ above
does not depend on $N$.)
\begin{lem}
\label{lem:Lambda_F2}For any $\beta$, $E$ and $q$, with $\mathscr{P}_{2}\left(\beta\right)$
defined by (\ref{eq:P2}), 
\begin{equation}
\FrfF\left(E,q\right)\leq\FrfFt\left(E,q\right)=\frac{1}{2}\log\left(1-q^{2}\right)+\beta Eq^{p}+\mathscr{P}_{2}\left(\left|\alpha_{2}\left(q\right)\right|\beta\right)+\frac{1}{2}\beta^{2}\sum_{k=3}^{p}\alpha_{k}^{2}\left(q\right).\label{eq:a7}
\end{equation}

\end{lem}
We also derive bounds on the fluctuations.
\begin{lem}
\label{lem:devFE-RegIV}For any $\beta$, $u$ and $x>0$, 
\begin{equation}
\mathbb{P}_{u,0}\left\{ \left|\log\left(Z_{N,\beta}\left({\rm Band}\right)\right)-\mathbb{E}_{u,0}\log\left(Z_{N,\beta}\left({\rm Band}\right)\right)\right|>Nx\right\} \leq e^{-\frac{Nx^{2}}{2V\beta^{2}}},\label{eq:1226-5}
\end{equation}
where, with $\alpha_{k}\left(q\right)$ defined in (\ref{eq:alpha_k}),
\[
V:=V\left(q_{1},q_{2}\right)=\sum_{k=2}^{p}\sup_{q\in\left(q_{1},q_{2}\right)}\left(\alpha_{k}\left(q\right)\right)^{2}.
\]

\end{lem}
The rest of the section is taken up with the proofs of the three lemmas.

\subsection*{Proof of Lemma \ref{lem:Lambda_F_bound}}

Using the fact that as $|q|\to1$ , $\FrfF\left(\frac{u}{N},q\right)\to-\infty$
uniformly in $u\in J_{N}$, we may, as we will, make the assumption
that $-1<q_{1},q_{2}<1$. By partitioning $(q_{1},q_{2})$ to finitely
many intervals and letting the maximal length of an interval approach
$0$, it is enough to show that the left-hand side of (\ref{eq:15031})
is smaller than $c\left(q_{2}-q_{1}\right)$ for some $c>0$ independent
of $N$. 

By Corollary \ref{cor:conditional laws}, similarly to (\ref{eq:16}),
under $\mathbb{P}_{u,0}$, $Z_{N,\beta}\left({\rm Band}\right)$ has
the same distribution as
\[
\int_{q_{1}}^{q_{2}}\frac{\omega_{N-1}}{\omega_{N}}\left(1-q^{2}\right)^{\frac{N-3}{2}}e^{-\beta uq^{p}}Z_{N-1,\beta}\left(\bar{H}_{N}^{\hat{\mathbf{n}},2+}|_{q}\right)dq.
\]
Therefore, the proof is concluded if we show that 
\begin{equation}
X_{0}\triangleq\sup_{q\in\left(q_{1},q_{2}\right)}\left|\frac{1}{N}\log\left(Z_{\beta}\left(\bar{H}_{N}^{\hat{\mathbf{n}},2+}|_{q}\right)\right)-\frac{1}{N}\mathbb{E}\left\{ \log\left(Z_{\beta}\left(\bar{H}_{N}^{\hat{\mathbf{n}},2+}|_{q_{1}}\right)\right)\right\} \right|\label{eq:a}
\end{equation}
satisfies 
\begin{equation}
\limsup_{N\to\infty}\frac{1}{N}\log\left(\mathbb{P}\left\{ |X_{0}|\geq t+\left(q_{2}-q_{1}\right)c_{1}\right\} \right)\leq-c_{2}t^{2},\label{eq:a1}
\end{equation}
for some $c_{1},\, c_{2}>0$.

Using (\ref{eq:1220-2}) one has (for example, by differentiating
by $q$ the first logarithm in (\ref{eq:a})) that (\ref{eq:a1})
follows if we prove that for any $2\leq k\leq p$ and large enough
$N$, 
\begin{equation}
\mathbb{P}\left\{ \max_{\boldsymbol{\sigma}\in\mathbb{S}^{N-2}}\left|H_{N-1}^{{\rm pure\,}k}\left(\boldsymbol{\sigma}\right)\right|\geq tN+c_{1}^{\left(k\right)}N\right\} \leq e^{-c_{2}^{\left(k\right)}Nt^{2}}\label{eq:maxboundPspin}
\end{equation}
for appropriate constants $c_{1}^{\left(k\right)},\, c_{2}^{\left(k\right)}>0$.
The bound of (\ref{eq:maxboundPspin}) follows by Corollary \ref{cor:min}
for $k\geq3$ and the connection to GOE matrices as in (\ref{eq:HG})
for $k=2$ and the Borell-TIS inequality \cite[Theorem 2.1.1]{RFG}.\qed

\subsection*{Proof of Lemma \ref{lem:Lambda_F2}}

Let $\mathbb{E}_{\mathbf{G}\left(\hat{\mathbf{n}}\right)}\left\{ \cdot\right\} $
denote the conditional expectation given $\mathbf{G}\left(\hat{\mathbf{n}}\right)$
(defined in (\ref{eq:G})). By (\ref{eq:Hnhat2}) and the independence
of $\bar{H}_{N}^{\hat{\mathbf{n}},k}$ for different $k$,
\[
\mathbb{E}_{\mathbf{G}\left(\hat{\mathbf{n}}\right)}\left\{ Z_{N-1,\beta}\left(\bar{H}_{N}^{\hat{\mathbf{n}},2+}|_{q}\right)\right\} =\mathbb{E}\left\{ e^{\beta\bar{H}_{N}^{\hat{\mathbf{n}},3+}\left(\boldsymbol{\sigma}_{q}\right)}\right\} \cdot Z_{N-1,\beta}\left(\bar{H}_{N}^{\hat{\mathbf{n}},2}|_{q}\right),
\]
where $\boldsymbol{\sigma}_{q}$ is an arbitrary point such that $R\left(\boldsymbol{\sigma}_{q},\hat{\mathbf{n}}\right)=q$.
Thus, by Jensen's inequality and (\ref{eq:1220-2}), 
\begin{align}
 & \mathbb{E}\left\{ \log\left(Z_{N-1,\beta}\left(\bar{H}_{N}^{\hat{\mathbf{n}},2+}|_{q}\right)\right)\right\} \leq\mathbb{E}\left\{ \log\left(\mathbb{E}_{\mathbf{G}\left(\hat{\mathbf{n}}\right)}\left\{ Z_{N-1,\beta}\left(\bar{H}_{N}^{\hat{\mathbf{n}},2+}|_{q}\right)\right\} \right)\right\} \label{eq:0101-1}\\
 & =\frac{1}{2}N\beta^{2}\sum_{k=3}^{p}\left(\alpha_{k}\left(q\right)\right)^{2}+\mathbb{E}\left\{ \log\left(Z_{N-1,\beta}\left(\bar{H}_{N}^{\hat{\mathbf{n}},2}|_{q}\right)\right)\right\} \nonumber \\
 & =\log\left(\mathbb{E}\left\{ Z_{N-1,\beta}\left(\bar{H}_{N}^{\hat{\mathbf{n}},3+}|_{q}\right)\right\} \right)+\mathbb{E}\left\{ \log\left(Z_{N-1,\beta}\left(\bar{H}_{N}^{\hat{\mathbf{n}},2}|_{q}\right)\right)\right\} .\nonumber 
\end{align}

By (\ref{eq:1220-2}), $Z_{N-1,\beta}\left(\bar{H}_{N}^{\hat{\mathbf{n}},2}|_{q}\right)$
is equal to the partition function of $H_{N-1}^{{\rm pure\,}2}\left(\boldsymbol{\sigma}\right)$
at temperature $\beta\left|\alpha_{2}\left(q\right)\right|\cdot\sqrt{\frac{N}{N-1}}$.
Thus, by Corollary \ref{cor:2spinFE} (and using the monotonicity
in temperature of the partition function) we have that
\[
\lim_{N\to\infty}\frac{1}{N}\mathbb{E}\left\{ \log\left(Z_{N-1,\beta}\left(\bar{H}_{N}^{\hat{\mathbf{n}},2}|_{q}\right)\right)\right\} =\mathscr{P}_{2}\left(\beta\right).
\]
This completes the proof.\qed

\subsection*{Proof of Lemma \ref{lem:devFE-RegIV}}

We will show that $\frac{1}{N}\log Z_{N,\beta}\left({\rm Band}\right)$
is a Lipschitz function of a set of i.i.d standard Gaussian variables
with Lipschitz constant bounded from above by $\beta\sqrt{V/N}$.
The required bound will follow by standard concentration inequalities
(see e.g. \cite[Lemma 2.3.3]{Matrices}). By Corollary \ref{cor:conditional laws},
(\ref{eq:1220-2}) and (\ref{eq:Hamiltonian}) we have the following.
The law of the field $\{H_{N}(\boldsymbol{\sigma})\}_{\boldsymbol{\sigma}}$
under $\mathbb{P}_{u,0}$ is identical to that of $\{h\left(\boldsymbol{\sigma},u,\mathbf{J}\right)\}_{\boldsymbol{\sigma}}$
where $\mathbf{J}=(J_{i_{1},...,i_{k}})$ is an array of random variables
with $1\leq i_{j}\leq N$, $2\leq k\leq p$, whose elements are i.i.d
standard Gaussian variables, and where the (deterministic) function
$h$ is given by 
\begin{equation}
h\left(\boldsymbol{\sigma},u,\mathbf{x}\right)=u\left(R\left(\boldsymbol{\sigma},\hat{\mathbf{n}}\right)\right)^{p}+\alpha_{k}\left(q\left(\boldsymbol{\sigma}\right)\right)\sum_{k=2}^{p}\frac{N^{1/2}}{\left(N-1\right)^{k/2}}\sum_{i_{1},...,i_{k}=1}^{N}x_{i_{1},...,i_{k}}\tilde{\sigma}_{i_{1}}\cdots\tilde{\sigma}_{i_{k}},\label{eq:h}
\end{equation}
where $\mathbf{x}=(x_{i_{1},...,i_{k}})$ is an array of real numbers
as above, and $\tilde{\boldsymbol{\sigma}}=\left(\tilde{\sigma}_{1},...,\tilde{\sigma}_{N-1}\right)$
is the vector of norm $\sqrt{N-1}$ defined in (\ref{eq:69}).

For any $i_{1},...,i_{k}$,
\begin{align*}
D_{i_{1},...,i_{k}} & \triangleq\frac{d}{dx_{i_{1},...,i_{k}}}\log\left(\int_{{\rm Band}}\exp\left\{ -\beta h\left(\boldsymbol{\sigma},u,\mathbf{x}\right)\right\} d\mu_{N}\left(\boldsymbol{\sigma}\right)\right)\\
 & =-\beta\frac{N^{1/2}}{\left(N-1\right)^{k/2}}\cdot\frac{\int_{{\rm Band}}\alpha_{k}\left(q\left(\boldsymbol{\sigma}\right)\right)\tilde{\sigma}_{i_{1}}\cdots\tilde{\sigma}_{i_{k}}\exp\left\{ -\beta h\left(\boldsymbol{\sigma},u,\mathbf{x}\right)\right\} d\mu_{N}\left(\boldsymbol{\sigma}\right)}{\int_{{\rm Band}}\exp\left\{ -\beta h\left(\boldsymbol{\sigma},u,\mathbf{x}\right)\right\} d\mu_{N}\left(\boldsymbol{\sigma}\right)}.
\end{align*}
The ratio of integrals in the last equation can be viewed as an expectation
under a Gibbs measure on the band which corresponds to (\ref{eq:h}).
Denote expectation by this measure by $\left\langle \,\cdot\,\right\rangle _{h}$,
so that the ratio is simply $\langle\alpha_{k}\left(q\left(\boldsymbol{\sigma}\right)\right)\tilde{\sigma}_{i_{1}}\cdots\tilde{\sigma}_{i_{k}}\rangle_{h}$.
We then have 
\begin{align}
\sum_{k=2}^{p}\sum_{i_{1},...,i_{k}=1}^{N}\left(D_{i_{1},...,i_{k}}\right)^{2} & =\sum_{k=2}^{p}\beta^{2}\frac{N}{\left(N-1\right)^{k}}\sum_{i_{1},...,i_{k}=1}^{N}\langle\alpha_{k}\left(q\left(\boldsymbol{\sigma}\right)\right)\tilde{\sigma}_{i_{1}}\cdots\tilde{\sigma}_{i_{k}}\rangle_{h}^{2}\label{eq:0304}\\
 & \leq\sum_{k=2}^{p}\beta^{2}\frac{N}{\left(N-1\right)^{k}}\sup_{q\in(q_{1},q_{2})}\left(\alpha_{k}\left(q\right)\right)^{2}\sum_{i_{1},...,i_{k}=1}^{N}\left\langle (\tilde{\sigma}_{i_{1}}\cdots\tilde{\sigma}_{i_{k}})^{2}\right\rangle _{h}.\nonumber 
\end{align}
Note that
\[
\sum_{i_{1},...,i_{k}=1}^{N}\left\langle (\tilde{\sigma}_{i_{1}}\cdots\tilde{\sigma}_{i_{k}})^{2}\right\rangle _{h}=\left\langle \left\Vert \tilde{\boldsymbol{\sigma}}\right\Vert _{2}^{2k}\right\rangle _{h}=\left(N-1\right)^{k}.
\]
Since the Lipschitz constant mentioned in the beginning is equal to
$N^{-1}$ times the square root of (\ref{eq:0304}), the proof is
completed.\qed

\section{\label{sec:Bounds-on-contributions}Bounds on contributions to $Z_{N,\beta}$}

In this section we derive the bounds on contributions required for
the proof of Theorem \ref{thm:Geometry}. In Section \ref{sub:Restricting-to-caps}
we make precise the argument about restriction to caps outlined in
Section \ref{sec:outline}. In Section \ref{sub:Overlap-depth-regions}
we define various overlap-depth regions and state the bounds we shall
need for each. We prove them in Section \ref{sub:pfbds} using the
results on the conditional law of masses derived in Sections \ref{sec:range(q**,1)}
and \ref{sec:range(q***,q**)} and corollaries of the Kac-Rice formula
from Appendix I.

\subsection{\label{sub:Restricting-to-caps}Restriction to caps}

In (\ref{eq:Reg*bd}) below we shall prove that for any $\delta>0$
, w.h.p $Z_{N,\beta}\geq\exp\left\{ N\left(\Frf\left(E_{0},q_{*}\right)-\delta\right)\right\} $.
Observe that 
\begin{equation}
Z_{N,\beta}\left(\left\{ \boldsymbol{\sigma}:H_{N}\left(\boldsymbol{\sigma}\right)\geq u\right\} \right)\leq\mu_{N}\left(\left\{ \boldsymbol{\sigma}:H_{N}\left(\boldsymbol{\sigma}\right)\geq u\right\} \right)e^{-\beta u}\leq e^{-\beta u}.\label{eq:LSbd}
\end{equation}
Hence, setting 
\begin{equation}
u_{{\rm LS}}:=u_{{\rm LS}}\left(\beta\right)=-\Frf\left(E_{0},q_{*}\right)N/\beta,\label{eq:uLS}
\end{equation}
for any $\delta>0$ 
\begin{equation}
Z_{N,\beta}\left(\left\{ \boldsymbol{\sigma}:H_{N}\left(\boldsymbol{\sigma}\right)\geq u_{{\rm LS}}+\delta N\right\} \right)/Z_{N,\beta}\overset{N\to\infty}{\longrightarrow}0.\label{eq:lvlset}
\end{equation}
The (random) set in (\ref{eq:lvlset}) is related to critical points
by the following lemma. Recall that $q_{{\rm LS}}=1-C_{{\rm LS}}\frac{\log\beta}{\beta}$
was defined in (\ref{eq:qLS}) with an arbitrary $C_{{\rm LS}}>\left(2p\left(E_{0}-E_{\infty}\right)\right)^{-1}$.
\begin{lem}
\label{lem:LSnegligible}For large enough $\beta$, for small enough
$\delta:=\delta(\beta)>0$ , 
\begin{equation}
\lim_{N\to\infty}\mathbb{P}\left\{ \left\{ \boldsymbol{\sigma}:H_{N}\left(\boldsymbol{\sigma}\right)<u_{{\rm LS}}+\delta N\right\} \subset\cup_{\boldsymbol{\sigma}_{0}\in\mathscr{C}_{N}\left(-\infty,u_{{\rm LS}}+\delta N\right)}{\rm Cap}_{N}\left(\boldsymbol{\sigma}_{0},q_{{\rm LS}}\right)\right\} =1.\label{eq:LSandCaps}
\end{equation}
\end{lem}
\begin{proof}
Any connected component of $\left\{ \boldsymbol{\sigma}:H_{N}\left(\boldsymbol{\sigma}\right)<u_{{\rm LS}}+\delta N\right\} $
contains at least one critical point $\boldsymbol{\sigma}_{0}\in\mathscr{C}_{N}\left(-\infty,u_{{\rm LS}}+\delta N\right)$
(one local minimum point, in particular). Thus, it will be enough
to show that w.h.p for any critical $\boldsymbol{\sigma}_{0}\in\mathscr{C}_{N}\left(-\infty,u_{{\rm LS}}+\delta N\right)$
the connected component of $\boldsymbol{\sigma}_{0}$ is contained
in ${\rm Cap}_{N}\left(\boldsymbol{\sigma}_{0},q_{{\rm LS}}\right)$.
This will follow if we show that 
\begin{equation}
\limsup_{N\to\infty}\frac{1}{N}\log\mathbb{E}\left|\left\{ \boldsymbol{\sigma}_{0}\in\mathscr{C}_{N}\left(-E_{0}N,u_{{\rm LS}}+\delta N\right):\,\inf_{\boldsymbol{\sigma}:R(\boldsymbol{\sigma},\boldsymbol{\sigma}_{0})=q_{{\rm LS}}}H_{N}\left(\boldsymbol{\sigma}\right)<u_{{\rm LS}}+\delta N\right\} \right|<0,\label{eq:08038}
\end{equation}
since by Corollary \ref{cor:min}, $\mathbb{P}\left\{ \mathscr{C}_{N}\left(-\infty,-E_{0}N\right)=\emptyset\right\} \to1$
as $N\to\infty$. We recall that $-\Frf\left(E_{0},q_{*}\right)/\beta=u_{{\rm LS}}/N\to-E_{0}$
as $\beta\to\infty$, $\Theta_{p}\left(x\right)$ is continuous, and
$\Theta_{p}\left(-E_{0}\right)=0$. Also, 
\[
\mathbb{P}_{u,0}\left\{ \inf_{\boldsymbol{\sigma}:R(\boldsymbol{\sigma},\hat{\mathbf{n}})=q_{{\rm LS}}}H_{N}\left(\boldsymbol{\sigma}\right)\leq u_{{\rm LS}}+\delta N\right\} =\mathbb{P}_{u,0}\left\{ \inf_{\boldsymbol{\sigma}\in\mathbb{S}^{N-2}}H_{N}|_{q_{{\rm LS}}}\left(\boldsymbol{\sigma}\right)\leq u_{{\rm LS}}+\delta N\right\} 
\]
is decreasing with $u$. By Lemma \ref{lem:15}, in order to finish
the proof it will be enough to show that
\begin{equation}
\limsup_{\beta\to\infty}\limsup_{N\to\infty}\frac{1}{N}\log\mathbb{P}_{-E_{0}N,0}\left\{ \inf_{\boldsymbol{\sigma}\in\mathbb{S}^{N-2}}H_{N}|_{q_{{\rm LS}}}\left(\boldsymbol{\sigma}\right)\leq u_{{\rm LS}}+\delta N\right\} <0.\label{eq:0101-9}
\end{equation}

By Corollary \ref{cor:conditional laws}, Lemma \ref{lem:HNdecomposition}
and (\ref{eq:Hnhat2}), the probability in (\ref{eq:0101-9}) is bounded
from above by
\[
\mathbb{P}\left\{ \frac{1}{\sqrt{2}}\alpha_{2}\left(q_{{\rm LS}}\right)N\lambda_{N}+\sum_{k=3}^{p}\alpha_{k}\left(q_{{\rm LS}}\right)\sqrt{\frac{N}{N-1}}\inf_{\boldsymbol{\sigma}}H_{N-1}^{{\rm pure\,}k}\left(\boldsymbol{\sigma}\right)\leq u_{{\rm LS}}+\delta N+E_{0}\left(q_{{\rm LS}}\right)^{p}N\right\} ,
\]
where $\lambda_{N}$ is the minimal eigenvalue of an $N-1$ dimensional
GOE matrix and $\lambda_{N}$ and $\left\{ H_{N-1}^{{\rm pure\,}k}\left(\boldsymbol{\sigma}\right)\right\} _{\boldsymbol{\sigma}}$,
$k\geq3$, are independent of each other.

From Theorem \ref{thm:ABAC} and since $\alpha_{k}\left(q_{{\rm LS}}\right)=O\left(\left(\log\beta/\beta\right)^{k/2}\right)$,
for any $t>0$, 
\begin{equation}
\limsup_{\beta\to\infty}\limsup_{N\to\infty}\frac{1}{N}\log\mathbb{P}\left\{ \sum_{k=3}^{p}\alpha_{k}\left(q_{{\rm LS}}\right)\sqrt{\frac{N}{N-1}}\inf_{\boldsymbol{\sigma}}H_{N-1}^{{\rm pure\,}k}\left(\boldsymbol{\sigma}\right)\leq t\frac{\log\beta}{\beta}N\right\} <0.\label{eq:0101-10}
\end{equation}
From some calculus and the definitions (\ref{eq:39}), (\ref{eq:7}),
(\ref{eq:qLS}) and (\ref{eq:alpha_k}), one has that $q_{*}=1-\frac{\chi_{1}}{\sqrt{2p(p-1)}}+O\left(\beta^{-2}\right)$
and 
\begin{align*}
\lim_{\beta\to\infty}\frac{\beta}{\log\beta}\frac{1}{N}\left(u_{{\rm LS}}+E_{0}q_{{\rm LS}}^{p}N\right) & =\frac{1}{2}-E_{0}pC_{{\rm LS}},\\
\lim_{\beta\to\infty}\frac{\beta}{\log\beta}\frac{1}{\sqrt{2}}\alpha_{2}\left(q_{{\rm LS}}\right) & =\frac{1}{2}pE_{\infty}C_{{\rm LS}}.
\end{align*}
Therefore, based on (\ref{eq:0101-10}), the left-hand side of (\ref{eq:0101-9})
is bounded from above by 
\[
\limsup_{\beta\to\infty}\limsup_{N\to\infty}\frac{1}{N}\log\left\{ \frac{1}{2}pE_{\infty}C_{{\rm LS}}\lambda_{N}\leq\frac{1}{2}-E_{0}pC_{{\rm LS}}+\delta/2\right\} .
\]
Since for any $\epsilon>0$, $\limsup_{N\to\infty}\frac{1}{N}\log\mathbb{P}\left\{ \lambda_{N}\leq-2-\epsilon\right\} <0$
(see e.g.\cite[Lemma 6.3]{BDG}) and since we assumed that $C_{{\rm LS}}>\left(2p\left(E_{0}-E_{\infty}\right)\right)^{-1}$,
the proof is completed.
\end{proof}

\subsection{\label{sub:Overlap-depth-regions}Overlap-depth regions and bounds}

For the convenience of the reader, we recall that
\begin{align*}
{\rm Cont}_{N,\beta}\left(A\times B\right) & =Z_{N,\beta}\left(\cup_{\boldsymbol{\sigma}_{0}\in\mathscr{C}_{N}\left(B\right)}\left\{ \boldsymbol{\sigma}:R\left(\boldsymbol{\sigma},\boldsymbol{\sigma}_{0}\right)\in A\right\} \right),\\
{\rm Reg}_{*}\left(c,\kappa,\kappa'\right) & =\left(q_{*}-cN^{-1/2},q_{*}+cN^{-1/2}\right)\times\left(m_{N}-\kappa',m_{N}+\kappa\right).
\end{align*}
Let $\tau_{N}>0$ be a sequence such that $\tau_{N}\to0$ as $N\to\infty$,
required to satisfy a certain relation which will be specified shortly
(see Remark \ref{rem:tauN}). In order to bound the contribution of
${\rm Reg}_{{\rm UB}}\left(\delta,c,\kappa,\kappa'\right)$ (see \ref{eq:RegUB})
we define the following regions, 
\begin{align*}
{\rm Reg}_{{\rm I}}\left(c,\kappa'\right) & =\left(\left(q_{**},1\right)\setminus\left(q_{*}-cN^{-1/2},\, q_{*}+cN^{-1/2}\right)\right)\times\left(m_{N}-\kappa',\,-E_{0}N+\tau_{N}N\right),\\
{\rm Reg}_{{\rm II}}\left(c,\kappa\right) & =\left(q_{*}-cN^{-1/2},\, q_{*}+cN^{-1/2}\right)\times\left(m_{N}+\kappa,\,-E_{0}N+\tau_{N}N\right),\\
{\rm Reg}_{{\rm III}}\left(\tau,\delta\right) & =\left(q_{**},1\right)\times\left(-E_{0}N+\tau N,\, u_{{\rm LS}}+\delta N\right),\\
{\rm Reg}_{{\rm IV}}\left(\delta\right) & =\left(q_{{\rm LS}},q_{**}\right)\times\left(m_{N}-\kappa',\, u_{{\rm LS}}+\delta N\right).
\end{align*}

\begin{figure}[h]
~\\
~\\
~\\
%\begin{overpic}[grid,tics=10,width=0.65\textwidth]{Reg2png}
%\hspace*{-1.7cm}
\begin{overpic}[width=0.65\textwidth]{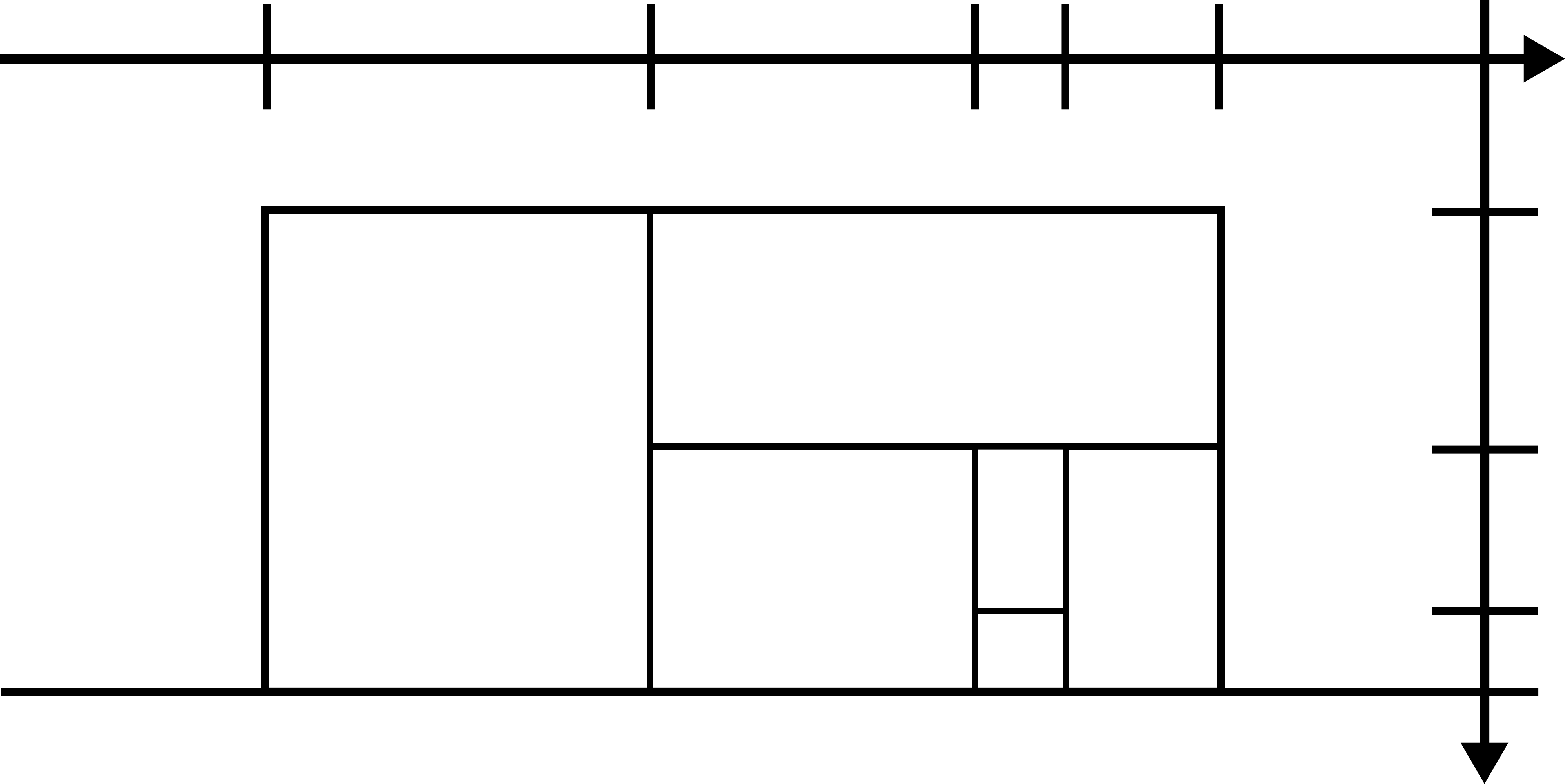}
\put (77,52) {$1$} 
\put (40,52) {$q_{**}$} 
\put (58,52) {$q_*\pm cN^{-\frac{1}{2}}$}
\put (16,52) {$q_{\scriptscriptstyle\rm{LS}}$}
\put (102,46) {Overlap} 
\put (100,36) {$u_{\scriptscriptstyle\rm{LS}}+\delta N$} 
\put (100,10.5) {$m_N+\kappa$} 
\put (100,5.5) {$m_N-\kappa'$} 
\put (100,20.7) {$-E_0 N+\tau_N N$}
\put (26,19) {\huge IV} 
\put (56.5,26.8) {\huge III}
\put (63,14.2) {\huge II}
\put (51.4,11.3) {\huge I}
\put (71.8,11.3) {\huge I}
\put (63.6,5.35) {\huge *}
%\put (103,34) {$\mbox{Reg}_*(c,\kappa,\kappa')$} 
%\put (103,26.65) {$(\rm{Reg}_{\scriptscriptstyle\rm{LS}}(\delta))^c$} 
%\put (103,19.30) {$\rm{Reg}_{\scriptscriptstyle\rm{GS}}(\kappa')$} 
\put (102,0.5) {Depth} 
%\put (23,25) {$\mbox{Reg}_{\scriptscriptstyle\rm{UB}}(\delta,c,\kappa,\kappa')$}  
\end{overpic}~\\
~\\
\protect\caption{\label{fig:8.1}Regions diagram. The letters correspond to subscripts
in the definitions of the regions. (for ${\rm Reg}_{{\rm III}}\left(\tau,\delta\right)$,
above we take $\tau=\tau_{N}$)}
\end{figure}
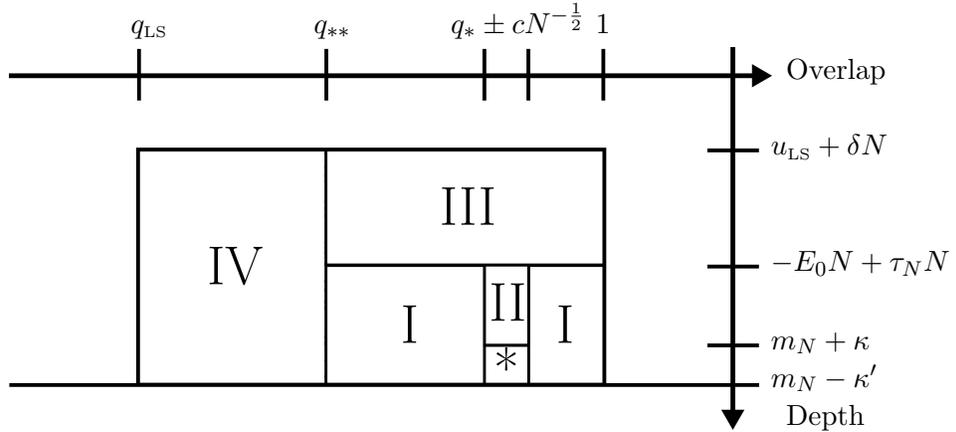

We now state the bounds we need on the regions above. Let $\ftail:\left(0,\infty\right)\to\left(0,\infty\right)$
be a function, which will be fixed henceforth, such that (with $c_{p}$
defined in (\ref{eq:c_p})) 
\begin{equation}
\lim_{\kappa\to\infty}\left(\mathbb{P}\left\{ e^{Y_{*}}\leq\ftail\left(\kappa\right)\right\} \right)^{1/2}\int_{0}^{\kappa}dv\cdot e^{c_{p}v}=0.\label{eq:ftail}
\end{equation}
We remind the reader that the definitions of $\Frf\left(E,q\right)$
and $\mathfrak{V}_{N,\beta}\left(u\right)$ are given in (\ref{eq:6})
and (\ref{eq:V(u)}) respectively.
\begin{prop}
\label{prop:bounds}For large enough $\beta$, for any positive $c,\,\tau,\,\kappa'$
and small enough $\delta$, 
\begin{align}
\lim_{\kappa,c\to\infty}\liminf_{N\to\infty}\mathbb{P}\left\{ {\rm Cont}_{N,\beta}\left({\rm Reg}_{*}\left(c,\kappa,\kappa'\right)\right)\geq\ftail\left(\kappa\right)\mathfrak{V}_{N,\beta}\left(m_{N}+\kappa\right)\right\} =1,\label{eq:Reg*bd}\\
\lim_{c\to\infty}\limsup_{N\to\infty}\mathbb{E}\left\{ {\rm Cont}_{N,\beta}\left({\rm Reg}_{{\rm I}}\left(c,\kappa'\right)\right)\right\} /\mathfrak{V}_{N,\beta}\left(m_{N}\right)=0,\label{eq:RegIbd}\\
\lim_{\kappa\to\infty}\limsup_{N\to\infty}\mathbb{E}\left\{ {\rm Cont}_{N,\beta}\left({\rm Reg}_{{\rm II}}\left(c,\kappa\right)\right)\right\} /\mathfrak{V}_{N,\beta}\left(m_{N}\right)=0,\label{eq:RegIIbd}\\
\limsup_{N\to\infty}\frac{1}{N}\log\left(\mathbb{E}\left\{ {\rm Cont}_{N,\beta}\left({\rm Reg}_{{\rm III}}\left(\tau,\delta\right)\right)\right\} \right)<\Frf\left(E_{0},q_{*}\right),\label{eq:RegIIIbd}\\
\limsup_{N\to\infty}\frac{1}{N}\log\left(\mathbb{P}\left\{ \frac{1}{N}\log\left({\rm Cont}_{N,\beta}\left({\rm Reg}_{{\rm IV}}\left(\delta\right)\right)\right)>\Frf\left(E_{0},q_{*}\right)-\delta\right\} \right)<0.\label{eq:RegIVbd}
\end{align}
\end{prop}
\begin{rem}
\label{rem:tauN}We assume that $\tau_{N}$ satisfies $\mathbb{E}\left\{ {\rm Cont}_{N,\beta}\left({\rm Reg}_{{\rm III}}\left(\tau_{N},\delta\right)\right)\right\} /\mathfrak{V}_{N,\beta}\left(m_{N}\right)\to0$,
as $N\to\infty$. The existence of such sequence follows from (\ref{eq:RegIIIbd})
and the fact that $N^{-1}\log\left(\mathfrak{V}_{N,\beta}\left(m_{N}\right)\right)\to\Frf\left(E_{0},q_{*}\right)$.
\end{rem}

\subsection{\label{sub:pfbds}Proof of Proposition \ref{prop:bounds}}

\subsubsection{Proof of (\ref{eq:Reg*bd})}

Denoting 
\[
{\rm Band}_{*}\left(\boldsymbol{\sigma}_{0}\right)={\rm Band}\left(\boldsymbol{\sigma}_{0},q_{*}-cN^{-1/2},q_{*}+cN^{-1/2}\right).
\]
to finish the proof it will certainly be enough to show that
\[
\lim_{\kappa,c\to\infty}\liminf_{N\to\infty}\mathbb{P}\left\{ \exists\boldsymbol{\sigma}_{0}\in\mathscr{C}_{N}\left(m_{N},m_{N}+\kappa\right):Z_{N,\beta}\left({\rm Band}_{*}\left(\boldsymbol{\sigma}_{0}\right)\right)\geq\ftail\left(\kappa\right)\mathfrak{V}_{N,\beta}\left(m_{N}+\kappa\right)\right\} =1.
\]
This will follow if we show that
\begin{equation}
\lim_{\kappa\to\infty}\liminf_{N\to\infty}\mathbb{P}\left\{ \left|\mathscr{C}_{N}\left(m_{N},m_{N}+\kappa\right)\right|\geq1\right\} =1,\label{eq:1225-1}
\end{equation}
and
\begin{equation}
\lim_{\kappa,c\to\infty}\limsup_{N\to\infty}\mathbb{E}\left|\left\{ \boldsymbol{\sigma}_{0}\in\mathscr{C}_{N}\left(m_{N},m_{N}+\kappa\right):\frac{Z_{N,\beta}\left({\rm Band}_{*}\left(\boldsymbol{\sigma}_{0}\right)\right)}{\mathfrak{V}_{N,\beta}\left(m_{N}+\kappa\right)}<\ftail\left(\kappa\right)\right\} \right|=0.\label{eq:1225-2}
\end{equation}
Equation (\ref{eq:1225-1}) follows from Theorem \ref{thm:ext proc}.
By Lemma \ref{lem:16}, the left-hand side of (\ref{eq:1225-2}) is
bounded from above by
\begin{equation}
\lim_{\kappa,c\to\infty}\limsup_{N\to\infty}C\cdot\int_{0}^{\kappa}dv\cdot e^{c_{p}v}\left(\mathbb{P}_{m_{N}+v,0}\left\{ \frac{Z_{N,\beta}\left({\rm Band}_{*}\left(\boldsymbol{\sigma}_{0}\right)\right)}{\mathfrak{V}_{N,\beta}\left(m_{N}+\kappa\right)}<\ftail\left(\kappa\right)\right\} \right)^{1/2}.\label{eq:1225-4}
\end{equation}
By Propositions \ref{prop:means}, \ref{prop:Gaussflucts} and the
fact that $\mathfrak{V}_{N,\beta}\left(m_{N}+v\right)$ decreases
with $v$, (\ref{eq:1225-4}) is bounded from above by 
\begin{align*}
 & \lim_{\kappa,c\to\infty}\limsup_{N\to\infty}C\cdot\int_{0}^{\kappa}dv\cdot e^{c_{p}v}\left(\mathbb{P}_{m_{N}+v,0}\left\{ \frac{Z_{N,\beta}\left({\rm Band}_{*}\left(\boldsymbol{\sigma}_{0}\right)\right)}{\mathfrak{V}_{N,\beta}\left(m_{N}+v\right)}<\ftail\left(\kappa\right)\right\} \right)^{1/2}\\
 & \leq\lim_{\kappa\to\infty}C\left(\mathbb{P}\left\{ e^{Y_{*}}\leq\ftail\left(\kappa\right)\right\} \right)^{1/2}\int_{0}^{\kappa}dv\cdot e^{c_{p}v}.
\end{align*}
Therefore the lemma follows from our assumption on $\ftail\left(\kappa\right)$.\qed

\subsubsection{Proof of (\ref{eq:RegIbd})}

Set $J_{N}=\left(m_{N}-\kappa',\,-E_{0}N+\tau_{N}N\right)$ and 
\begin{align*}
{\rm Band}_{{\rm I}}\left(\boldsymbol{\sigma}_{0}\right) & ={\rm Band}\left(\boldsymbol{\sigma}_{0},q_{**},1\right)\setminus{\rm Band}\left(\boldsymbol{\sigma}_{0},q_{*}-cN^{-1/2},q_{*}+cN^{-1/2}\right).
\end{align*}
Of course, 
\begin{equation}
{\rm Cont}_{N,\beta}\left({\rm Reg}_{{\rm I}}\left(c,\kappa',\tau_{N}N\right)\right)\leq\sum_{\boldsymbol{\sigma}_{0}\in\mathscr{C}_{N}\left(J_{N}\right)}Z_{N,\beta}\left({\rm Band}_{{\rm I}}\left(\boldsymbol{\sigma}_{0}\right)\right),\label{eq:1225-6}
\end{equation}
so an appropriate bound on the expectation of the right-hand side
of (\ref{eq:1225-6}) is sufficient.

We will first bound the expectation of the same with ${\rm Band}_{{\rm I}}\left(\boldsymbol{\sigma}_{0}\right)$
replaced by 
\[
{\rm Band}_{{\rm I}}^{\prime}\left(\boldsymbol{\sigma}_{0},\epsilon\right)={\rm Band}\left(\boldsymbol{\sigma}_{0},q_{**},1\right)\setminus{\rm Band}\left(\boldsymbol{\sigma}_{0},q_{*}-\epsilon,q_{*}+\epsilon\right).
\]
By Lemma \ref{lem:Z_1st_mom} and Corollary \ref{cor:conditional laws},
with $\varphi\left(x\right)=\Frf\left(E_{0},q_{*}\right)-a\left(\epsilon\right)$
(i.e., independent of $x$) and small enough $a\left(\epsilon\right)>0$,
with $I(\delta)=(-E_{0},-E_{0}+\delta)$,
\begin{align*}
\limsup_{N\to\infty}\sup_{u\in NI(\delta)}\left\{ \frac{1}{N}\log\left(\mathbb{E}_{u,0}\left\{ Z_{N,\beta}\left({\rm Band}_{{\rm I}}^{\prime}\left(\hat{\mathbf{n}},\epsilon\right)\right)\right\} \right)-\varphi\left(\frac{u}{N}\right)\right\} \\
\leq\limsup_{N\to\infty}\left\{ \frac{1}{N}\log\left(\mathbb{E}_{-E_{0},N,0}\left\{ Z_{N,\beta}\left({\rm Band}_{{\rm I}}^{\prime}\left(\hat{\mathbf{n}},\epsilon\right)\right)\right\} \right)-\varphi\left(\frac{u}{N}\right)\right\}  & \leq0.
\end{align*}
Therefore, by Lemma \ref{lem:17}, for any $\delta$, $\epsilon>0$,
\begin{align*}
 & \limsup_{N\to\infty}\frac{1}{N}\log\left(\mathbb{E}\left\{ \sum_{\boldsymbol{\sigma}_{0}\in\mathscr{C}_{N}\left(NI(\delta)\right)}Z_{N,\beta}\left({\rm Band}_{{\rm I}}^{\prime}\left(\boldsymbol{\sigma}_{0},\epsilon\right)\right)\right\} \right)\\
 & \quad\quad\leq\sup_{E\in I(\delta)}\left\{ \Theta_{p}\left(E\right)+\Frf\left(E_{0},q_{*}\right)-a\left(\epsilon\right)\right\} .
\end{align*}
Therefore, from the fact that $\Theta_{p}$ is continuous and $\Theta_{p}\left(-E_{0}\right)=0$
and $J_{N}\subset I(\delta)$ for large $N$ and any $\delta$,
\[
\limsup_{N\to\infty}\frac{1}{N}\log\left(\mathbb{E}\left\{ \sum_{\boldsymbol{\sigma}_{0}\in\mathscr{C}_{N}\left(J_{N}\right)}Z_{N,\beta}\left({\rm Band}_{{\rm I}}^{\prime}\left(\boldsymbol{\sigma}_{0},\epsilon\right)\right)\right\} \right)\leq\Frf\left(E_{0},q_{*}\right)-a\left(\epsilon\right).
\]
Since this holds true for any $\epsilon$ and $N^{-1}\log\left(\mathfrak{V}_{N,\beta}\left(m_{N}\right)\right)\overset{N\to\infty}{\to}\Frf\left(E_{0},q_{*}\right)$,
there exists a sequence $\epsilon_{N}>0$ such that $\epsilon_{N}\overset{N\to\infty}{\to}0$
for which
\begin{equation}
\limsup_{N\to\infty}\frac{\mathbb{E}\left\{ \sum_{\boldsymbol{\sigma}_{0}\in\mathscr{C}_{N}\left(J_{N}\right)}Z_{N,\beta}\left({\rm Band}_{{\rm I}}^{\prime}\left(\boldsymbol{\sigma}_{0},\epsilon_{N}\right)\right)\right\} }{\mathfrak{V}_{N,\beta}\left(m_{N}\right)}\leq0.\label{eq:1225-5}
\end{equation}

What remains is to prove an appropriate upper bound for the expectations
of the masses $\sum_{\boldsymbol{\sigma}_{0}\in\mathcal{C}\left(J_{N}\right)}Z_{N,\beta}\left({\rm Band}_{{\rm I}}^{\left(i\right)}\left(\boldsymbol{\sigma}_{0}\right)\right)$,
$i=1,2$, with 
\begin{align*}
{\rm Band}_{{\rm I}}^{\left(1\right)}\left(\boldsymbol{\sigma}_{0}\right) & ={\rm Band}_{N}\left(\boldsymbol{\sigma}_{0},q_{*}-\epsilon_{N},q_{*}-cN^{-1/2}\right),\\
{\rm Band}_{{\rm I}}^{\left(2\right)}\left(\boldsymbol{\sigma}_{0}\right) & ={\rm Band}_{N}\left(\boldsymbol{\sigma}_{0},q_{*}+cN^{-1/2},q_{*}+\epsilon_{N}\right),
\end{align*}
where we assume without loss of generality that $\epsilon_{N}>cN^{-1/2}$.

By Corollary \ref{cor:conditional laws}, 
\[
\mathbb{E}_{u,0}\left\{ \left(Z_{N,\beta}\left({\rm Band}_{{\rm I}}^{\left(1\right)}\left(\hat{\mathbf{n}}\right)\right)\right)^{2}\right\} =e^{-2\beta q_{*}^{p}\left(u-m_{N}\right)\left(1+o\left(1\right)\right)}\mathbb{E}_{m_{N},0}\left\{ \left(Z_{N,\beta}\left({\rm Band}_{{\rm I}}^{\left(1\right)}\left(\hat{\mathbf{n}}\right)\right)\right)^{2}\right\} ,
\]
uniformly in $u\in J_{N}$, as $N\to\infty$. From Corollary \ref{cor:2nd_moment},
(\ref{eq:29-2}), (\ref{eq:13-3}), and (\ref{eq:13-2}), uniformly
in $u\in J_{N}$, as $N\to\infty$, 
\begin{align*}
 & \negthickspace\negthickspace\negthickspace\mathbb{E}_{m_{N},0}\left\{ \left(Z_{N,\beta}\left({\rm Band}_{{\rm I}}^{\left(1\right)}\left(\hat{\mathbf{n}}\right)\right)\right)^{2}\right\} \\
 & \leq(1+o(1))\frac{1}{\sqrt{C_{*}}}\left(\mathbb{E}_{m_{N},0}\left\{ Z_{N,\beta}\left({\rm Band}_{{\rm I}}^{\left(1\right)}\left(\hat{\mathbf{n}}\right)\right)\right\} \right)^{2}\\
 & \leq(1+o(1))\left(\mathfrak{V}_{N,\beta}\left(m_{N}\right)r\left(c\right)\right)^{2},
\end{align*}
for some $r\left(c\right)\overset{c\to\infty}{\to}0$.

Now, Lemma \ref{lem:18} yields, for large enough $N$,
\begin{align*}
 & \negthickspace\negthickspace\negthickspace\mathbb{E}\left\{ \sum_{\boldsymbol{\sigma}_{0}\in\mathcal{C}\left(J_{N}\right)}Z_{N,\beta}\left({\rm Band}_{{\rm I}}^{\left(1\right)}\left(\boldsymbol{\sigma}_{0}\right)\right)\right\} \\
 & \leq C\mathfrak{V}_{N,\beta}\left(m_{N}\right)r\left(c\right)\int_{J_{N}}du\cdot e^{c_{p}\left(u-m_{N}\right)-\beta q_{*}^{p}\left(u-m_{N}\right)\left(1+o\left(1\right)\right)}\\
 & \leq C\mathfrak{V}_{N,\beta}\left(m_{N}\right)r\left(c\right)y\left(\beta\right),
\end{align*}
for some $y\left(\beta\right)<\infty$, assuming $\beta$ is large
enough so $\beta q_{*}^{p}>c_{p}$. By similar arguments, a similar
inequality holds for $B_{I}^{\left(2\right)}\left(\boldsymbol{\sigma}_{0},\epsilon\right)$.
Combined with (\ref{eq:1225-5}) and (\ref{eq:1225-6}) this proves
the lemma.\qed

\subsubsection{Proof (\ref{eq:RegIIbd})}

Set $J_{N}=\left(m_{N}+\kappa,\,-E_{0}N+\tau_{N}N\right)$ and 
\begin{align*}
{\rm Band}_{{\rm II}}\left(\boldsymbol{\sigma}_{0}\right) & ={\rm Band}\left(\boldsymbol{\sigma}_{0},q_{*}-cN^{-1/2},q_{*}+cN^{-1/2}\right).
\end{align*}
We have that 
\[
{\rm Cont}_{N,\beta}\left({\rm Reg}_{{\rm II}}\left(c,\kappa\right)\right)\leq\sum_{\boldsymbol{\sigma}_{0}\in\mathscr{C}_{N}\left(J_{N}\right)}Z_{N,\beta}\left({\rm Band}_{{\rm II}}\left(\boldsymbol{\sigma}_{0}\right)\right).
\]

By similar arguments to those used in the proof of (\ref{eq:RegIbd}),
we have here that for large $N$
\begin{equation}
\mathbb{E}\left\{ \sum_{\boldsymbol{\sigma}_{0}\in\mathscr{C}_{N}\left(J_{N}\right)}Z_{N,\beta}\left({\rm Band}_{{\rm II}}\left(\boldsymbol{\sigma}_{0}\right)\right)\right\} \leq C\mathfrak{V}_{N,\beta}\left(m_{N}\right)\int_{J_{N}}du\cdot e^{c_{p}\left(u-m_{N}\right)-\beta q_{*}^{p}\left(u-m_{N}\right)\left(1+o\left(1\right)\right)},\label{eq:08037}
\end{equation}
where $C>0$ is a universal constant (in particular, independent of
$\beta$, which will be useful in the proof of Proposition \ref{prop:large beta}).
In this case however we are taking the limit in $\kappa$ instead
of $c$, and since for large $\beta$ 
\[
\lim_{\kappa\to\infty}\limsup_{N\to\infty}\int_{J_{N}}du\cdot e^{c_{p}\left(u-m_{N}\right)}e^{-\beta q_{*}^{p}\left(u-m_{N}\right)\left(1+o\left(1\right)\right)}=0,
\]
the proof is completed.\qed

\subsubsection{Proof of (\ref{eq:RegIIIbd})}

Recall that $u_{{\rm LS}}=-\Frf\left(E_{0},q_{*}\right)N/\beta$ and
set 
\begin{align*}
J_{N}=NJ & =\left(\left(-E_{0}+\tau\right)N,\,-\Frf\left(E_{0},q_{*}\right)N/\beta+\delta N\right),\\
{\rm Cap}_{{\rm III}}\left(\boldsymbol{\sigma}_{0}\right) & ={\rm Cap}\left(\boldsymbol{\sigma}_{0},q_{**}\right).
\end{align*}
We have that 
\[
{\rm Cont}_{N,\beta}\left({\rm Reg}_{{\rm III}}\left(\tau,\delta\right)\right)\leq\sum_{\boldsymbol{\sigma}_{0}\in\mathscr{C}_{N}\left(J_{N}\right)}Z_{N,\beta}\left({\rm Cap}_{{\rm III}}\left(\boldsymbol{\sigma}_{0}\right)\right).
\]

By Corollary \ref{cor:conditional laws}, uniformly in $u\in J_{N}$,
\[
\mathbb{E}_{u,0}\left\{ Z_{N,\beta}\left({\rm Cap}_{{\rm III}}\left(\hat{\mathbf{n}}\right)\right)\right\} \leq e^{-\beta q_{**}^{p}\left(u+E_{0}N\right)}\mathbb{E}_{-E_{0}N,0}\left\{ Z_{N,\beta}\left({\rm Cap}_{{\rm III}}\left(\hat{\mathbf{n}}\right)\right)\right\} ,
\]
and therefore by Proposition \ref{prop:means},
\[
\limsup_{N\to\infty}\sup_{u\in J_{N}}\left\{ \frac{1}{N}\log\left(\mathbb{E}_{u,0}\left\{ Z_{N,\beta}\left({\rm Cap}_{{\rm III}}\left(\hat{\mathbf{n}}\right)\right)\right\} \right)-\left(\Frf\left(E_{0},q_{*}\right)-\beta q_{**}^{p}\left(\frac{u}{N}+E_{0}\right)\right)\right\} \leq0.
\]

From Lemma \ref{lem:17}, 
\[
\limsup_{N\to\infty}\frac{1}{N}\log\left(\mathbb{E}\left\{ \sum_{\boldsymbol{\sigma}_{0}\in\mathscr{C}_{N}\left(J_{N}\right)}Z_{N,\beta}\left({\rm Cap}_{{\rm III}}\left(\hat{\mathbf{n}}\right)\right)\right\} \right)\leq\sup_{E\in J}\left\{ \Theta_{p}\left(E\right)+\Frf\left(E_{0},q_{*}\right)-\beta q_{**}^{p}\left(E+E_{0}\right)\right\} .
\]
For large enough $\beta$, the supremum is equal to 
\[
\Theta_{p}\left(-E_{0}+\tau\right)+\Frf\left(E_{0},q_{*}\right)-\tau\beta q_{**}^{p}
\]
and is strictly less than $\Frf\left(E_{0},q_{*}\right)$. This completes
the proof. \qed

\subsubsection{Proof of (\ref{eq:RegIVbd})}

Let $\delta>0$ and set 
\[
J_{N}=NJ=N\left(-E_{0},\,-\Frf\left(E_{0},q_{*}\right)/\beta+\delta\right)
\]
and note that (see (\ref{eq:uLS}), (\ref{eq:m_N})) for large $N$,
$(m_{N}-\kappa',\, u_{{\rm LS}}+\delta N)\subset J_{N}$. Let $q_{1}<q_{2}$
be some overlaps in $(-1,1)$ and denote $W_{N,\beta}\left(\boldsymbol{\sigma}_{0}\right)=\frac{1}{N}\log Z_{N,\beta}({\rm Band}(\boldsymbol{\sigma}_{0},q_{1},q_{2}))$.
Let $\epsilon>0$ be a number such that 
\begin{equation}
\forall q\in(q_{{\rm LS}},q_{**}):\,\,\,\Frf\left(E_{0},q_{*}\right)-\FrfFt\left(E_{0},q\right)>\epsilon.\label{eq:a6}
\end{equation}
By Lemmas \ref{lem:Lambda_F_bound}, \ref{lem:Lambda_F2} and \ref{lem:devFE-RegIV},
as $N\to\infty$, 
\begin{align*}
 & \sup_{u\in J_{N}}\mathbb{P}_{u,0}\left\{ W_{N,\beta}\left(\hat{\mathbf{n}}\right)\geq\Frf\left(E_{0},q_{*}\right)-\epsilon\right\} \leq\mathbb{P}_{-NE_{0},0}\left\{ W_{N,\beta}\left(\hat{\mathbf{n}}\right)\geq\Frf\left(E_{0},q_{*}\right)-\epsilon\right\} \\
 & \leq\mathbb{P}_{-NE_{0},0}\left\{ W_{N,\beta}\left(\hat{\mathbf{n}}\right)-\mathbb{E}_{-NE_{0},0}W_{N,\beta}\left(\hat{\mathbf{n}}\right)\geq\Frf\left(E_{0},q_{*}\right)-\epsilon-\sup_{q\in\left(q_{1},q_{2}\right)}\FrfFt\left(E_{0},q\right)-o(1)\right\} \\
 & \leq\exp\left\{ -N\left(L_{\epsilon}\left(q_{1},q_{2}\right)-o(1)\right)\right\} ,
\end{align*}
where 
\[
L_{\epsilon}\left(q_{1},q_{2}\right)\triangleq\frac{\inf_{q\in\left(q_{1},q_{2}\right)}\left(\Frf\left(E_{0},q_{*}\right)-\epsilon-\FrfFt\left(E_{0},q\right)\right)^{2}}{2\beta^{2}\sum_{k=2}^{p}\sup_{q\in\left(q_{1},q_{2}\right)}\left(\alpha_{k}\left(q\right)\right)^{2}}.
\]
By Lemma \ref{lem:15}
\begin{align}
\limsup_{N\to\infty}\frac{1}{N}\log\left(\mathbb{E}\left|\left\{ \boldsymbol{\sigma}_{0}\in\mathscr{C}_{N}\left(J_{N}\right):\, W_{N,\beta}\left(\boldsymbol{\sigma}_{0}\right)\geq\Frf\left(E_{0},q_{*}\right)-\epsilon\right\} \right|\right) & \leq\sup_{E\in J}\Theta_{p}\left(E\right)-L_{\epsilon}\left(q_{1},q_{2}\right).\label{eq:e1}
\end{align}

Recall that $\Frf\left(E_{0},q_{*}\right)/\beta\to E_{0}$ as $\beta\to\infty$,
and note that $\Theta_{p}\left(E\right)$ increases in $E<0$. Thus,
for any given $\tau>0$, if $\beta$ is large enough and $\delta$
is small enough, then by Theorem \ref{thm:ABAC} 
\begin{equation}
\limsup_{N\to\infty}\frac{1}{N}\log\left(\mathbb{E}\left|\left\{ \boldsymbol{\sigma}_{0}\in\mathscr{C}_{N}\left(J_{N}\right)\right\} \right|\right)=\sup_{E\in J}\Theta_{p}\left(E\right)<\tau.\label{eq:a5}
\end{equation}
We claim that in order to finish the proof it will be sufficient to
show that there exist positive $\epsilon$, $c$ and $\beta_{0}$
such that (\ref{eq:a6}) holds and for any $\beta>\beta_{0}$ and
there exists $l:=l\left(\beta\right)$ such that for any $q_{1}$,
$q_{2}\in[q_{{\rm LS}},q_{**}]$ such that $0<q_{2}-q_{1}<l$ we have
that
\[
L_{\epsilon}\left(q_{1},q_{2}\right)>c.
\]
To see this, note that if the above holds and we choose $\beta_{0}$
large enough and $\delta$ small enough then, first we have from (\ref{eq:a5})
that 
\[
\limsup_{N\to\infty}\frac{1}{N}\log\left(\mathbb{P}\left\{ \frac{1}{N}\log\left(\left|\left\{ \boldsymbol{\sigma}_{0}\in\mathscr{C}_{N}\left(J_{N}\right)\right\} \right|\right)\geq\epsilon/2\right\} \right)\leq-\epsilon/4,
\]
and second 
\[
\limsup_{N\to\infty}\frac{1}{N}\log\left(\mathbb{P}\left\{ \exists\boldsymbol{\sigma}_{0}\in\mathscr{C}_{N}\left(J_{N}\right):\,\frac{1}{N}\log\left(Z_{N,\beta}\left({\rm Band}(\boldsymbol{\sigma}_{0},q_{{\rm LS}},q_{**})\right)\right)\geq\Frf\left(E_{0},q_{*}\right)-\epsilon\right\} \right)\leq-c/2.
\]

With $\beta$ fixed, 
\[
L_{\epsilon}\left(q_{1},q_{2}\right)\overset{q_{2}\searrow q_{1}}{\to}L_{\epsilon}\left(q_{1}\right)\triangleq\frac{\left(\Frf\left(E_{0},q_{*}\right)-\epsilon-\FrfFt\left(E_{0},q_{1}\right)\right)^{2}}{2\beta^{2}\sum_{k=2}^{p}\left(\alpha_{k}\left(q_{1}\right)\right)^{2}}
\]
uniformly in $q_{1}\in[q_{{\rm LS}},q_{**}]$. Hence, it will be enough
to show that, for large enough $\beta$, 
\begin{equation}
\sup_{q\in\left(q_{{\rm LS}},q_{**}\right)}L_{\epsilon}\left(q\right)>c.\label{eq:0101-4}
\end{equation}

Using the approximations 
\[
q_{*}=1-\frac{\chi_{1}}{\sqrt{2p(p-1)}}+O\left(\beta^{-2}\right),\,\,\, q_{**}=1-\frac{\chi_{3}}{\sqrt{2p(p-1)}}+O\left(\beta^{-2}\right),
\]
one can verify that, since $\Frf\left(E,q\right)\geq\FrfFt\left(E,q\right)$
(e.g., from (\ref{eq:a7}) and the fact that $\sum_{k=3}^{p}\alpha_{k}^{2}\left(q\right)=1-\sum_{k=0}^{2}\alpha_{k}^{2}\left(q\right)$),
\begin{equation}
\liminf_{\beta\to\infty}\left\{ \Frf\left(E_{0},q_{*}\right)-\FrfFt\left(E_{0},q_{**}\right)\right\} \geq\lim_{\beta\to\infty}\left(\Frf\left(E_{0},q_{*}\right)-\Frf\left(E_{0},q_{**}\right)\right)=M-m>0,\label{eq:0703-1}
\end{equation}
where $M$ and $m$ are the unique local maximum and minimum in $(0,\infty)$,
respectively, of the function $2^{-1}\log t-\sqrt{2}tE_{0}/E_{\infty}+2^{-1}t^{2}$.

If $q\in\left(q_{{\rm LS}},q_{**}\right)$, then $0<q<q_{c}$ (assuming
$\beta$ is large and all of those overlaps are defined) and thus
$\left|\alpha_{2}\left(q\right)\beta\right|>\sqrt{\frac{1}{2}}$.
By some calculus 
\[
\sup_{q\in\left(q_{{\rm LS}},q_{**}\right)}\left|\alpha_{k}\left(q\right)\cdot\frac{d}{dq}\alpha_{k}\left(q\right)\right|\leq t_{k}'\left(\frac{\beta}{\log\beta}\right)^{-k+1},
\]
for some $t_{k}>0$ and, using the fact that $q_{{\rm LS}}\to1$ as
$\beta\to\infty$, 
\begin{equation}
\lim_{\beta\to\infty}\sup_{q\in\left(q_{{\rm LS}},q_{**}\right)}\left|\frac{1}{\beta}\frac{d}{dq}\FrfFt\left(E_{0},q\right)-p\left(E_{0}-E_{\infty}\right)\right|=0.\label{eq:0101-7-1}
\end{equation}
It follows from (\ref{eq:0703-1}) and (\ref{eq:0101-7-1}) that for
large enough $\beta$ and $q\in(q_{{\rm LS}},q_{**})$, for some $c_{1}$,
$c_{2}>0$, 
\begin{equation}
\Frf\left(E_{0},q_{*}\right)-\FrfFt\left(E_{0},q\right)\geq2c_{1}+c_{2}\beta\left(q_{**}-q\right).\label{eq:0101-5-1}
\end{equation}
Since $q_{{\rm LS}}\to1$ as $\beta\to\infty$, from the definition
\ref{eq:alpha_k} of $\alpha_{k}\left(q\right)$, for any $q\in[q_{{\rm LS}},q_{**}]$,
\begin{equation}
\sum_{k=2}^{p}\left(\alpha_{k}\left(q\right)\right)^{2}\leq c_{3}\left(\alpha_{2}\left(q\right)\right)^{2},\label{eq:0101-6-1}
\end{equation}
for some $c_{3}>1$ for any $\beta$ large enough. Hence, with $\epsilon=c_{1}$,
if $1-2\left(1-q_{**}\right)<q\leq q_{**}$, then $\alpha_{2}\left(q\right)\leq2\alpha_{2}\left(q_{**}\right)$
and therefore 
\[
L_{c_{1}}\left(q\right)\geq\frac{c_{1}^{2}}{2\beta^{2}c_{3}\left(2\alpha_{2}\left(q_{**}\right)\right)^{2}};
\]
and if $q_{{\rm LS}}\leq q\leq1-2\left(1-q_{**}\right)$ then $q_{**}-q\geq(1-q)/2$
and, for an appropriate $c_{4}>0$ assuming $\beta$ is large, 
\[
L_{c_{1}}\left(q\right)\geq\frac{\left(c_{2}\beta\frac{1}{2}\left(1-q\right)\right)^{2}}{2\beta^{2}c_{3}\left(\alpha_{2}\left(q\right)\right)^{2}}\geq c_{4}.
\]
Since $\beta\cdot\alpha_{2}(q_{**})$ is independent of $\beta$,
the proof is completed.\qed

\section{\label{sec:thm1}Proof of Theorem \ref{thm:Geometry}}

\subsection{Proof of part (\ref{enu:Geometry1})}

For the definitions of the various regions used below see Section
\ref{sub:Overlap-depth-regions} and (\ref{eq:RegUB}). Throughout
the proof $\beta$ and the positive constants $c$, $\kappa$, $\kappa'$
are assumed to be large enough and $\delta$ is assumed to be small
enough whenever needed. Let $\epsilon>0$ be an arbitrary small number.
For given $\kappa$ and $\kappa'$, assume $k$ is large enough so
that, by Corollary \ref{cor:min} with probability at least $1-\epsilon$
for large $N$
\[
Z_{N,\beta}\left(\cup_{i\in[k]}{\rm Band}_{i}(cN^{-1/2})\right)\geq{\rm Cont}_{N,\beta}\left({\rm Reg}_{*}\left(c,\kappa,\kappa'\right)\right).
\]
From Lemma \ref{lem:LSnegligible} and Corollary \ref{cor:min}, with
probability at least $1-\epsilon$ for large $N$, $H_{N}(\boldsymbol{\sigma}_{0}^{1})>m_{N}-\kappa'$
and 
\begin{align*}
Z_{N,\beta}\left(\mathbb{S}^{N-1}\setminus\cup_{i\in[k]}{\rm Band}_{i}(cN^{-1/2})\right) & \leq Z_{N,\beta}\left(\left\{ \boldsymbol{\sigma}:H_{N}\left(\boldsymbol{\sigma}\right)\geq u_{{\rm LS}}+\delta N\right\} \right)\\
 & +{\rm Cont}_{N,\beta}\left({\rm Reg}_{{\rm UB}}\left(\delta,c,\kappa,\kappa'\right)\right),
\end{align*}
where ${\rm Reg}_{{\rm UB}}(\delta,c,\kappa,\kappa')$ is contained,
up to a set of Lebesgue measure zero, in 
\[
{\rm Reg}_{{\rm I}}\left(c,\kappa'\right)\cup{\rm Reg}_{{\rm II}}\left(c,\kappa\right)\cup{\rm Reg}_{{\rm III}}(\tau_{N},\delta)\cup{\rm Reg}_{{\rm IV}}\left(\delta\right).
\]

By Proposition \ref{prop:bounds}, (\ref{eq:LSbd}) and (\ref{eq:uLS}),
Remark \ref{rem:tauN}, and the facts that
\[
N^{-1}\log\left(\mathfrak{V}_{N,\beta}\left(m_{N}\right)\right)\overset{N\to\infty}{\to}\Frf\left(E_{0},q_{*}\right)
\]
and $\ftail\left(\kappa\right)\mathfrak{V}_{N,\beta}\left(m_{N}+\kappa\right)/\mathfrak{V}_{N,\beta}\left(m_{N}\right)>C$
for some $C>0$ independent of $N$, we have that with probability
at least $1-3\epsilon$ for large $N$
\[
\frac{Z_{N,\beta}\left(\mathbb{S}^{N-1}\setminus\cup_{i\in[k]}{\rm Band}_{i}(cN^{-1/2})\right)}{Z_{N,\beta}\left(\cup_{i\in[k]}{\rm Band}_{i}(cN^{-1/2})\right)}<\epsilon,
\]
and the proof is completed.\qed

\subsection{Proof of part (\ref{enu:Geometry2})}

Let $\epsilon,\,\delta>0$ and $k\geq1$ be arbitrary. Choose $\kappa:=\kappa(k,\epsilon)>0$
large enough such that by Corollary \ref{cor:min}, setting $J_{N}=\left(m_{N}-\kappa,m_{N}+\kappa\right)$,
with probability at least $1-\epsilon$ for large $N$
\[
\mathscr{C}_{N}\left(J_{N}\right)\supset\left\{ \boldsymbol{\sigma}_{0}^{i}:\, i\in[k]\right\} .
\]
Define (see definition (\ref{eq:ZtimesZ2})) 
\[
D_{\boldsymbol{\sigma}}^{(1)}=\frac{\left(Z\times Z\right)_{N,\beta}\left({\rm Band}\left(\boldsymbol{\sigma},cN^{-1/2}\right);\left[-1,1\right]\setminus\left(-\rho'N^{-1/2},\rho'N^{-1/2}\right)\right)}{\left(\mathbb{E}_{u,0}\left\{ Z_{N,\beta}\left({\rm Band}\left(\hat{\mathbf{n}},cN^{-1/2}\right)\right)\right\} \right)^{2}}
\]
with 
\[
{\rm Band}\left(\boldsymbol{\sigma},cN^{-1/2}\right)={\rm Band}\left(\boldsymbol{\sigma},q_{*}-cN^{-1/2},q_{*}+cN^{-1/2}\right).
\]
Using part \ref{enu:3_lem2nd} of Lemma \ref{lem:Z_2nd_mom}, assume
$\rho':=\rho'(\delta,\epsilon,\kappa)>0$ is large enough such that
for large $N$, $\sup_{u\in J_{N}}\mathbb{E}_{u,0}D_{\hat{\mathbf{n}}}^{(1)}$
is small enough so that by Lemma \ref{lem:16} 
\[
\mathbb{P}\left\{ \exists\boldsymbol{\sigma}_{0}\in\mathscr{C}_{N}\left(J_{N}\right):\, D_{\boldsymbol{\sigma}_{0}}^{(1)}>\delta\right\} <\epsilon.
\]

Define 
\[
D_{\boldsymbol{\sigma}}^{(2)}=\left(\frac{Z_{N,\beta}\left({\rm Band}\left(\boldsymbol{\sigma},cN^{-1/2}\right)\right)}{\mathbb{E}_{u,0}\left\{ Z_{N,\beta}\left({\rm Band}\left(\hat{\mathbf{n}},cN^{-1/2}\right)\right)\right\} }\right)^{2}.
\]
Assume $\delta':=\delta'(\epsilon,\kappa)$ is small enough so that
by Proposition \ref{prop:Gaussflucts} for large $N$,$\sup_{u\in J_{N}}\mathbb{P}_{u,0}\{D_{\hat{\mathbf{n}}}^{(2)}<\delta'\}$
is small enough so that by Lemma \ref{lem:16}
\[
\mathbb{P}\left\{ \exists\boldsymbol{\sigma}_{0}\in\mathscr{C}_{N}\left(J_{N}\right):\, D_{\boldsymbol{\sigma}_{0}}^{(2)}<\delta'\right\} <\epsilon.
\]

Combining the above we have that with probability at least $1-3\epsilon$,
for large $N$ and any $i\in[k]$,
\begin{equation}
D_{\boldsymbol{\sigma}_{0}^{i}}^{(1)}/D_{\boldsymbol{\sigma}_{0}^{i}}^{(2)}\leq\delta/\delta'.\label{eq:w/w}
\end{equation}
Since (see (\ref{eq:projective overlap}) for the definition of $R_{\boldsymbol{\sigma}_{0}}(\boldsymbol{\sigma},\boldsymbol{\sigma}')$)
\begin{align}
R\left(\boldsymbol{\sigma},\boldsymbol{\sigma}^{\prime}\right) & =R(\boldsymbol{\sigma},\boldsymbol{\sigma}_{0})R(\boldsymbol{\sigma}',\boldsymbol{\sigma}_{0})\nonumber \\
 & +R_{\boldsymbol{\sigma}_{0}}(\boldsymbol{\sigma},\boldsymbol{\sigma}')N^{-1}\left\Vert \boldsymbol{\sigma}-\boldsymbol{\sigma}_{0}R(\boldsymbol{\sigma},\boldsymbol{\sigma}_{0})\right\Vert \left\Vert \boldsymbol{\sigma}'-\boldsymbol{\sigma}_{0}R(\boldsymbol{\sigma}',\boldsymbol{\sigma}_{0})\right\Vert ,\label{eq:08031}
\end{align}
if $\boldsymbol{\sigma},\boldsymbol{\sigma}^{\prime}\in{\rm Band}\left(\boldsymbol{\sigma}_{0},cN^{-1/2}\right)$
and $|R_{\boldsymbol{\sigma}_{0}}(\boldsymbol{\sigma},\boldsymbol{\sigma}')|<\rho'N^{-1/2}$
, then
\[
\left|R\left(\boldsymbol{\sigma},\boldsymbol{\sigma}^{\prime}\right)-q_{*}^{2}\right|<\rho N^{-1/2}
\]
for some $\rho:=\rho(c,\rho')$ independent of $N$. Thus, under (\ref{eq:w/w})
\[
G_{N,\beta}^{c,i}\otimes G_{N,\beta}^{c,i}\left\{ \left|R\left(\boldsymbol{\sigma},\boldsymbol{\sigma}^{\prime}\right)\mp q_{*}^{2}\right|>\rho N^{-1/2}\right\} <\delta/\delta'.
\]
By choosing $\delta$ small enough compared to $\delta'$, (\ref{eq:2601})
follows for the case $\pm i=i$. For even $p$, the case $\pm i=-i$
follows  from the fact that $H_{N}\left(\boldsymbol{\sigma}\right)=H_{N}\left(-\boldsymbol{\sigma}\right)$.\qed

\subsection{\label{sub:pfThm1Pt3}Proof of part (\ref{enu:Geometry3})}

We first prove the following lemma.
\begin{lem}
\label{lem:center}For any $q\in(-1,1)$ there exists a function $\delta_{q}\left(\epsilon\right)>0$
with $\lim_{\epsilon\to0^{+}}\delta_{q}\left(\epsilon\right)=0$ such
that the following holds. Let $q\in(-1,1)$, $0<\epsilon<|q|,\,1-|q|$,
let $M$ be a fixed deterministic measure on ${\rm Band}\left(\boldsymbol{\sigma}_{0},q-\epsilon,q+\epsilon\right)$
(where the dimension $N$ is fixed) and assume that 
\begin{equation}
M\otimes M\left\{ |R(\boldsymbol{\sigma},\boldsymbol{\sigma}')-q^{2}|>\epsilon\right\} <\epsilon.\label{eq:10031}
\end{equation}
Then, denoting $\boldsymbol{\sigma}_{\perp}=\boldsymbol{\sigma}-\boldsymbol{\sigma}_{0}R\left(\boldsymbol{\sigma},\boldsymbol{\sigma}_{0}\right)$,

\begin{equation}
\sup_{\boldsymbol{\tau}\in\mathbb{S}^{N-1}}M\left\{ \left|R(\boldsymbol{\sigma}_{\perp},\boldsymbol{\tau})\right|>\delta_{q}(\epsilon)\right\} \leq\delta_{q}(\epsilon).\label{eq:08032-1}
\end{equation}
\end{lem}
\begin{proof}
Assume towards contradiction that there exist $q\in(-1,1)$, $\delta>0$,
which will be fixed from now on, and $\epsilon>0$ as small as we
wish such that (\ref{eq:10031}) holds and for some $\boldsymbol{\tau}\in\mathbb{S}^{N-1}$,
\begin{equation}
M\left\{ \left|R(\boldsymbol{\sigma}_{\perp},\boldsymbol{\tau})\right|>\delta\right\} \geq\delta.\label{eq:10032}
\end{equation}
Let $\boldsymbol{\sigma}_{j}$, $1\leq j$, be distributed according
to $M^{\otimes\infty}$ (i.e., be an i.i.d sequence of samples from
$M$), and define $\boldsymbol{\sigma}_{j,\perp}$ similarly to $\boldsymbol{\sigma}_{\perp}$.
By (\ref{eq:10031}) and (\ref{eq:08031}) there exists a deterministic
function $\rho_{q}(\epsilon)$ such that $\lim_{\epsilon\to0^{+}}\rho_{q}\left(\epsilon\right)=0$
and, for any $k\geq1$, 
\begin{align*}
 & \negthickspace\negthickspace M^{\otimes\infty}\left\{ \max_{i<j\leq k}\left|R\left(\boldsymbol{\sigma}_{i,\perp},\boldsymbol{\sigma}_{j,\perp}\right)\right|>\rho_{q}(\epsilon)\right\} \\
 & \leq M^{\otimes\infty}\left\{ \max_{i<j\leq k}\left|R\left(\boldsymbol{\sigma}_{i},\boldsymbol{\sigma}_{j}\right)-q^{2}\right|>\epsilon\right\} \leq\frac{k(k-1)}{2}\epsilon.
\end{align*}
From (\ref{eq:10032}), 
\[
M^{\otimes\infty}\left\{ \min_{i\leq k}\left|R\left(\boldsymbol{\sigma}_{i,\perp},\boldsymbol{\tau}\right)\right|>\delta\right\} \geq\delta^{k}.
\]
For arbitrary $k\geq1$, assuming $\epsilon$ is such that $\delta^{k}>k(k-1)\epsilon/2$,
we conclude that there exists deterministic vectors, related by $\mathbf{v}_{0}=\boldsymbol{\tau}/\left\Vert \boldsymbol{\tau}\right\Vert ,\,\mathbf{v}_{1}=\boldsymbol{\sigma}_{1,\perp}/\|\boldsymbol{\sigma}_{1,\perp}\|,...,\mathbf{v}_{k}=\boldsymbol{\sigma}_{k,\perp}/\|\boldsymbol{\sigma}_{k,\perp}\|$
to realizations of the vectors above, such that the matrix $(\mathbf{v}_{i}\cdot\mathbf{v}_{j})$
is of the from 
\[
\mathbf{A}=\left(\begin{array}{ccccc}
1 & \delta_{1} & \delta_{2} & \cdots & \delta_{k}\\
\delta_{1} & 1 & \rho_{1,2} & \cdots & \rho_{1,k}\\
\delta_{2} & \rho_{1,2} & \ddots &  & \vdots\\
\vdots & \vdots &  & \ddots & \rho_{k-1,k}\\
\delta_{k} & \rho_{1,k} & \cdots & \rho_{k-1,k} & 1
\end{array}\right)
\]
with $|\delta_{i}|>\delta$ and $|\rho_{i,j}|\leq\rho_{q}(\epsilon)$.
With $\mathbf{x}=\left(1,-\delta_{1},...,-\delta_{k}\right)$, $\left\Vert \mathbf{v}_{0}-\delta_{1}\mathbf{v}_{1}-\cdots-\delta_{k}\mathbf{v}_{k}\right\Vert _{2}^{2}=\mathbf{x}\mathbf{A}\mathbf{x}^{T}$.
However, by a direct computation one sees that $\mathbf{x}\mathbf{A}\mathbf{x}^{T}\leq1-k\delta^{2}+k^{2}\delta^{2}\rho_{q}(\epsilon)$,
which is negative assuming $k>1/\delta^{2}$ and $\epsilon$ is small
enough. We conclude that our initial assumption is false and therefore
the proof is completed.
\end{proof}
Denote $q_{i}=R(\boldsymbol{\sigma},\boldsymbol{\sigma}_{0}^{i})$
and $q_{j}=R(\boldsymbol{\sigma}^{\prime},\boldsymbol{\sigma}_{0}^{j})$.
From part (\ref{enu:Geometry2}) of Theorem \ref{thm:Geometry} and
Lemma \ref{lem:center} applied once with $\boldsymbol{\tau}=\boldsymbol{\sigma}$
and $M=G_{N,\beta}^{c,j}$ and once with $\boldsymbol{\tau}=\boldsymbol{\sigma}_{0}^{j}$
and $M=G_{N,\beta}^{c,i}$, for arbitrary $\delta>0$, with probability
that goes to $1$ as $N\to\infty$, 
\begin{align}
G_{N,\beta}^{c,i}\otimes G_{N,\beta}^{c,j}\left\{ \frac{1}{N}\left|\left\langle \boldsymbol{\sigma},\boldsymbol{\sigma}^{\prime}-q_{j}\boldsymbol{\sigma}_{0}^{j}\right\rangle \right|\leq\delta\right\}  & \geq1-\delta,\nonumber \\
G_{N,\beta}^{c,i}\otimes G_{N,\beta}^{c,j}\left\{ \frac{1}{N}\left|\left\langle \boldsymbol{\sigma}_{0}^{j},\boldsymbol{\sigma}-q_{i}\boldsymbol{\sigma}_{0}^{i}\right\rangle \right|\leq\delta\right\}  & \geq1-\delta.\label{eq:2803-2}
\end{align}
Whenever the two events above occur
\[
\frac{1}{N}\left|\left\langle \boldsymbol{\sigma},\boldsymbol{\sigma}^{\prime}\right\rangle -q_{j}\left\langle \boldsymbol{\sigma},\boldsymbol{\sigma}_{0}^{j}\right\rangle +q_{j}\left(\left\langle \boldsymbol{\sigma}_{0}^{j},\boldsymbol{\sigma}\right\rangle -q_{i}\left\langle \boldsymbol{\sigma}_{0}^{j},\boldsymbol{\sigma}_{0}^{i}\right\rangle \right)\right|\leq2\delta,
\]
and since $|q_{i}-q_{*}|$, $|q_{j}-q_{*}|<cN^{-1/2}$, for large
$N$,
\[
\left|R(\boldsymbol{\sigma},\boldsymbol{\sigma}^{\prime})-q_{*}^{2}R(\boldsymbol{\sigma}_{0}^{i},\boldsymbol{\sigma}_{0}^{j})\right|\leq3\delta.
\]
The proof is completed by Corollaries \ref{cor:min} and \ref{cor:orth}.\qed

\section{\label{sec:thmFE}Proofs of Theorem \ref{thm:free-energy}, Proposition
\ref{prop:large beta} and Corollary \ref{cor:FE_u_Hess_rep}}

\subsection{Proof of Theorem \ref{thm:free-energy}}

Let $\epsilon>0$ be an arbitrary number. For large enough $\kappa$,
$c>0$, by (\ref{eq:thm1_1}) and Corollary \ref{cor:min}, for large
$N$
\begin{equation}
\mathbb{P}\left\{ \left|NF_{N,\beta}-\log\left(C_{N,\beta}^{c,\kappa}\right)\right|>\epsilon\right\} <\epsilon,\label{eq:08034}
\end{equation}
with $C_{N,\beta}^{c,\kappa}={\rm Cont}_{N,\beta}\left({\rm Reg}_{*}\left(c,\kappa,\kappa\right)\right)$.
Hence to finish the proof it is sufficient to prove that for any $\delta>0$,
there exist $c_{0}(\delta)$, $\kappa_{0}(\delta)>0$ such that for
any $c>c_{0}(\delta)$, $\kappa>\kappa_{0}(\delta)$ there exists
$t=t(\delta,c,\kappa)>0$ such that for large $N$, 
\begin{align*}
\mathbb{P}\left\{ \left|\log\left(C_{N,\beta}^{c,\kappa}\right)-N\Frf\left(E_{0},q_{*}\right)+\frac{\beta q_{*}^{p}}{2c_{p}}\log N\right|>t\right\}  & <\delta.
\end{align*}
This follows from (\ref{eq:Reg*bd}) and (\ref{eq:08037}).\qed

\subsection{Proof of Proposition \ref{prop:large beta}}

In Corollary \ref{cor:min} it is stated that $\mathbb{P}\{H_{N}(\boldsymbol{\sigma}_{0}^{1})-m_{N}\geq x\}\to\exp\left\{ -c_{p}^{-1}e^{c_{p}x}\right\} $,
as $N\to\infty$. In view of this and the fact that $\lim_{\beta\to\infty}q_{*}=1$,
to finish the proof it is enough to prove that for any $\delta>0$,
\[
\limsup_{\beta\to\infty}\limsup_{N\to\infty}\mathbb{P}\left\{ \left|\frac{1}{\beta q_{*}^{p}}\left(NF_{N,\beta}-\log\left(\mathfrak{V}_{N,\beta}\left(m_{N}\right)\right)\right)+H_{N}\left(\boldsymbol{\sigma}_{0}^{1}\right)-m_{N}\right|\geq\delta\right\} \leq\delta.
\]
From part (\ref{enu:Geometry1}) of Theorem \ref{thm:Geometry}, the
above will follow if we show that for any $\delta>0$, 
\begin{align}
\limsup_{\beta\to\infty}\limsup_{c,k\to\infty}\limsup_{N\to\infty}\mathbb{P}\left\{ \left|\frac{1}{\beta q_{*}^{p}}\left[\log\Big(\sum_{i\in[k]}Z_{N,\beta}\left({\rm Band}_{i}\left(cN^{-1/2}\right)\right)\Big)\right.\right.\right.\label{eq:b7}\\
\left.\left.\left.\vphantom{\sum_{i\in[k]}}-\log\left(Z_{N,\beta}\left({\rm Band}_{1}\left(cN^{-1/2}\right)\right)\right)\right]\right|\geq\delta\right\}  & \leq\delta,\nonumber 
\end{align}
and
\begin{align}
\limsup_{\beta\to\infty}\limsup_{c\to\infty}\limsup_{N\to\infty}\mathbb{P}\left\{ \left|\frac{1}{\beta q_{*}^{p}}\left[\log\left(Z_{N,\beta}\left({\rm Band}_{1}\left(cN^{-1/2}\right)\right)\right)\right.\right.\right.\label{eq:b3}\\
\left.\left.\vphantom{\frac{1}{\beta q_{*}^{p}}}\left.-\log\left(\mathfrak{V}_{N,\beta}\left(m_{N}\right)\right)\right]+H_{N}\left(\boldsymbol{\sigma}_{0}^{1}\right)-m_{N}\right|\geq\delta\right\}  & \leq\delta.\nonumber 
\end{align}
(Where we used the fact that the bands above are disjoint for large
enough $\beta$ and $N$, since $\lim_{\beta\to\infty}q_{*}=1$ and
by Corollary \ref{cor:orth}.)

Fix $\beta>0$ which will be assumed to be large enough wherever needed.
Let $\delta>0$ be arbitrary. First, we prove (\ref{eq:b3}) by showing
that, with high probability, $H_{N}\left(\boldsymbol{\sigma}_{0}^{1}\right)$
is not far from $m_{N}$ and that for any critical value close enough
to $m_{N}$ the fluctuation of the mass of the corresponding band
from its expectation is not large. Choose $\kappa=\kappa(\delta)$
large enough so that by Corollary \ref{cor:min},%
\footnote{We remark that $\kappa$ is independent of $\beta$, as will be all
other parameters we choose in the rest of the proof. This is of course
crucial to the proof.%
}
\begin{equation}
\limsup_{N\to\infty}\mathbb{P}\{\left|H_{N}(\boldsymbol{\sigma}_{0}^{1})-m_{N}\right|\geq\kappa\}<\delta/8.\label{eq:b1}
\end{equation}
Denote $Z_{N,\beta,c}(\boldsymbol{\sigma}_{0})=Z_{N,\beta}\left({\rm Band}\left(\boldsymbol{\sigma}_{0},q_{*}-cN^{-1/2},q_{*}+cN^{-1/2}\right)\right)$.
Choose $t=t(\delta,\kappa)$ sufficiently large so that by Propositions
\ref{prop:means} and \ref{prop:Gaussflucts} and Lemma \ref{lem:16},
\begin{align}
\limsup_{c\to\infty}\limsup_{N\to\infty}\mathbb{P}\left\{ \exists\boldsymbol{\sigma}_{0}\in\mathscr{C}_{N}\left(m_{N}-\kappa,m_{N}+\kappa\right):\right.\label{eq:b2}\\
\left.\left|\log\left(Z_{N,\beta,c}(\boldsymbol{\sigma}_{0})\right)-\log\left(\mathfrak{V}_{N,\beta}\left(H_{N}(\boldsymbol{\sigma}_{0})\right)\right)\right|\geq t\right\}  & <\delta/8.\nonumber 
\end{align}
Since (by the definition of $\mathfrak{V}_{N,\beta}(u)$ (\ref{eq:V(u)}))
\[
\log\left(\mathfrak{V}_{N,\beta}\left(H_{N}(\boldsymbol{\sigma}_{0})\right)\right)-\log\left(\mathfrak{V}_{N,\beta}\left(m_{N}\right)\right)=-H_{N}\left(\boldsymbol{\sigma}_{0}\right)+m_{N},
\]
(\ref{eq:b1}) and (\ref{eq:b2}) imply (\ref{eq:b3}). They also
imply a lower bound on the mass corresponding to $H_{N}\left(\boldsymbol{\sigma}_{0}^{1}\right)$.
Namely, 
\begin{align*}
\limsup_{c\to\infty}\limsup_{N\to\infty}\mathbb{P}\left\{ \log\left(Z_{N,\beta}\left({\rm Band}_{1}\left(cN^{-1/2}\right)\right)\right)\right.\\
\left.\leq\log\left(\mathfrak{V}_{N,\beta}\left(m_{N}-\kappa\right)\right)-t\right\}  & <\delta/4.
\end{align*}
Therefore, using (\ref{eq:08037}) we can choose $\kappa'=\kappa'(\delta,\kappa,t)$
such that for any fixed $\kappa''>0$, 
\begin{equation}
\limsup_{c\to\infty}\limsup_{N\to\infty}\mathbb{P}\left\{ \frac{\sum_{\boldsymbol{\sigma}_{0}\in\mathscr{C}_{N}\left(m_{N}-\kappa'',m_{N}-\kappa'\right)}Z_{N,\beta,c}(\boldsymbol{\sigma}_{0})}{Z_{N,\beta}\left({\rm Band}_{1}\left(cN^{-1/2}\right)\right)}\geq\frac{\delta}{4}\right\} <\delta/2,\label{eq:b6}
\end{equation}
uniformly in $\beta\geq\beta_{0}$ for some large enough $\beta_{0}$
(which is independent of our choice of all other parameters). That
is, we have an upper bound on the masses of bands corresponding to
critical values far enough from $m_{N}$ (within a microscopic distance).

Lastly, we need to bound the weights of the bands corresponding to
critical values larger than $m_{N}-\kappa'$ that are not equal to
$H_{N}\left(\boldsymbol{\sigma}_{0}^{1}\right)$. This is done by
bounding the number of such points and the fluctuations of the corresponding
masses from their expectation. Choose large enough $k=k(\delta,\kappa')$
so that from Corollary \ref{cor:min},
\[
\limsup_{N\to\infty}\mathbb{P}\left\{ \left|\mathscr{C}_{N}\left(-\infty,m_{N}-\kappa'\right)\right|\geq k\right\} <\delta/8.
\]
Similarly to (\ref{eq:b2}), choose $t'=t'(\delta,\kappa)$ large
enough so that
\begin{align}
\limsup_{c\to\infty}\limsup_{N\to\infty}\mathbb{P}\left\{ \exists\boldsymbol{\sigma}_{0}\in\mathscr{C}_{N}\left(-\infty,m_{N}-\kappa'\right):\right.\label{eq:b2-1}\\
\left.\left|\log\left(Z_{N,\beta,c}(\boldsymbol{\sigma}_{0})\right)-\log\left(\mathfrak{V}_{N,\beta}\left(H_{N}(\boldsymbol{\sigma}_{0})\right)\right)\right|\geq t'\right\}  & <\delta/8.\nonumber 
\end{align}
Since $\mathfrak{V}_{N,\beta}(u)$ is decreasing, it follows that
\begin{equation}
\limsup_{c\to\infty}\limsup_{N\to\infty}\mathbb{P}\left\{ \frac{\sum_{\boldsymbol{\sigma}_{0}\in\mathscr{C}_{N}\left(m_{N}-\kappa'',m_{N}-\kappa'\right)}Z_{N,\beta,c}(\boldsymbol{\sigma}_{0})}{\mathfrak{V}_{N,\beta}\left(H_{N}(\boldsymbol{\sigma}_{0}^{1})\right)}\geq ke^{t'}\right\} <\delta/2.\label{eq:b5}
\end{equation}
Since the bounds in (\ref{eq:b6}) and (\ref{eq:b5}) are independent
of $\beta$ (assuming it is large enough) and since the difference
in (\ref{eq:b7}) is divided by $\beta q_{*}^{p}$, (\ref{eq:b7})
follows from the above and the proof is completed.\qed

\subsection{Proof of Corollary \ref{cor:FE_u_Hess_rep}}

Let $0<a$ and set $J_{N}=(m_{N}-a,m_{N}+a)$. For any $c,\,\delta>0$,
by Lemmas \ref{lem:ConcentrationZ} and \ref{lem:16},
\[
\lim_{N\to\infty}\mathbb{E}\left|\left\{ \boldsymbol{\sigma}_{0}\in\mathscr{C}_{N}\left(J_{N}\right):\,\left|\frac{Z_{N,\beta,c}(\boldsymbol{\sigma}_{0})}{\mathbb{E}_{H_{N}(\boldsymbol{\sigma}_{0}),0,\mathbf{\mathbf{G}}(\boldsymbol{\sigma}_{0})}\left\{ Z_{N,\beta,c}(\hat{\mathbf{n}})\right\} }-1\right|\geq\delta\right\} \right|=0,
\]
where $Z_{N,\beta,c}(\boldsymbol{\sigma}_{0})=Z_{N,\beta}\left({\rm Band}\left(\boldsymbol{\sigma}_{0},q_{*}-cN^{-1/2},q_{*}+cN^{-1/2}\right)\right)$.
It therefore also follows that there exist positive sequences $a_{N}>0$,
$\delta_{N}=o(1)$, $c_{N}=o(N^{1/2})$ with $a_{N},\, c_{N}\to\infty$,
as $N\to\infty$, such that the above holds with them instead of the
corresponding constants. By Corollary \ref{cor:min}, there exists
a sequence $k_{N}\geq1$ such that $k_{N}\to\infty$ and 
\[
\lim_{N\to\infty}\mathbb{P}\left\{ \forall i\in[k_{N}]:\,\left|Z_{N,\beta,c_{N}}(\boldsymbol{\sigma}_{0}^{i})/Z_{N,i}-1\right|<\delta_{N}\right\} =1,
\]
where we define
\begin{equation}
Z_{N,i}\triangleq\mathbb{E}_{H_{N}(\boldsymbol{\sigma}_{0}^{i}),0,\mathbf{\mathbf{G}}(\boldsymbol{\sigma}_{0}^{i})}\left\{ Z_{N,\beta,c_{N}}(\hat{\mathbf{n}})\right\} .\label{eq:ZNi}
\end{equation}
By Corollaries \ref{cor:min} and \ref{cor:orth}, by choosing $k_{N}$
that increases slower if needed, we also have that 
\[
\lim_{N\to\infty}\mathbb{P}\left\{ \forall i,j\in[k_{N}],\, i\neq\pm j:\,\left|R\left(\boldsymbol{\sigma}_{0}^{i},\boldsymbol{\sigma}_{0}^{j}\right)\right|\leq\tau_{N}\right\} =1,
\]
for some $\tau_{N}=o(1)$. For large $\beta$, since $q_{*}(\beta)\to1$
as $\beta\to\infty$, for $i$ and $j$ with $\left|R\left(\boldsymbol{\sigma}_{0}^{i},\boldsymbol{\sigma}_{0}^{j}\right)\right|\leq\tau_{N}$,
assuming $N$ is large enough,
\[
\cap_{a=i,j}{\rm Band}\left(\boldsymbol{\sigma}_{0}^{a},q_{*}-c_{N}N^{-1/2},q_{*}+c_{N}N^{-1/2}\right)=\emptyset
\]
implying that 
\[
\lim_{N\to\infty}\mathbb{P}\Big\{\sum_{i\in[k_{N}]}Z_{N,\beta,c_{N}}(\boldsymbol{\sigma}_{0}^{i})\leq Z_{N,\beta}\Big\}=1.
\]
Combined with (\ref{eq:thm1_1}) this completes the proof.\qed

\section{\label{sec:chaos}Proof of Theorem \ref{thm:temp_chaos}}

Define the overlap distribution 
\begin{equation}
\zeta_{N}\left(\cdot\right)=\mathbb{E}\left\{ M_{N,1}\otimes M_{N,2}\left\{ R\left(\boldsymbol{\sigma},\boldsymbol{\sigma}^{\prime}\right)\in\cdot\right\} \right\} \label{eq:zeta}
\end{equation}
and let $\mathcal{B}\left(q,\epsilon\right)\subset\mathbb{R}$ denote
the ball of radius $\epsilon$ centered at $q$. For $\beta_{1}$,
$\beta_{2}>0$ set 
\begin{equation}
q_{12}=q_{*}\left(\beta_{1}\right)q_{*}\left(\beta_{2}\right){\rm \,\,\,\, and\,\,\,}\mathcal{Q}=\left\{ 0,q_{12},\left(-1\right)^{p+1}q_{12}\right\} .\label{eq:q12Q}
\end{equation}
Theorem \ref{thm:temp_chaos} follows directly from the following
proposition, which also contains information about the overlap distribution.
\begin{prop}
\label{prop:temp_chaos}For large enough $\beta_{1},\beta_{2}>0$,
either different or equal, with $M_{N,i}=G_{N,\beta_{i}}$, there
exists $v>0$ such that the following holds. For any $\epsilon>0$,
\begin{align}
 & \mbox{any \ensuremath{q_{0}\in\mathcal{Q}}\,\ is charged}:\, & \liminf_{N\to\infty}\zeta_{N}\left(\mathcal{B}\left(q_{0},\epsilon\right)\right)>v,\label{eq:1402-1-1}\\
 & q\notin\mathcal{Q}\mbox{ vanish}: & \lim_{N\to\infty}\zeta_{N}\left(\left(\cup_{q\in\mathcal{Q}}\mathcal{B}\left(q,\epsilon\right)\right)^{c}\right)=0.\label{eq:1402-2-1}
\end{align}
\end{prop}
\begin{rem}
\label{rem:multipleoverlap}For even $p\geq4$, (\ref{eq:1402-2-1})
was already known -- it was proved by Panchenko and Talagrand \cite{multipleoverlap}. 
\end{rem}
In the next subsections we relate the overlap distribution (\ref{eq:zeta})
to critical points and prove Proposition \ref{prop:temp_chaos}.

\subsection{\label{sub:chcrtpts}Chaotic behavior and critical points}

This section is devoted to an auxiliary result which reduces questions
about chaos to ones about the behavior of the critical points and
values of the Hamiltonian $H_{N}\left(\boldsymbol{\sigma}\right)$
under perturbations. Suppose $H_{N}^{(1)}\left(\boldsymbol{\sigma}\right)$
and $H_{N}^{(2)}\left(\boldsymbol{\sigma}\right)$ are two copies
of the Hamiltonian $H_{N}\left(\boldsymbol{\sigma}\right)$ defined
on the same probability space, but not necessarily independent, and
let $G_{N,\beta}^{(1)}$ and $G_{N,\beta}^{(2)}$ denote the corresponding
Gibbs measures. Similarly, let $\boldsymbol{\sigma}_{0}^{i,(1)}$
and $\boldsymbol{\sigma}_{0}^{i,(2)}$ be the enumerations of the
critical points of $H_{N}^{(1)}\left(\boldsymbol{\sigma}\right)$
and $H_{N}^{(2)}\left(\boldsymbol{\sigma}\right)$, respectively,
and with $W_{i}^{(n)}=Z_{N,\beta_{n}}\left({\rm Band}_{i,\beta_{n}}^{(n)}\left(cN^{-1/2}\right)\right)$
(defined accordingly by (\ref{eq:B_sigma}) and (\ref{eq:rel partition}))
define 
\begin{equation}
\zeta_{N,c,k}^{{\rm crt}}\left(\cdot\right)\triangleq\mathbb{E}\left\{ \sum_{i,j\in[k]}W_{i}^{(1)}W_{j}^{(2)}\delta_{q_{12}R(\boldsymbol{\sigma}_{0}^{i,(1)},\boldsymbol{\sigma}_{0}^{j,(2)})}\left(\cdot\right)\right\} ,\label{eq:12037}
\end{equation}
where $q_{12}$ is given in (\ref{eq:q12Q}). Finally, let $d_{BL}$
denote the bounded Lipschitz metric (see e.g. \cite[Appendix D]{LDbook}).
\begin{lem}
\label{lem:choastocrtpts}Assume $\beta_{1},\,\beta_{2}$ are large
enough and let $M_{N,1}=G_{N,\beta_{1}}^{(1)}$ and $M_{N,2}=G_{N,\beta_{2}}^{(2)}$.
Then
\begin{equation}
\lim_{c,k\to\infty}\lim_{N\to\infty}d_{BL}\left(\zeta_{N}\left(\cdot\right),\,\zeta_{N,c,k}^{{\rm crt}}\left(\cdot\right)\right)=0.\label{eq:12033}
\end{equation}
\end{lem}
\begin{proof}
Define the conditional measure 
\[
G_{N,\beta_{n}}^{c,i,(n)}(\cdot)=G_{N,\beta_{n}}^{(n)}(\cdot\cap{\rm Band}_{i,\beta_{n}}^{(n)}\left(cN^{-1/2}\right))/G_{N,\beta_{n}}^{(n)}({\rm Band}_{i,\beta_{n}}^{(n)}\left(cN^{-1/2}\right)).
\]
Similarly to (\ref{eq:2803-2}) and the discussion surrounding it,
by Lemma \ref{lem:center} and using part (\ref{enu:Geometry2}) of
Theorem \ref{thm:Geometry}, for any fixed $i$ and $j$, for any
$\delta>0$, with probability that goes to $1$ as $N\to\infty$,
\[
G_{N,\beta_{1}}^{c,i,(1)}\otimes G_{N,\beta_{2}}^{c,j,(2)}\left\{ \left|R(\boldsymbol{\sigma},\boldsymbol{\sigma}^{\prime})-q_{12}R(\boldsymbol{\sigma}_{0}^{i,(1)},\boldsymbol{\sigma}_{0}^{j,(2)})\right|<\delta\right\} >1-\delta.
\]
The proof is completed by part (\ref{enu:Geometry1}) of Theorem \ref{thm:Geometry}.
\end{proof}

\subsection{\label{sub:pfproptmpch}Proof of Proposition \ref{prop:temp_chaos}}

The key to the proof is that both Gibbs measures $G_{N,\beta_{i}}$
are concentrated around the critical points of the same random field
$H_{N}(\boldsymbol{\sigma})$ which through (\ref{eq:12033}) determine
the overlap distribution. More precisely, in view of Lemmas \ref{lem:choastocrtpts},
\ref{cor:min} and \ref{cor:orth}, we have that (\ref{eq:1402-2-1})
holds and that in order to prove (\ref{eq:1402-1-1}) we only need
to show that there exist $\kappa>0$ and $a,\, a'\in(0,1)$ such that
for large enough $c$,
\begin{equation}
\left\{ \exists\boldsymbol{\sigma}_{1},\boldsymbol{\sigma}_{2}\in\mathscr{C}_{N}(m_{N}-\kappa,m_{N}+\kappa):\,\boldsymbol{\sigma}_{1}\neq\pm\boldsymbol{\sigma}_{2},\, G_{N,\beta_{i}}\left({\rm Band}_{\beta_{j},c}\left(\boldsymbol{\sigma}_{i}\right)\right)>a,\,\forall i,j=1,2\right\} ,\label{eq:12031}
\end{equation}
occurs with probability at least $a'$, for large enough $N$, with
\[
{\rm Band}_{\beta_{j},c}\left(\boldsymbol{\sigma}\right)={\rm Band}\left(\boldsymbol{\sigma},q_{*}(\beta_{j})-cN^{-1/2},q_{*}(\beta_{j})+cN^{-1/2}\right).
\]

Let $\epsilon>0$ be an arbitrarily small number and choose large
enough $\kappa=\kappa(\epsilon)$ such that the following three hold.
Firstly, for any large enough $c>0$, by (\ref{eq:1225-2}), for $j=1,\,2$,
\[
\limsup_{N\to\infty}\mathbb{E}\left|\left\{ \boldsymbol{\sigma}_{0}\in\mathscr{C}_{N}\left(m_{N},m_{N}+\kappa\right):\,\frac{Z_{N,\beta_{j}}\left({\rm Band}_{\beta_{j},c}\left(\boldsymbol{\sigma}\right)\right)}{\mathfrak{V}_{N,\beta_{j}}\left(m_{N}+\kappa\right)}<\ftail\left(\kappa\right)\right\} \right|<\epsilon,
\]
where $\ftail$ satisfies (\ref{eq:ftail}). Secondly, by Theorem
\ref{thm:free-energy}, (\ref{eq:V(u)}) and (\ref{eq:m_N}), for
some large enough $t=t(\kappa,\epsilon)$
\begin{equation}
\limsup_{N\to\infty}\mathbb{P}\left\{ NF_{N,\beta_{j}}-\log\left(\mathfrak{V}_{N,\beta_{j}}\left(m_{N}+\kappa\right)\right)>t\right\} <\epsilon.\label{eq:F-1}
\end{equation}
And thirdly, by Theorem \ref{thm:ext proc} 
\[
\liminf_{N\to\infty}\mathbb{P}\left\{ \mathscr{C}_{N}\left(m_{N},m_{N}+\kappa\right)<4\right\} <\epsilon.
\]
For this choice of parameters, we have that with probability at least
$a'=1-3\epsilon$, there exist $\boldsymbol{\sigma}_{1},\boldsymbol{\sigma}_{2}$
as in (\ref{eq:12031}) with the lower bound satisfied with $a=\ftail(\kappa)e^{-t}$
for large $c$. This completes the proof.\qed

\section{\label{sec:discussion}Concluding remarks}

\subsection*{The process of log-masses}

One can associate a point process to the log-masses of the bands (say
with some $c_{N}=o(N^{1/2})$, going to $\infty$), 
\[
\xi_{N,\beta}^{W}=\sum\delta_{W_{N,\beta}^{i}},\,\,\,\, W_{N,\beta}^{i}=\log\left(Z_{N,\beta}\left({\rm Band}_{i}\left(c_{N}N^{-1/2}\right)\right)\right).
\]
Similarly to the point process of extremal critical values (\ref{eq:xi_N}),
we predict that $(1+\iota_{p})^{-1}\xi_{N,\beta}^{W}$, shifted by
and appropriate term of $m_{N}^{W}=\log\left(\mathfrak{V}_{N,\beta}\left(m_{N}\right)\right)+O(1)$,
converges as $N\to\infty$ to a Poisson point process, however with
intensity $\exp\left\{ \frac{c_{p}x}{\beta q_{*}^{p}}\right\} dx$;
this would also determine the limiting law of fluctuations of the
free energy. Evidence for this can be found in Corollary \ref{cor:conditional laws}
and the proof of Corollary \ref{cor:FE_u_Hess_rep}, and through the
connection of Poisson point processes of exponential intensities,
the Poisson-Dirichlet distribution, and the distribution of pure state
masses; see \cite{PanchenkoBook,TalagrandPstates,JagannathApxUlt}.
The convergence of $\xi_{N,\beta}^{W}$ should be possible to prove
by establishing invariance under perturbations of the disorder and
using Liggett's characterization of shifted Poisson point processes
\cite{Liggett}, as done for the point process of critical points
in \cite{pspinext}. This will be studied in future work.

\subsection*{$\protect\Frf\left(E_{0},q_{*}\right)$ and the Parisi functional}

Theorem \ref{thm:free-energy} implies that $F_{N,\beta}$ converges
in distribution to $\Frf\left(E_{0},q_{*}\right)$ which must coincide
with the spherical variant of the Parisi formula \cite{pSPSG,Talag,Chen}.
We recall that pure spherical models are known to exhibit 1-step replica
symmetry breaking \cite[Proposition 2.2]{Talag}. The Parisi functional
corresponding to $m\delta_{0}+(1-m)\delta_{q}$ (see \cite[(1.16)]{multipleoverlap},
where in our case $\xi\left(q\right)=q^{p}$) is given by
\begin{equation}
P\left(m,q\right)=\frac{1}{2}\left(\beta^{2}\left(1-q^{p}+mq^{p}\right)+\log\left(1-q\right)+\frac{1}{m}\log\left(1+\frac{mq}{1-q}\right)\right).\label{eq:2501-1}
\end{equation}

Reassuringly, Theorem \ref{thm:Geometry} implies that the limiting
overlap distribution function (restricted to $\left[0,1\right]$)
is of the from $m\delta_{0}+(1-m)\delta_{q}$ with $q=q_{*}^{2}$.
Our prediction regarding the point process of log-masses suggests,
using the connection to the Poisson-Dirichlet distribution, that $m$
should be equal to $m_{*}=c_{p}/\left(\beta q_{*}^{p}\right)$. Therefore,
one would expect that 
\begin{equation}
\Frf\left(E_{0},q_{*}\right)=P\left(m_{*},q_{*}^{2}\right).\label{eq:a12}
\end{equation}
This can be verified by a direct calculation.%
\footnote{Substituting  (\ref{eq:2501-1}) and (\ref{eq:6}) into (\ref{eq:a12}),
canceling like-terms and dividing by $\beta q_{*}^{p}/2$, one finds
that (\ref{eq:a12}) holds if and only if 
\[
c_{p}+\frac{1}{c_{p}}\log\left(\frac{1-q_{*}^{2}\left(1-c_{p}/\left(\beta q_{*}^{p}\right)\right)}{1-q_{*}^{2}}\right)-2E_{0}+\beta pq_{*}^{p-2}\left(1-q_{*}^{2}\right)=0.
\]
Using (\ref{eq:c_p}), (\ref{eq:39}) and (\ref{eq:a39}), by straightforward
algebra, one can see that the left-hand side of the above is equal
to $2\Theta_{p}\left(-E_{0}\right)$ (see (\ref{eq:Theta_p})), which
by the definition of $E_{0}$ (\ref{eq:E0}), verifies (\ref{eq:a12}). %
}

\subsection*{Temperature chaos in mixed models}

Suppose $H_{N}^{{\rm m}}(\boldsymbol{\sigma})=\sum_{p\geq2}\gamma_{p}H_{N,p}(\boldsymbol{\sigma})$,
where $H_{N,p}(\boldsymbol{\sigma})$ are independent pure spherical
models. We recall that the mixed model $H_{N}^{{\rm m}}(\boldsymbol{\sigma})$
is called generic if and only if $\sum_{p\geq2}p^{-1}\Indicator\left(\gamma_{p}\neq0\right)=\infty$.
Panchenko \cite{PanchenkoChaos} proved that mixed even $p$-spin
Ising spin glass models (with or without external field) exhibit temperature
chaos. His methods are expected to apply to spherical spin glasses
as well. In contrast, we have seen that if $\gamma_{p_{0}}=c>0$ for
some $p_{0}\geq3$ and $\gamma_{p}=0$ for all $p\neq p_{0}$ then
chaos does not occur. The transition in chaotic behavior as soon as
one moves from a pure model to a generic model, even if $\gamma_{p}$,
$p\neq p_{0}$, are extremely small, may seem surprising. 

First moment calculations for the critical points as in Theorem \ref{thm:ABAC}
were carried out for the mixed case by Auffinger and Ben Arous in
\cite{ABA2}. Assuming that $\gamma_{p_{0}}=c$ for some $p_{0}\geq3$
and $\gamma_{p}$, $p\neq p_{0}$, are sufficiently small, we have
checked that the second moment computation of \cite{2nd} works for
$H_{N}^{{\rm m}}(\boldsymbol{\sigma})$ (using a certain truncation
argument), at least on the exponential level (i.e., a result similar
to \cite[Theorem 3]{2nd} can be proved). We expect that, under the
same assumption, the proof of the concentration result, Theorem \ref{thm:2ndConcentration},
and the convergence of the extremal process, Theorem \ref{thm:ext proc},
proved in \cite{2nd,pspinext}, also carry over to the mixed case.

One significant difference, however, between the pure and mixed models
is related to the decomposition of Section \ref{sec:Decomposition}.
In the mixed case, instead of just one variable, in the right-hand
side of (\ref{eq:EucderivativeJprime}) and (\ref{eq:euc1}) we have
a sum over $p$ of the expressions corresponding to the case of pure
$p$-spin, multiplied by $\gamma_{p}$. Consequently, upon conditioning
on $H_{N}^{{\rm m}}(\hat{\mathbf{n}})=u$, $\nabla H_{N}^{{\rm m}}(\hat{\mathbf{n}})=0$
and $\nabla^{2}H_{N}^{{\rm m}}(\hat{\mathbf{n}})=\mathbf{A}$, the
terms in the decomposition that correspond to $0$, $1$ and $2$
are not deterministic, in contrast to the pure case. Therefore, we
expect that in the mixed case it is possible that the contribution
related to critical values which are roughly equal to some $u=m_{N}+v$
is not monotone in $v$.

Suppose that the convergence of the extremal process $\xi_{N}$ of
(\ref{eq:xi_N}) and the fact that for large $\beta$ the Gibbs measure
is supported on bands around critical points with critical values
approximately $m_{N}$ (defined appropriately) do carry over to some
mixed model $H_{N}^{{\rm m}}(\boldsymbol{\sigma})$. Even if this
is case, the occurrence of chaos could be explained by the following
mechanism, related to the above: the bands that carry most of the
Gibbs mass correspond to critical values which are approximately equal
to $m_{N}+v_{N}(\beta)$ where $\lim_{N\to\infty}|v_{N}(\beta)-v_{N}(\beta')|=\infty$,
for large $\beta\neq\beta'$, and therefore the centers of the relevant
bands for $\beta$ are orthogonal to those corresponding to $\beta'$
(assuming an equivalent of Corollary \ref{cor:orth} holds for $H_{N}^{{\rm m}}(\boldsymbol{\sigma})$).
In the language of the extremal point process of critical values $\xi_{N}$
and the point process of log-masses $\xi_{N,\beta}^{W}$ the corresponding
picture is that the leading particles of $\xi_{N,\beta}^{W}$ correspond
(by being related to the same critical point of the Hamiltonian) to
points of $\xi_{N}$ of depth that depends on $\beta$.

\subsection*{Transition to disorder chaos of the Gibbs measure and ground state}

Let $H_{N}^{\prime}\left(\boldsymbol{\sigma}\right)$ be an i.i.d
copy of $H_{N}\left(\boldsymbol{\sigma}\right)$ and set, for $t\in\left[0,1\right]$,
\begin{equation}
H_{N,t}\left(\boldsymbol{\sigma}\right)=\left(1-t\right)H_{N}\left(\boldsymbol{\sigma}\right)+\sqrt{2t-t^{2}}H_{N}^{\prime}\left(\boldsymbol{\sigma}\right).\label{eq:Hnt}
\end{equation}
Denote the Gibbs measure of $H_{N,t}$ by $G_{N,t,\beta}$. For $\beta=\infty$
and odd $p$ let $G_{N,\infty}$ be the probability measure concentrated
at the global minimum point of $H_{N}\left(\boldsymbol{\sigma}\right)$
and define $G_{N,t,\infty}$ similarly. For even $p$, for which there
are two global minimum points a.s., assume that each is sampled with
probability $1/2$. We say that disorder chaos occurs if
\begin{equation}
\exists q_{0}\in\left[-1,1\right],\,\forall\epsilon>0:\,\,\,\lim_{N\to\infty}\mathbb{E}\left\{ M_{N,1}\otimes M_{N,2}\left\{ \left|R\left(\boldsymbol{\sigma},\boldsymbol{\sigma}^{\prime}\right)-q_{0}\right|>\epsilon\right\} \right\} =0\label{eq:disorder-1}
\end{equation}
holds with $M_{N,1}=G_{N,\beta}$ and $M_{N,2}=G_{N,t,\beta}$, for
any $t\in\left(0,1\right)$. Disorder chaos of the ground state, or
ground state chaos for short, is defined similarly with $\beta=\infty$.
Chen, Hsieh, Hwang and Sheu \cite{ChenHHS} proved disorder chaos
for spherical mixed (and in particular, pure as we consider) even
$p$-spin model, with or without external fields. For the same models,
Chen and Sen \cite{ChenSen} proved ground state chaos. See also \cite{ChenDisChaos}
and \cite{ChenVar} by Chen for related results for models with Ising
spins.

Once it is known that chaos occurs, it is natural to ask whether chaos
can be seen on a finer scale, i.e., when we take $t_{N}\to0$. In
this regard, we mention that Chatterjee \cite{ChatterjeeDisChaos}
proved for Ising mixed even $p$-spin models without an external field
and $\beta<\infty$ that (\ref{eq:disorder-1}) holds with $M_{N,1}=G_{N,\beta}$
and $M_{N,2}=G_{N,t_{N},\beta}$ with any $t_{N}$ such that $t_{N}\log N\to\infty$.
For the pure spherical models the following precise transition holds,
for large enough $\beta\in\left(0,\infty\right)$ or $\beta=\infty$: 
\begin{enumerate}
\item for any $s>0$, with $t_{N}=s/N$, (\ref{eq:disorder-1}) does not
hold with $M_{N,1}=G_{N,\beta}$ and $M_{N,2}=G_{N,t_{N},\beta}$;
\item in contrast, it does hold if $t_{N}=s_{N}/N=o(1)$ assuming $s_{N}\to\infty$
as $N\to\infty$.
\end{enumerate}
In Appendix II we prove the first of the two statements based on a
certain mixing property for the deepest critical values under perturbations
of the disorder, which relies on tools developed by Zeitouni and the
author in \cite{pspinext}. We also explain in Appendix II how the
second statement can be proved by an extension of certain results
of \cite{pspinext}. A full proof of this extension would require
a rewrite of a significant part of \cite{pspinext}. Instead of doing
so, we will refer to the original text and detail the required changes.

\section*{\label{sec:KR}Appendix I: the Kac-Rice formula and related auxiliary
results}

The Kac-Rice formula \cite[Theorem 12.1.1]{RFG} is a basic tool in
our analysis, allowing us to relate several quantities of interest
to the conditional probability $\mathbb{P}_{u,0}$. Below we prove
a number of simple derivatives of the formula; some using the results
of Section \ref{sub:crtpts} on critical points and values.

The Kac-Rice formula can be used to express the mean number of critical
points $\boldsymbol{\sigma}_{0}$ on the sphere at which the values
of some other fields belong to some target set. For us, those other
fields will be usually related to masses of bands around $\boldsymbol{\sigma}_{0}$.
The variant of the Kac-Rice formula that we use is \cite[Theorem 12.1.1]{RFG}.
In the notation of \cite[Theorem 12.1.1]{RFG} we will consider situations
where, with some function $g_{N}\left(\boldsymbol{\sigma}\right)$
\begin{equation}
M=\mathbb{S}^{N-1},\,\,\, f\left(\boldsymbol{\sigma}\right)=\nabla H_{N}\left(\boldsymbol{\sigma}\right),\,\,\, u=0\in\mathbb{R}^{N-1},\,\,\, h\left(\boldsymbol{\sigma}\right)=\left(H_{N}\left(\boldsymbol{\sigma}\right),g_{N}\left(\boldsymbol{\sigma}\right)\right),\label{eq:KRpar}
\end{equation}
where we recall that, with $E=(E_{i})_{i=1}^{N-1}$ being an orthonormal
frame field on the sphere, we denote 
\[
\nabla H_{N}\left(\hat{\mathbf{n}}\right)=\left(E_{i}H_{N}\left(\hat{\mathbf{n}}\right)\right)_{i=1}^{N-1}\mbox{\,\,\,\, and\,\,}\nabla^{2}H_{N}\left(\hat{\mathbf{n}}\right)=\left(E_{i}E_{j}H_{N}\left(\hat{\mathbf{n}}\right)\right)_{i,j=1}^{N-1}.
\]

The application of \cite[Theorem 12.1.1]{RFG} requires $h\left(\boldsymbol{\sigma}\right)$
to satisfy certain non-degeneracy conditions - namely, conditions
(a)-(g) in \cite[Theorem 12.1.1]{RFG}. We will say that $g_{N}\left(\boldsymbol{\sigma}\right)$
is tame if the conditions are satisfied and if $\{h\left(\boldsymbol{\sigma}\right)\}_{\boldsymbol{\sigma}}$
is a stationary random field. Using (\ref{eq:Hamiltonian}), the conditions
are easy to check in any case we will apply the formula and this will
be left to the reader. Let 
\begin{equation}
\omega_{N}=\frac{2\pi^{N/2}}{\Gamma\left(N/2\right)}\label{eq:omega_vol-1}
\end{equation}
denote the surface area of the $N-1$-dimensional unit sphere. The
following is obtained from a direct application of the formula.
\begin{lem}
\label{lem:KR}Suppose that $g_{N}\left(\boldsymbol{\sigma}\right)$
is tame and $D_{N}$ and $J_{N}$ are some intervals, then 
\begin{align}
 & \mathbb{E}\left|\left\{ \boldsymbol{\sigma}_{0}\in\mathscr{C}_{N}\left(J_{N}\right):\, g_{N}\left(\boldsymbol{\sigma}_{0}\right)\in D_{N}\right\} \right|=\omega_{N}\left(\left(N-1\right)\frac{p-1}{2\pi}\right)^{\frac{N-1}{2}}\nonumber \\
 & \times\int_{J_{N}}du\frac{1}{\sqrt{2\pi N}}e^{-\frac{u^{2}}{2N}}\mathbb{E}_{u,0}\left\{ \left|\det\left(\frac{\nabla^{2}H_{N}\left(\hat{\mathbf{n}}\right)}{\sqrt{p\left(p-1\right)\left(N-1\right)/N}}\right)\right|\mathbf{1}\Big\{ g_{N}\left(\hat{\mathbf{n}}\right)\in D_{N}\Big\}\right\} .\label{eq:genericKR}
\end{align}
\end{lem}
\begin{proof}
Applying \cite[Theorem 12.1.1]{RFG} with (\ref{eq:KRpar}) yields
the integral formula 
\begin{align}
 & \mathbb{E}\left|\left\{ \boldsymbol{\sigma}_{0}\in\mathscr{C}_{N}\left(J_{N}\right):\, g_{N}\left(\boldsymbol{\sigma}_{0}\right)\in D_{N}\right\} \right|=\int_{\mathbb{S}^{N-1}}d\nu\left(\boldsymbol{\sigma}\right)\phi_{\nabla H_{N}\left(\boldsymbol{\sigma}\right)}\left(0\right)\label{eq:q0306}\\
 & \times\mathbb{E}\left\{ \left.\left|\det\left(\nabla^{2}H_{N}\left(\boldsymbol{\sigma}\right)\right)\right|\mathbf{1}\Big\{ H_{N}\left(\boldsymbol{\sigma}\right)\in J_{N},\, g_{N}\left(\boldsymbol{\sigma}\right)\in D_{N}\Big\}\,\right|\,\nabla H_{N}\left(\boldsymbol{\sigma}\right)=0\right\} ,\nonumber 
\end{align}
where $\phi_{\nabla H_{N}\left(\boldsymbol{\sigma}\right)}\left(x\right)$
is the Gaussian density of $\nabla H_{N}\left(\boldsymbol{\sigma}\right)$
and $\nu$ is the standard measure on the sphere (not normalized).
Since the integrand above is independent of $\boldsymbol{\sigma}$,%
\footnote{If we apply the Kac-Rice formula to compute the expectation as is
(\ref{eq:q0306}) with the additional restriction that $\boldsymbol{\sigma}_{0}$
belongs to some measurable subset $B\subset\mathbb{S}^{N-1}$, then
what changes in the integral formula on the right-hand side of (\ref{eq:q0306})
is that the domain of integration $\mathbb{S}^{N-1}$ is replaced
by $B$. Thus, the corresponding integrand is a continuous Radon-Nikodym
derivative w.r.t the Lebesgue measure which, since $\left(H_{N}\left(\boldsymbol{\sigma}\right),g_{N}\left(\boldsymbol{\sigma}\right)\right)$
is stationary, is invariant to rotations of the sphere. It is therefore
a constant function.%
} we can replace the integral with the value of the integrand evaluated
at $\boldsymbol{\sigma}=\hat{\mathbf{n}}$ and multiply by $\omega_{N}N^{\frac{N-1}{2}}$,
the volume of $\mathbb{S}^{N-1}$. By Lemma \ref{lem:Hhat_expressions},
$\nabla H_{N}\left(\hat{\mathbf{n}}\right)\sim\mathcal{N}\left(0,p\mathbf{I}_{N-1}\right)$,
so that $\phi_{\nabla H_{N}\left(\boldsymbol{\sigma}\right)}\left(0\right)=(2\pi p)^{-\frac{N-1}{2}}$,
and $H_{N}\left(\hat{\mathbf{n}}\right)$ and $\nabla H_{N}\left(\hat{\mathbf{n}}\right)$
are independent. The proof is completed by conditioning on $H_{N}\left(\hat{\mathbf{n}}\right)$
in addition to $\nabla H_{N}\left(\boldsymbol{\sigma}\right)=0$ and
some calculus.
\end{proof}
Several derivatives of (\ref{eq:genericKR}) are of use to us. Their
proofs will involve the rate function $\Theta_{p}\left(E\right)$
of Theorem \ref{thm:ABAC} which we now define explicitly. Let $\nu^{*}$
denote the semicircle measure, the density of which with respect to
Lebesgue measure is 
\begin{equation}
\frac{d\nu^{*}}{dx}=\frac{1}{2\pi}\sqrt{4-x^{2}}\mathbf{1}_{\left|x\right|\leq2},\label{eq:semicirc}
\end{equation}
and define the function (see, e.g., \cite[Proposition II.1.2]{logpotential})
\begin{align}
\Omega(x) & \triangleq\int_{\mathbb{R}}\log\left|\lambda-x\right|d\nu^{*}\left(\lambda\right)\label{eq:Omega}\\
 & =\begin{cases}
\frac{x^{2}}{4}-\frac{1}{2}, & \mbox{ if }0\leq\left|x\right|\leq2,\\
\frac{x^{2}}{4}-\frac{1}{2}-\left[\frac{\left|x\right|}{4}\sqrt{x^{2}-4}-\log\left(\sqrt{\frac{x^{2}}{4}-1}+\frac{\left|x\right|}{2}\right)\right], & \mbox{ if }\left|x\right|>2.
\end{cases}\nonumber 
\end{align}
The exponential growth rate function of (\ref{eq:1st_mom}) is given
\cite{A-BA-C} by 
\begin{equation}
\Theta_{p}\left(E\right)=\begin{cases}
\frac{1}{2}+\frac{1}{2}\log\left(p-1\right)-\frac{E^{2}}{2}+\Omega\left(\gamma_{p}E\right), & \mbox{ if }u<0,\\
\frac{1}{2}\log\left(p-1\right), & \mbox{ if }u\geq0,
\end{cases}\label{eq:Theta_p}
\end{equation}
 where $\gamma_{p}=\sqrt{p/\left(p-1\right)}$.
\begin{lem}
\label{lem:15}Assume the conditions of Lemma \ref{lem:KR} with $J_{N}=NJ$
where $J\subset\left(-\infty,0\right)$ is a fixed finite interval.
Let $\varphi:\mathbb{R}\to\mathbb{R}$ be a continuous function and
assume that 
\begin{equation}
\limsup_{N\to\infty}\sup_{u\in J_{N}}\left\{ \frac{1}{N}\log\left(\mathbb{P}_{u,0}\left\{ g_{N}\left(\hat{\mathbf{n}}\right)\in D_{N}\right\} \right)-\varphi\left(\frac{u}{N}\right)\right\} \leq0.\label{eq:1224-1}
\end{equation}
Then 
\begin{equation}
\limsup_{N\to\infty}\frac{1}{N}\log\left(\mathbb{E}\left|\left\{ \boldsymbol{\sigma}_{0}\in\mathscr{C}_{N}\left(J_{N}\right):\, g_{N}\left(\boldsymbol{\sigma}_{0}\right)\in D_{N}\right\} \right|\right)\leq\sup_{E\in J}\left\{ \Theta_{p}\left(E\right)+\varphi\left(E\right)\right\} .\label{eq:71}
\end{equation}
\end{lem}
\begin{proof}
From (\ref{eq:genericKR}) and H{\"{o}}lder's inequality,
\begin{align}
 & \mathbb{E}\left|\left\{ \boldsymbol{\sigma}_{0}\in\mathscr{C}_{N}\left(J_{N}\right):\, g_{N}\left(\boldsymbol{\sigma}_{0}\right)\in D_{N}\right\} \right|\leq\omega_{N}\left(\left(N-1\right)\frac{p-1}{2\pi}\right)^{\frac{N-1}{2}}\label{eq:75}\\
 & \times\int_{J_{N}}du\frac{1}{\sqrt{2\pi N}}e^{-\frac{u^{2}}{2N}}\left(\mathbb{E}_{u,0}\left\{ \left|\det\left(\mathbf{V}\right)\right|^{a}\right\} \right)^{1/a}\left(\mathbb{P}_{u,0}\left\{ g_{N}\left(\hat{\mathbf{n}}\right)\in D_{N}\right\} \right)^{1/b},\nonumber 
\end{align}
where 
\begin{equation}
\mathbf{V}\triangleq\frac{\nabla^{2}H_{N}\left(\hat{\mathbf{n}}\right)}{\sqrt{p\left(p-1\right)\left(N-1\right)/N}},\label{eq:80}
\end{equation}
$a>1$ is an arbitrary number, and $b:=b(a)=a/\left(a-1\right)$. 

First, 
\begin{equation}
\lim_{N\to\infty}\omega_{N}\left(\left(N-1\right)\frac{p-1}{2\pi}\right)^{\frac{N-1}{2}}=\frac{1}{2}+\frac{1}{2}\log\left(p-1\right).\label{eq:1224-3}
\end{equation}
Second, by a change of variables $u\mapsto Nv$, 
\begin{equation}
\limsup_{N\to\infty}\frac{1}{N}\log\left(\int_{J_{N}}du\frac{1}{\sqrt{2\pi N}}e^{-\frac{u^{2}}{2N}}\exp\left\{ N\left(\Omega\left(\gamma_{p}\frac{u}{N}\right)+\varphi\left(\frac{u}{N}\right)\right)\right\} \right)\leq\sup_{x\in J}\left\{ \Omega\left(\gamma_{p}x\right)-\frac{x^{2}}{2}+\varphi\left(x\right)\right\} ,\label{eq:1224-2}
\end{equation}
where $\gamma_{p}=\sqrt{p/\left(p-1\right)}$ and $\Omega$ is defined
in (\ref{eq:Omega}). Thus, by (\ref{eq:Theta_p}), the lemma follows
if we show that for any $a>1$, 
\begin{equation}
\limsup_{N\to\infty}\sup_{u\in J_{N}}\left\{ \frac{1}{Na}\log\left(\mathbb{E}_{u,0}\left\{ \left|\det\mathbf{V}\right|^{a}\right\} \right)-\Omega\left(\gamma_{p}\frac{u}{N}\right)\right\} \leq0,\label{eq:74}
\end{equation}
and that 
\begin{equation}
\limsup_{a\to\infty}\limsup_{N\to\infty}\sup_{u\in J_{N}}\left\{ \frac{1}{Nb(a)}\log\left(\left(\mathbb{P}_{u,0}\left\{ g_{N}\left(\hat{\mathbf{n}}\right)\in D_{N}\right\} \right)\right)-\varphi\left(\frac{u}{N}\right)\right\} \leq0.\label{eq:a3}
\end{equation}
The inequality (\ref{eq:a3}) obviously follows from (\ref{eq:1224-1}),
since $\lim_{a\to\infty}b(a)=1$. 

From (\ref{eq:G}) and (\ref{eq:GMrel}), the conditional law of $\mathbf{V}$
under $\mathbb{P}_{u,0}\left\{ \cdot\right\} $ is identical to that
of 
\begin{equation}
\mathbf{M}_{u}:=\mathbf{M}_{u,N-1}\triangleq\mathbf{M}-\gamma_{p}\frac{u}{\sqrt{N\left(N-1\right)}}\mathbf{I},\label{eq:munewer}
\end{equation}
where $\mathbf{M}$ is a GOE matrix and $\mathbf{I}$ is the identity
matrix, both of dimension $N-1$. For any $0<\epsilon<1<\kappa$,
with $\lambda_{i}$ denoting the eigenvalues of $\mathbf{M}$ and
$\lambda_{*}\left(u\right)=\max_{i}\left|\lambda_{i}-\gamma_{p}\frac{u}{\sqrt{N\left(N-1\right)}}\right|$,

\begin{align}
\mathbb{E}\left\{ \left|\det\left(\mathbf{M}_{u}\right)\right|^{a}\right\}  & \leq\mathbb{E}\left\{ \exp\left\{ a\sum_{i}\log_{\epsilon}^{\kappa}\left(\left|\lambda_{i}-\gamma_{p}\frac{u}{\sqrt{N\left(N-1\right)}}\right|\right)\right\} \right\} \nonumber \\
 & +\mathbb{E}\left\{ \left(\lambda_{*}\left(u\right)\right)^{a\left(N-1\right)}\mathbf{1}\left\{ \left(\lambda_{*}\left(u\right)\right)\geq\kappa\right\} \right\} ,\label{eq:72}
\end{align}
where 
\begin{align}
\log_{\epsilon}^{\kappa}\left(x\right) & =\begin{cases}
\log\left(\epsilon\right) & \,\,\mbox{if }x\leq\epsilon,\\
\log x & \,\,\mbox{if }x\in\left(\epsilon,\kappa\right),\\
\log\kappa & \,\,\mbox{if }x\geq\kappa.
\end{cases}\label{eq:77}
\end{align}

From the upper bound on the maximal eigenvalue of \cite[Lemma 6.3]{BDG},
\begin{equation}
\lim_{\kappa\to\infty}\limsup_{N\to\infty}\frac{1}{N}\log\left(\mathbb{E}\left\{ \left(\lambda_{*}\left(u\right)\right)^{a\left(N-1\right)}\mathbf{1}\left\{ \left(\lambda_{*}\left(u\right)\right)\geq\kappa\right\} \right\} \right)=-\infty,\label{eq:73}
\end{equation}
uniformly for $u\in J_{N}$. The empirical measure of eigenvalues
of GOE matrices $L_{N}=\frac{1}{N-1}\sum_{i=1}^{N-1}\lambda_{i}$
satisfies a large deviation principle with speed $N^{2}$ and a good
rate function $J_{0}\left(\nu\right)$, in the space of Borel probability
measures on $\mathbb{R}$ equipped with the weak topology, which is
compatible with the Lipschitz bounded metric; see \cite[Theorem 2.1.1]{BAG97}.
The good rate function $J_{0}\left(\nu\right)$ satisfies $J_{0}\left(\nu\right)=0$
if and only if $\nu=\nu^{*}$ is the semicircle law (\ref{eq:semicirc}).
Combined with the fact that the functions $\log_{\epsilon}^{\kappa}\left(\left|\cdot-u'\right|\right)$
have the same Lipschitz constant and bound for all $u'\in\mathbb{R}$,
this implies that for the event 
\[
F\left(u,\delta\right)=\left\{ \left|\frac{1}{N-1}\sum_{i}\log_{\epsilon}^{\kappa}\left(\left|\lambda_{i}-\gamma_{p}\frac{u}{\sqrt{N\left(N-1\right)}}\right|\right)-\int\log_{\epsilon}^{\kappa}\left(\left|\lambda-\gamma_{p}\frac{u}{N}\right|\right)d\mu^{*}\left(\lambda\right)\right|>\delta\right\} ,
\]
we have for any $\delta>0$ some positive number $d\left(\delta\right)$
such that
\[
\limsup_{N\to\infty}\sup_{u\in J_{N}}\frac{1}{N^{2}}\log\left(\mathbb{P}\left\{ F\left(u,\delta\right)\right\} \right)<-d\left(\delta\right).
\]
By taking $\kappa$ large enough and $\epsilon$ small enough, combining
this with (\ref{eq:72}) and (\ref{eq:73})  proves (\ref{eq:74})
and completes the proof.
\end{proof}
For intervals $J_{N}$ of length $o\left(N\right)$ around $-NE_{0}$
the following is also useful for us.
\begin{lem}
\label{lem:16}Assume the conditions of Lemma \ref{lem:KR} with $J_{N}=\left(m_{N}-a_{N},m_{N}+a_{N}^{\prime}\right)$
for some sequences $a_{N},\, a_{N}^{\prime}=o\left(N\right)$. Let
$c_{p}$ be as defined in (\ref{eq:c_p}). Then for large enough $N$,
\begin{align}
 & \mathbb{E}\left|\left\{ \boldsymbol{\sigma}_{0}\in\mathscr{C}_{N}\left(J_{N}\right):\, g_{N}\left(\boldsymbol{\sigma}_{0}\right)\in D_{N}\right\} \right|\nonumber \\
 & \leq C\int_{J_{N}}du\cdot e^{c_{p}\left(u-m_{N}\right)}\left(\mathbb{P}_{u,0}\left\{ g_{N}\left(\hat{\mathbf{n}}\right)\in D_{N}\right\} \right)^{1/2},\label{eq:76}
\end{align}
where $C>0$ is an appropriate constant.\end{lem}
\begin{proof}
We shall use the notation (\ref{eq:80}) introduced in the proof of
Lemma \ref{lem:15}. Since the conditional law of $\mathbf{V}$ under
$\mathbb{P}_{u,0}\left\{ \cdot\right\} $ is identical to that of
the shifted GOE matrix (\ref{eq:munewer}), as a particular case of
\cite[Corollary 23]{2nd}, 
\begin{equation}
\left(\mathbb{E}_{u,0}\left\{ \left|\det\left(\mathbf{V}\right)\right|^{2}\right\} \right)^{1/2}\leq C\mathbb{E}_{u,0}\left\{ \left|\det\left(\mathbf{V}\right)\right|\right\} ,\label{eq:1225-3}
\end{equation}
 uniformly in $u\in J_{N}$, for some constant $C>0$. As in the proof
of Lemma \ref{lem:15}, (\ref{eq:75}) holds and therefore, taking
$a=2$ and using \ref{eq:1225-3}, we obtain that 
\begin{align}
 & \mathbb{E}\left|\left\{ \boldsymbol{\sigma}_{0}\in\mathscr{C}_{N}\left(J_{N}\right):\, g_{N}\left(\boldsymbol{\sigma}_{0}\right)\in D_{N}\right\} \right|\leq C\omega_{N}\left(\left(N-1\right)\frac{p-1}{2\pi}\right)^{\frac{N-1}{2}}\label{eq:75-1-1}\\
 & \int_{J_{N}}du\frac{1}{\sqrt{2\pi N}}e^{-\frac{u^{2}}{2N}}\mathbb{E}_{u,0}\left\{ \left|\det(\mathbf{V})\right|\right\} \left(\mathbb{P}_{u,0}\left\{ g_{N}\left(\hat{\mathbf{n}}\right)\in D_{N}\right\} \right)^{1/2},\nonumber 
\end{align}
By Lemma \ref{lem:KR},
\begin{align*}
 & \mathbb{E}\left|\mathscr{C}_{N}\left(J_{N}\right)\right|=\int_{J_{N}}du\omega_{N}\left(\left(N-1\right)\frac{p-1}{2\pi}\right)^{\frac{N-1}{2}}\frac{1}{\sqrt{2\pi N}}e^{-\frac{u^{2}}{2N}}\mathbb{E}_{u,0}\left\{ \left|\det(\mathbf{V})\right|\right\} .
\end{align*}
By definition, up to a factor of $2$ in the even $p$ case (related
to the normalization factor in (\ref{eq:xi_N})), the integrand above
is equal to the density of the intensity measure of the extremal point
process of critical points $\xi_{N}$, defined in (\ref{eq:xi_N}),
shifted by $m_{N}$. Thus, by \cite[Proposition 3]{pspinext} it converges
uniformly to $e^{c_{p}\left(u-m_{N}\right)}$ (see also Theorem \ref{thm:ext proc}
above). Combined with (\ref{eq:75-1-1}) this yields (\ref{eq:76}).
\end{proof}
The non-negative random variable 
\[
\sum_{\boldsymbol{\sigma}_{0}\in\mathscr{C}_{N}\left(J_{N}\right)}Z_{N,\beta}\left({\rm Band}\left(\boldsymbol{\sigma}_{0},q,q^{\prime}\right)\right)
\]
can be approximated by
\[
\sum_{\boldsymbol{\sigma}_{0}\in\mathscr{C}_{N}\left(J_{N}\right)}\sum_{i\leq k}t_{i}\mathbf{1}\left\{ Z_{N,\beta}\left({\rm Band}\left(\boldsymbol{\sigma}_{0},q,q^{\prime}\right)\right)\in\left[t_{i},t_{i+1}\right)\right\} ,
\]
where $\left[t_{i},t_{i+1}\right)$, $i\leq k$, form a finite partition
of $\left[0,\infty\right)$. Combining this with the monotone convergence
theorem and (\ref{eq:genericKR}) one obtains the following.
\begin{cor}
We have that
\begin{align}
 & \negthickspace\negthickspace\mathbb{E}\left\{ \sum_{\boldsymbol{\sigma}_{0}\in\mathscr{C}_{N}\left(J_{N}\right)}Z_{N,\beta}\left({\rm Band}\left(\boldsymbol{\sigma}_{0},q_{N},q_{N}^{\prime}\right)\right)\right\} =\omega_{N}\left(\left(N-1\right)\frac{p-1}{2\pi}\right)^{\frac{N-1}{2}}\nonumber \\
 & \times\int_{J_{N}}du\frac{1}{\sqrt{2\pi N}}e^{-\frac{u^{2}}{2N}}\mathbb{E}_{u,0}\left\{ \left|\det\left(\frac{\nabla^{2}H_{N}\left(\hat{\mathbf{n}}\right)}{\sqrt{p\left(p-1\right)\left(N-1\right)/N}}\right)\right|Z_{N,\beta}\left({\rm Band}\left(\hat{\mathbf{n}},q_{N},q_{N}^{\prime}\right)\right)\right\} .\label{eq:78}
\end{align}

\end{cor}
We have the following exponential bound on the expectation above.
\begin{lem}
\label{lem:17}Let $J_{N}=NJ$ where $J\subset\mathbb{R}$ is a fixed
interval. Let $\varphi:\mathbb{R}\to\mathbb{R}$ be a continuous function
and suppose that 
\begin{equation}
\limsup_{N\to\infty}\sup_{u\in J_{N}}\left\{ \frac{1}{N}\log\left(\mathbb{E}_{u,0}\left\{ Z_{N,\beta}\left({\rm Band}\left(\hat{\mathbf{n}},q_{N},q_{N}^{\prime}\right)\right)\right\} \right)-\varphi\left(\frac{u}{N}\right)\right\} \leq0.\label{eq:81}
\end{equation}
Then 
\begin{equation}
\limsup_{N\to\infty}\frac{1}{N}\log\left(\mathbb{E}\left\{ \sum_{\boldsymbol{\sigma}_{0}\in\mathscr{C}_{N}\left(J_{N}\right)}Z_{N,\beta}\left({\rm Band}\left(\boldsymbol{\sigma}_{0},q_{N},q_{N}^{\prime}\right)\right)\right\} \right)\leq\sup_{x\in J}\left\{ \Theta_{p}\left(x\right)+\varphi\left(x\right)\right\} ,\label{eq:71-1}
\end{equation}
\end{lem}
\begin{proof}
Abbreviate ${\rm Band}\left(\boldsymbol{\sigma}_{0}\right)={\rm Band}\left(\boldsymbol{\sigma}_{0},q_{N},q_{N}^{\prime}\right)$.
By H{\"{o}}lder's inequality, 
\[
b\mapsto\log\left(\mathbb{E}_{u,0}\left\{ \left(Z_{N,\beta}\left({\rm Band}\left(\hat{\mathbf{n}}\right)\right)\right)^{b}\right\} \right)
\]
is a convex function. Hence, for $b\in\left(1,2\right)$, for any
$u\in J_{N}$, 
\begin{align}
 & \negthickspace\negthickspace\frac{1}{Nb}\log\left(\mathbb{E}_{u,0}\left\{ \left(Z_{N,\beta}\left({\rm Band}\left(\hat{\mathbf{n}}\right)\right)\right)^{b}\right\} \right)\nonumber \\
 & \leq\frac{1}{Nb}\left(\left(2-b\right)\log\left(\mathbb{E}_{u,0}\left\{ Z_{N,\beta}\left({\rm Band}\left(\hat{\mathbf{n}}\right)\right)\right\} \right)+\left(b-1\right)N\bar{C}\right),\label{eq:a4}
\end{align}
where we define 
\[
\bar{C}:=2\beta^{2}+\frac{2\beta}{N}\sup_{u\in J_{N}}|u|=2\beta^{2}+\frac{2\beta}{N}\sup_{x\in J}|x|
\]
and use the fact that
\begin{align*}
 & \log\left(\mathbb{E}_{u,0}\left\{ \left(Z_{N,\beta}\left({\rm Band}\left(\hat{\mathbf{n}}\right)\right)\right)^{2}\right\} \right)\\
 & \leq\log\left(\max_{\boldsymbol{\sigma},\boldsymbol{\sigma}'\in\mathbb{S}^{N-1}}\mathbb{E}_{u,0}\left\{ e^{-\beta H_{N}\left(\boldsymbol{\sigma}\right)-\beta H_{N}\left(\boldsymbol{\sigma}'\right)}\right\} \right)\leq2\beta|u|+2\beta^{2}N.
\end{align*}

For any $b>1$, with $a:=a(b)=b/\left(b-1\right)$, we have, using
the notation (\ref{eq:80}), 
\begin{align}
 & \negthickspace\negthickspace\negthickspace\negthickspace\mathbb{E}_{u,0}\left\{ \left|\det\left(\mathbf{V}\right)\right|Z_{N,\beta}\left({\rm Band}\left(\hat{\mathbf{n}}\right)\right)\right\} \label{eq:z12}\\
 & \leq\left(\mathbb{E}_{u,0}\left\{ \left|\det\left(\mathbf{V}\right)\right|^{a}\right\} \right)^{1/a}\left(\mathbb{E}_{u,0}\left\{ \left(Z_{N,\beta}\left({\rm Band}\left(\hat{\mathbf{n}}\right)\right)\right)^{b}\right\} \right)^{1/b}.\nonumber 
\end{align}
Using (\ref{eq:a4}) and (\ref{eq:74}) and letting $b\searrow1$we
obtain
\[
\limsup_{N\to\infty}\sup_{u\in J_{N}}\left\{ \frac{1}{N}\log\left(\mathbb{E}_{u,0}\left\{ \left|\det\left(\mathbf{V}\right)\right|Z_{N,\beta}\left({\rm Band}\left(\hat{\mathbf{n}}\right)\right)\right\} \right)-\Omega\left(\gamma_{p}\frac{u}{N}\right)-\varphi\left(\frac{u}{N}\right)\right\} \leq0.
\]
From (\ref{eq:78}), (\ref{eq:1224-3}) and (\ref{eq:1224-2}), we
conclude that (\ref{eq:71-1}) follows.
\end{proof}
We finish with the following lemma
\begin{lem}
\label{lem:18}Let $J_{N}=\left(m_{N}-a_{N},m_{N}+a_{N}^{\prime}\right)$
with some sequences $a_{N},\, a_{N}^{\prime}=o\left(N\right)$, and
define $c_{p}$ as in (\ref{eq:c_p}). Then, for large enough $N$,
\begin{align}
 & \negthickspace\negthickspace\mathbb{E}\left\{ \sum_{\boldsymbol{\sigma}_{0}\in\mathscr{C}_{N}\left(J_{N}\right)}Z_{N,\beta}\left({\rm Band}\left(\boldsymbol{\sigma}_{0},q_{N},q_{N}^{\prime}\right)\right)\right\} \label{eq:1224-4}\\
 & \leq C\cdot\int_{J_{N}}du\cdot e^{c_{p}\left(u-m_{N}\right)}\left(\mathbb{E}_{u,0}\left\{ \left(Z_{N,\beta}\left({\rm Band}\left(\hat{\mathbf{n}},q_{N},q_{N}^{\prime}\right)\right)\right)^{2}\right\} \right)^{1/2},\nonumber 
\end{align}
where $C>0$ is an appropriate constant.\end{lem}
\begin{proof}
By (\ref{eq:78}), (\ref{eq:z12}) and (\ref{eq:1225-3}) it follows
that, for an appropriate $C>0$, for large $N$,
\begin{align*}
 & \mathbb{E}\left\{ \sum_{\boldsymbol{\sigma}_{0}\in\mathscr{C}_{N}\left(J_{N}\right)}Z_{N,\beta}\left({\rm Band}\left(\boldsymbol{\sigma}_{0},q_{N},q_{N}^{\prime}\right)\right)\right\} \leq C\cdot\omega_{N}\left(\left(N-1\right)\frac{p-1}{2\pi}\right)^{\frac{N-1}{2}}\\
 & \times\int_{J_{N}}du\frac{1}{\sqrt{2\pi N}}e^{-\frac{u^{2}}{2N}}\mathbb{E}_{u,0}\left\{ \left|\det(\mathbf{V})\right|\right\} \left(\mathbb{E}_{u,0}\left\{ \left(Z_{N,\beta}\left({\rm Band}\left(\hat{\mathbf{n}},q_{N},q_{N}^{\prime}\right)\right)\right)^{2}\right\} \right)^{1/2},
\end{align*}
where $\mathbf{V}$ is defined in (\ref{eq:80}). The proof now follows
from the argument used in the proof of Lemma \ref{lem:16} right after
(\ref{eq:75-1-1}).
\end{proof}

\section*{Appendix II: transition to disorder chaos}

The two statements with which we concluded the discussion about transition
to disorder chaos in Section \ref{sec:discussion} follow from the
two proposition below. Recall the definition (\ref{eq:zeta}) of $\zeta_{N}\left(\cdot\right)$
and, similarly to (\ref{eq:q12Q}), set in this section 
\begin{equation}
q_{12}=q_{*}\left(\beta\right)q_{*}\left(\beta\right){\rm \,\,\,\, and\,\,\,}\mathcal{Q}=\left\{ 0,q_{12},\left(-1\right)^{p+1}q_{12}\right\} .\label{eq:q12Q-1}
\end{equation}

\begin{prop}
\label{prop:disorder_chaos-1}For large enough $\beta\in\left(0,\infty\right)$
or $\beta=\infty$ and any $s>0$ we have the following. With $t_{N}=s/N$,
let $M_{N,1}=G_{N,\beta}$ be the Gibbs measure of $H_{N}\left(\boldsymbol{\sigma}\right)$
and $M_{N,2}=G_{N,t_{N},\beta}$ be the Gibbs measure of $H_{N,t_{N}}\left(\boldsymbol{\sigma}\right)$,
see (\ref{eq:Hnt}). Then there exists $v=v(s,\beta)$ such that for
any $\epsilon>0$,
\begin{align}
 & \mbox{any \ensuremath{q_{0}\in\mathcal{Q}}\,\ is charged}:\, & \liminf_{N\to\infty}\zeta_{N}\left(\mathcal{B}\left(q_{0},\epsilon\right)\right)>v,\label{eq:dis1}\\
 & q\notin\mathcal{Q}\mbox{ vanish}: & \lim_{N\to\infty}\zeta_{N}\left(\left(\cup_{q\in\mathcal{Q}}\mathcal{B}\left(q,\epsilon\right)\right)^{c}\right)=0.\label{eq:dis2}
\end{align}

\end{prop}

\begin{prop}
\label{prop:disorder_chaos-2}For large enough $\beta\in\left(0,\infty\right)$
or $\beta=\infty$ and any sequence $t_{N}=s_{N}/N=o(1)$ with $s_{N}\to\infty$
as $N\to\infty$, with $M_{N,1}=G_{N,\beta}$ and $M_{N,2}=G_{N,t_{N},\beta}$,
\begin{equation}
\forall\epsilon>0:\,\,\,\lim_{N\to\infty}\zeta_{N}\left(\mathcal{B}\left(0,\epsilon\right)\right)=1.\label{eq:13031}
\end{equation}

\end{prop}
The rest of the appendix is devoted to the proof of Proposition \ref{prop:disorder_chaos-1}
and a sketch of a proof of Proposition \ref{prop:disorder_chaos-2}.

\subsection*{\label{sub:pfch1}Proof of Proposition \ref{prop:disorder_chaos-1}}

The main tool we use in the proof is the description of the perturbations
of critical points and values under perturbations of the disorder
developed by Zeitouni and the author in \cite{pspinext}. Roughly
speaking, when perturbations of the disorder are as in the statement
of Proposition \ref{prop:disorder_chaos-1}: (i) the position of deep
critical points does not change on the macroscopic level, and (ii)
the change in corresponding critical values is of order $O(1)$. Hence,
typically the collections of critical points around which the original
and perturbed Gibbs measures are concentrated are the same, up to
microscopic changes in their positions. Relying on this, by Lemmas
\ref{lem:choastocrtpts} and \ref{cor:orth} we will identify the
atoms of the overlap distribution as $N\to\infty$ limit. 

Fix $s>0$ throughout the proof and let $t_{N}=s/N$. Set 
\begin{equation}
e_{N}:=e_{N}(s)=\frac{s}{N}(\frac{3}{2}-\frac{s}{N})/\left(1-\frac{s}{N}\right)^{2}=O(1/N)\label{eq:eN}
\end{equation}
and note that with $H_{N}^{\prime}\left(\boldsymbol{\sigma}\right)$
being an independent copy of $H_{N}\left(\boldsymbol{\sigma}\right)$
in the definition (\ref{eq:Hnt}) of $H_{N,t}\left(\boldsymbol{\sigma}\right)$
we have that 
\begin{equation}
\frac{N}{N-s}H_{N,t_{N}}\left(\boldsymbol{\sigma}\right)=H_{N,s,e_{N}}^{+}\left(\boldsymbol{\sigma}\right),\label{eq:H+}
\end{equation}
where 
\begin{equation}
H_{N,s,r}^{+}\left(\boldsymbol{\sigma}\right)\triangleq H_{N}\left(\boldsymbol{\sigma}\right)+\sqrt{\frac{2s}{N}(1+r)}H_{N}^{\prime}\left(\boldsymbol{\sigma}\right).\label{eq:H++}
\end{equation}
In \cite{pspinext} the critical points and values of $H_{N}^{+}\left(\boldsymbol{\sigma}\right)=H_{N,s,r}^{+}\left(\boldsymbol{\sigma}\right)$
with $s=1/2$ and $r_{N}=0$ were related to those of the original
model $H_{N}\left(\boldsymbol{\sigma}\right)$, for large $N$. By
a simple modification the proofs carry over to the more general case
with fixed $s$ and any $r_{N}=o(1)$, and this variant is what we
use below. Define $\mathscr{C}_{N,t}\left(B\right)$ similarly to
$\mathscr{C}_{N}\left(B\right)$ using the random field $H_{N,t}\left(\boldsymbol{\sigma}\right)$.
By \cite[Lemma 8]{pspinext} (modified), for $C_{0}=(E_{0}-c_{p})/2$
(see (\ref{eq:E0}) and (\ref{eq:c_p})) there exists a sequence $\delta_{N}>0$
converging to $0$ as $N\to\infty$, such that for any $\kappa>0$,
with probability that goes to $1$ as $N\to\infty$ there exists a
mapping $\mathscr{G}_{N,\kappa}:\mathscr{C}_{N}\left(m_{N}-\kappa,m_{N}+\kappa\right)\to\mathscr{C}_{N,t_{N}}\left(\mathbb{R}\right)$
such that, if we denote $\boldsymbol{\sigma}'=\mathscr{G}_{N,\kappa}(\boldsymbol{\sigma})$,
then for any $\boldsymbol{\sigma}\in\mathscr{C}_{N}\left(m_{N}-\kappa,m_{N}+\kappa\right)$,
\begin{align}
\mbox{ val. pert.:\,\,} & \left|\frac{N}{N-s}H_{N,t_{N}}\left(\boldsymbol{\sigma}'\right)-\left(H_{N}\left(\boldsymbol{\sigma}\right)+\sqrt{\frac{2s}{N}}H_{N}^{\prime}\left(\boldsymbol{\sigma}\right)-2sC_{0}\right)\right|\leq\delta_{N},\label{eq:12035}\\
\mbox{ loc. pert.:\,\,} & R\left(\boldsymbol{\sigma},\boldsymbol{\sigma}'\right)\geq1-\delta_{N}.\nonumber 
\end{align}
Moreover, from \cite[Lemma 10]{pspinext} (modified) for any $\kappa'>0$,
\begin{equation}
\lim_{\kappa\to\infty}\liminf_{N\to\infty}\mathbb{P}\left\{ \mathscr{C}_{N,t_{N}}\left(m_{N}-\kappa',m_{N}+\kappa'\right)\subset\mathscr{G}_{N,\kappa}\left(\mathscr{C}_{N}\left(m_{N}-\kappa,m_{N}+\kappa\right)\right)\right\} =1.\label{eq:12036}
\end{equation}
We note that $H_{N}^{\prime}$ and $\mathscr{C}_{N}\left(m_{N}-\kappa,m_{N}+\kappa\right)$
are independent. Thus, from (\ref{eq:12035}), Corollary \ref{cor:min}
and an application of the union bound to bound $|H_{N}^{\prime}(\boldsymbol{\sigma}_{0})|$
uniformly in $\boldsymbol{\sigma}_{0}\in\mathscr{C}_{N}\left(m_{N}-\kappa,m_{N}+\kappa\right)$,
we also have
\begin{equation}
\lim_{\kappa\to\infty}\liminf_{N\to\infty}\mathbb{P}\left\{ \mathscr{C}_{N,t_{N}}\left(m_{N}-\kappa,m_{N}+\kappa\right)\supset\mathscr{G}_{N,\kappa}\left(\mathscr{C}_{N}\left(m_{N}-\kappa',m_{N}+\kappa'\right)\right)\right\} =1.\label{eq:12036-1}
\end{equation}

Denote by $\boldsymbol{\sigma}_{0,t_{N}}^{i}$ the critical points
of $H_{N,t_{N}}\left(\boldsymbol{\sigma}\right)$, similarly to $\boldsymbol{\sigma}_{0}^{i}$.
Let $\rho>0$ be an arbitrary number and fix an integer $i\neq0$,
assumed to be positive if $p$ is odd. From the above, we have that
for large enough real $\kappa>0$ and integer $k\geq1$, for large
$N$ with probability at least $1-\rho$, there exists $\mathscr{G}_{N,\kappa}$
as above, $\boldsymbol{\sigma}_{0}^{i}\in\mathscr{C}_{N}\left(m_{N}-\kappa,m_{N}+\kappa\right)$,
$\mathscr{G}_{N,\kappa}\left(\boldsymbol{\sigma}_{0}^{i}\right)=\boldsymbol{\sigma}_{0,t_{N}}^{j}$
for some $j\in[k]$, and $R(\boldsymbol{\sigma}_{0}^{i},\boldsymbol{\sigma}_{0,t_{N}}^{j})\geq1-\delta_{N}$.
By Corollaries \ref{cor:min} and \ref{cor:orth}, for large $N$
with probability at least $1-\rho$, for any $l\in[k]$, $l\neq\pm j$,
$|R(\boldsymbol{\sigma}_{0,t_{N}}^{l},\boldsymbol{\sigma}_{0,t_{N}}^{j})|\leq\delta_{N}^{\prime}$,
for some sequence $\delta_{N}^{\prime}=o(1)$. Thus, by Theorem \ref{thm:Geometry},
applied both to $H_{N}\left(\boldsymbol{\sigma}\right)$ and $H_{N,t_{N}}\left(\boldsymbol{\sigma}\right)$
(which are identically distributed), and by Lemma \ref{lem:choastocrtpts}
we have that (\ref{eq:dis2}) holds, in the case of large finite $\beta$.
By similar arguments to those in the proof of Proposition \ref{prop:temp_chaos},
(\ref{eq:dis1}) can also be proved.

Assuming $p$ is odd, let $\boldsymbol{\sigma}_{*}$ and $\boldsymbol{\sigma}_{*}^{t_{N}}$
be the global minimum of $H_{N}\left(\boldsymbol{\sigma}\right)$
and of $H_{N,t_{N}}\left(\boldsymbol{\sigma}\right)$, respectively.
From the above and Corollaries \ref{cor:min} and \ref{cor:orth},
two conclusions follow. First, for any $\delta>0$, 
\[
\lim_{N\to\infty}\mathbb{P}\left\{ R\left(\boldsymbol{\sigma}_{*},\boldsymbol{\sigma}_{*}^{t_{N}}\right)\in(-\delta,\delta)\cup(1-\delta,1)\right\} =1.
\]
Second, using Theorem \ref{thm:ext proc} and the fact that the Hamiltonian
$H_{N}^{\prime}\left(\boldsymbol{\sigma}\right)$ is independent of
$\mathscr{C}_{N}\left(\mathbb{R}\right)$, there exists $\rho>0$
small enough such that both the probability that $R\left(\boldsymbol{\sigma}_{*},\boldsymbol{\sigma}_{*}^{t_{N}}\right)\geq1-\delta$
and the probability that $\left|R\left(\boldsymbol{\sigma}_{*},\boldsymbol{\sigma}_{*}^{t_{N}}\right)\right|\leq\delta$
are larger than $\rho$. This completes the proof of Proposition \ref{prop:disorder_chaos-1}
for the case where $\beta=\infty$ and $p$ is odd. The case with
even $p$ follows similarly, by taking into account the fact that
$H_{N}\left(\boldsymbol{\sigma}\right)=H_{N}\left(-\boldsymbol{\sigma}\right)$.\qed

\subsection*{Sketch of proof of Proposition \ref{prop:disorder_chaos-2}}

As in the proof of Proposition \ref{prop:disorder_chaos-1}, here
too the perturbations of critical points and values will play an important
role. However, in contrast to the setting of Proposition \ref{prop:disorder_chaos-1},
perturbations of the critical values will diverge as $N\to\infty$
in the current setting.

Suppose that $H_{N,t}\left(\boldsymbol{\sigma}\right)$ is defined
by (\ref{eq:Hnt}) with $H_{N}^{\prime}\left(\boldsymbol{\sigma}\right)$
being an independent copy of $H_{N}\left(\boldsymbol{\sigma}\right)$.
Denote by $\mathscr{C}_{N,t}\left(\mathbb{R}\right)$ the set of critical
points of $H_{N,t}(\boldsymbol{\sigma})$ and set $C_{0}=(E_{0}-c_{p})/2$
(see (\ref{eq:E0}) and (\ref{eq:c_p})). Below we will explain how
the proof of Lemma 8 of \cite{pspinext} can be modified to obtain
the following. 
\begin{lem}
\label{lem:44}Suppose $s_{N}>0$ is a sequence such that $s_{N}=o(N)$
and $s_{N}\to\infty$ as $N\to\infty$ and set $t_{N}=s_{N}/N$. Then
there exists a sequence $\delta_{N}>0$ converging to $0$ as $N\to\infty$,
such that for any $\kappa>0$, with probability that goes to $1$
as $N\to\infty$, there exists a mapping $\mathscr{G}_{N,\kappa}:\mathscr{C}_{N}\left(m_{N}-\kappa,m_{N}+\kappa\right)\to\mathscr{C}_{N,t_{N}}\left(\mathbb{R}\right)$
such that, denoting $\boldsymbol{\sigma}'=\mathscr{G}_{N,\kappa}(\boldsymbol{\sigma})$,
for any $\boldsymbol{\sigma}\in\mathscr{C}_{N}\left(m_{N}-\kappa,m_{N}+\kappa\right)$
we have that

\begin{align}
\mbox{ val. pert.:\,\,} & \left|\frac{N}{N-s_{N}}H_{N,t_{N}}\left(\boldsymbol{\sigma}'\right)-\left(H_{N}\left(\boldsymbol{\sigma}\right)+\sqrt{\frac{2s_{N}}{N}}H_{N}^{\prime}\left(\boldsymbol{\sigma}\right)-2s_{N}C_{0}\right)\right|\leq\delta_{N}s_{N},\label{eq:12039}\\
\mbox{ loc. pert.:\,\,} & R\left(\boldsymbol{\sigma},\boldsymbol{\sigma}'\right)\geq1-\delta_{N}.\label{eq:120310}
\end{align}

\end{lem}
If for some $a\in(0,1/2)$, $\boldsymbol{\sigma}$ and $\boldsymbol{\sigma}'$
are two points that satisfy (\ref{eq:12039}) and 
\begin{equation}
\left|H_{N}\left(\boldsymbol{\sigma}\right)-m_{N}\right|,\left|N^{-1/2}H_{N}^{\prime}\left(\boldsymbol{\sigma}\right)\right|<s_{N}^{a},\label{eq:o2}
\end{equation}
then 
\begin{align*}
H_{N,t_{N}}\left(\boldsymbol{\sigma}'\right) & =\left(1-\frac{s_{N}}{N}\right)\left(H_{N}\left(\boldsymbol{\sigma}\right)+\sqrt{\frac{2s_{N}}{N}}H_{N}^{\prime}\left(\boldsymbol{\sigma}\right)-2s_{N}C_{0}\right)+o(s_{N})\\
 & =m_{N}+\left(E_{0}-2C_{0}\right)s_{N}+o(s_{N})=m_{N}+c_{p}s_{N}+o(s_{N}),
\end{align*}
where we used the fact that $m_{N}/N\to-E_{0}$. Thus, from Corollary
\ref{cor:min} and the fact that $H_{N}^{\prime}$ and $H_{N}$ are
independent, (\ref{eq:o2}) holds for any $\boldsymbol{\sigma}\in\mathscr{C}_{N}\left(m_{N}-\kappa,m_{N}+\kappa\right)$
with high probability, and for any $k\geq1$, 
\begin{align}
\lim_{\kappa\to\infty}\lim_{N\to\infty}\mathbb{P}\left\{ \forall i\in[k]:\,\boldsymbol{\sigma}_{0}^{i}\in\mathscr{C}_{N}\left(m_{N}-\kappa,m_{N}+\kappa\right),\vphantom{\left|H_{N,t}\left(\mathscr{G}_{N,\kappa}(\boldsymbol{\sigma}_{0}^{i})\right)-m_{N}-c_{p}s_{N}\right|}\right.\nonumber \\
\left.\left|H_{N,t_{N}}\left(\mathscr{G}_{N,\kappa}(\boldsymbol{\sigma}_{0}^{i})\right)-m_{N}-c_{p}s_{N}\right|\leq\tau s_{N}\right\}  & =1,\label{eq:a10}
\end{align}
where $\tau\in(0,c_{p})$ is an arbitrary constant.

Now, let $\epsilon'>0$, fix an arbitrary $k\geq1$ and let $i,\, j\in[k]$.
Let $\boldsymbol{\sigma}_{0,t_{N}}^{i}$ denote the critical points
of $H_{N,t_{N}}\left(\boldsymbol{\sigma}\right)$, defined similarly
to $\boldsymbol{\sigma}_{0}^{i}$. From the above, for large enough
$\kappa>0$, for large $N$ with probability at least $1-\epsilon'$,
there exists $\mathscr{G}_{N,\kappa}$ as above, $\boldsymbol{\sigma}_{0}^{i}\in\mathscr{C}_{N}\left(m_{N}-\kappa,m_{N}+\kappa\right)$,
and $H_{N,t_{N}}\left(\mathscr{G}_{N,\kappa}(\boldsymbol{\sigma}_{0}^{i})\right)\geq m_{N}+(c_{p}-\tau)s_{N}$,
and $R(\boldsymbol{\sigma}_{0}^{i},\mathscr{G}_{N,\kappa}(\boldsymbol{\sigma}_{0}^{i}))\geq1-\delta_{N}$
(where we recall that $c_{p}-\tau>0$). By Corollaries \ref{cor:min}
and \ref{cor:orth}, since for large $N$, $(c_{p}-\tau)s_{N}>\kappa$
and since $s_{N}=o(1)$, we have that with probability at least $1-\epsilon$,
$|R(\boldsymbol{\sigma}_{0,t_{N}}^{j},\mathscr{G}_{N,\kappa}(\boldsymbol{\sigma}_{0}^{i}))|\leq\delta_{N}^{\prime}$,
where $\delta_{N}^{\prime}>0$ is some sequence with $\delta_{N}^{\prime}=o(1)$.
Therefore, with probability at least $1-2\epsilon'$, for large $N$,
$|R(\boldsymbol{\sigma}_{0,t_{N}}^{j},\boldsymbol{\sigma}_{0}^{i})|\leq\delta_{N}^{\prime\prime}$,
for some sequence $\delta_{N}^{\prime\prime}>0$ with $\delta_{N}^{\prime\prime}=o(1)$.
By Theorem \ref{thm:Geometry} and Lemma \ref{lem:choastocrtpts},
(\ref{eq:13031}) follows. What remains is to explain how the proof
of \cite[Lemma 8]{pspinext} can be modified to prove Lemma \ref{lem:44}.

\subsubsection*{Sketch of proof of Lemma \ref{lem:44}}

Suppose $s_{N}^{\prime}>0$ is an arbitrary sequence such that $s_{N}^{\prime}=o(N)$
and $s_{N}^{\prime}\to\infty$ as $N\to\infty$ and set $t_{N}^{\prime}=s_{N}^{\prime}/N$.
Define $H_{N,s,r}^{+}\left(\boldsymbol{\sigma}\right)$ as in (\ref{eq:H++})
and note that $H_{N,s(1+r),0}^{+}\left(\boldsymbol{\sigma}\right)=H_{N,s,r}^{+}\left(\boldsymbol{\sigma}\right)$.
Therefore, setting $s_{N}=s_{N}^{\prime}(1+e_{N})$ where $e_{N}$
are defined in (\ref{eq:eN}), by (\ref{eq:H+}) we have that 
\begin{equation}
\frac{N}{N-s_{N}^{\prime}}H_{N,t_{N}^{\prime}}\left(\boldsymbol{\sigma}\right)=H_{N,s_{N},0}^{+}\left(\boldsymbol{\sigma}\right)=H_{N}\left(\boldsymbol{\sigma}\right)+\sqrt{\frac{2s_{N}}{N}}H_{N}^{\prime}\left(\boldsymbol{\sigma}\right)=:H_{N,s_{N}}^{+}\left(\boldsymbol{\sigma}\right).\label{eq:o1}
\end{equation}
Denote the set of critical points of $H_{N,s_{N}}^{+}\left(\boldsymbol{\sigma}\right)$
by $\mathscr{C}_{N,s_{N}}^{+}\left(\mathbb{R}\right)$ and note that
it coincides with the set of critical points of $H_{N,t_{N}^{\prime}}\left(\boldsymbol{\sigma}\right)$.
Thus, from (\ref{eq:o1}), by increasing $\delta_{N}$ if needed,
we have that in order to prove Lemma \ref{lem:44}, it is enough to
show the following. 
\begin{lem}
\label{lem:45}For any $s_{N}>0$ such that $s_{N}=o(N)$ and $s_{N}\to\infty$
as $N\to\infty$, we have the following. There exists a sequence $\delta_{N}>0$
converging to $0$ as $N\to\infty$, such that for any $\kappa>0$,
with probability that goes to $1$ as $N\to\infty$, there exists
a mapping $\mathscr{G}_{N,\kappa}:\mathscr{C}_{N}\left(m_{N}-\kappa,m_{N}+\kappa\right)\to\mathscr{C}_{N,s_{N}}^{+}\left(\mathbb{R}\right)$
such that, denoting $\boldsymbol{\sigma}'=\mathscr{G}_{N,\kappa}(\boldsymbol{\sigma})$,
for any $\boldsymbol{\sigma}\in\mathscr{C}_{N}\left(m_{N}-\kappa,m_{N}+\kappa\right)$
we have that
\begin{align}
\mbox{ val. pert.:\,\,} & \left|H_{N,s_{N}}^{+}\left(\boldsymbol{\sigma}'\right)-\left(H_{N}\left(\boldsymbol{\sigma}\right)+\sqrt{\frac{2s_{N}}{N}}H_{N}^{\prime}\left(\boldsymbol{\sigma}\right)-2s_{N}C_{0}\right)\right|\leq\delta_{N}s_{N},\label{eq:12039-3}\\
\mbox{ loc. pert.:\,\,} & R\left(\boldsymbol{\sigma},\boldsymbol{\sigma}'\right)\geq1-\delta_{N}.\label{eq:120310-3}
\end{align}

\end{lem}
In \cite[Lemma 8]{pspinext} the statement of Lemma \ref{lem:45}
is proved with $H_{N}^{+}\left(\boldsymbol{\sigma}\right):=H_{N,1/2}^{+}\left(\boldsymbol{\sigma}'\right)$,
i.e., with constant $s_{N}=1/2$. We now sketch how the proof of \cite[Lemma 8]{pspinext}
can be modified to obtain Lemma \ref{lem:45}. In the proof of \cite[Lemma 8]{pspinext},
`good' critical points are defined by several conditions related to
the perturbed $H_{N}^{+}\left(\boldsymbol{\sigma}\right)$ and unperturbed
fields $H_{N}\left(\boldsymbol{\sigma}\right)$ and approximations
of them (these are defined through the functions $g_{i}(\boldsymbol{\sigma})$
and sets $B_{i}$). In \cite[Lemmas 11, 12]{pspinext} it is proved
that the probability that all critical points in $\mathscr{C}_{N}\left(m_{N}-\kappa,m_{N}+\kappa\right)$
are good goes to $1$ as $N\to\infty$ and that on this event $\mathscr{G}_{N,\kappa}$
can be defined with the desired properties. The structure of the proof
for the modified case with $H_{N,s_{N}}^{+}\left(\boldsymbol{\sigma}\right)$
instead of $H_{N,1/2}^{+}\left(\boldsymbol{\sigma}\right)$ and all
the main arguments in the proof remain the same. What needs to be
modified are the definitions of the functions $g_{i}(\boldsymbol{\sigma})$
and sets $B_{i}$ through which good points are defined and several
inequalities involving them. The functions $g_{i}(\boldsymbol{\sigma})$
depend on additional functions and we also explain how these should
be modified. 

First, anywhere $\bar{f}_{\boldsymbol{\sigma}}^{\prime}(x)$, $\bar{f}_{\boldsymbol{\sigma}}^{\left(1\right)}\left(x\right)$
or $\bar{f}_{\boldsymbol{\sigma},lin}^{\prime}\left(x\right)$ appear
in the definitions of $g_{i}(\boldsymbol{\sigma})$ in \cite[Section 7.3]{pspinext}
and the definitions of $\bar{f}_{\boldsymbol{\sigma},apx}^{+}(x)$,
$Y_{\boldsymbol{\sigma}}$, $V_{\boldsymbol{\sigma}}$ and $\Delta_{\boldsymbol{\sigma}}$
in \cite[Section 7.2]{pspinext}, all of them should be multiplied
by $\sqrt{2s_{N}}$. Additionally, in the definitions of $g_{i}(\boldsymbol{\sigma})$
with $i=3,4$, the terms involving the trace should be multiplied
by $2s_{N}$. Lastly, the radii of balls and spheres in the suprema
and infima in the definitions of $g_{i}(\boldsymbol{\sigma})$ with
$i=5,...,8$ should be changed from $N^{-\alpha}$ to $w_{0}N^{-1/2}\sqrt{s_{N}}$,
where we assume that $w_{0}>0$ is some constant such that $2^{-3/2}K_{p,\delta}w_{0}^{2}-w_{0}>0$
where $K_{p,\delta}$ is the original constant in the definition of
$B_{7}$. After those changes $g_{i}(\boldsymbol{\sigma})$ are defined
appropriately. 

The sets $B_{i}$ should be modified as follows. Let $\hat{s}_{N}>0$
be some sequence such that $s_{N}=o(\hat{s}_{N})$ and $(\hat{s}_{N})^{3/2}=o(N^{1/2}s_{N})$
(and therefore, in particular $\hat{s}_{N}=o(N)$). Let $\mathcal{B}_{N}\left(\epsilon\right)\subset\mathbb{R}^{N}$
denote the Euclidean ball centered at $0$ with radius $\epsilon$.
Set $B_{2}=\mathcal{B}_{N-1}(N^{-1/2}\sqrt{\hat{s}_{N}})$, $B_{3}=\mathcal{B}_{1}\left(N^{-1/2}\hat{s}_{N}\right)$,
$B_{4}=s_{N}\left(C_{0}-\tau_{\delta}(N),C_{0}+\tau_{\delta}(N)\right)$,
$B_{5}=\mathcal{B}_{1}\left(CN^{-1/2}\hat{s}_{N}\sqrt{s_{N}}\right)$,
$B_{6}=\mathcal{B}_{1}\left(CN^{-1/2}(\hat{s}_{N})^{3/2}\right)$,
$B_{7}=\left(\frac{1}{4}K_{p,\delta}w_{0}^{2}\hat{s}_{N},\infty\right)$,
$B_{8}=\left(-\sqrt{\hat{s}_{N}s_{N}}w_{0},\infty\right)$, where
$\tau_{\delta}(N)$ is a sequence which is assumed to be large as
we need in the proofs but that still goes to $0$ as $N\to\infty$,
$\delta$ is small enough such that $K_{p,\delta}$ defined in \cite[Eq. (7.10)]{pspinext}
is positive, and $C>0$ is a constant that is assumed to be large
whenever needed in the proofs.

Once all the changes above have been applied, we can define the good
critical points of $H_{N}\left(\boldsymbol{\sigma}\right)$ as critical
points $\boldsymbol{\sigma}_{0}$ such that $g_{i}(\boldsymbol{\sigma}_{0})\in B_{i}$
for any $i=1,...,8$, exactly like in the original proof of \cite[Lemma 8]{pspinext}.
Then, an equivalent of \cite[Lemmas 11, 12]{pspinext} needs to proved;
i.e., we need to show for $\kappa>0$ fixed that with probability
going to $1$ as $N\to\infty$ all points in $\mathscr{C}_{N}\left(m_{N}-\kappa,m_{N}+\kappa\right)$
are good, and that on this event $\mathscr{G}_{N,\kappa}$ can be
defined appropriately. The proof of \cite[Lemma 12]{pspinext} follows
from \cite[Lemmas 14, 15, 16]{pspinext}. The statements of \cite[Lemmas 14, 15, 16]{pspinext}
in their original form, but with the modified $g_{i}(\boldsymbol{\sigma})$
and $B_{i}$, imply that for $\kappa>0$ fixed that with probability
going to $1$ as $N\to\infty$ all points in $\mathscr{C}_{N}\left(m_{N}-\kappa,m_{N}+\kappa\right)$
are good. The proofs of \cite[Lemmas 14, 15, 16]{pspinext} with the
modified $g_{i}(\boldsymbol{\sigma})$ and $B_{i}$ require minor
changes from the the original case. Those changes are only in formulas
and all the principal arguments should be left as they are. For example,
in the case involving $g_{2}(\boldsymbol{\sigma})$ in the proof of
\cite[Lemma 16]{pspinext} one needs to replace the expressions 
\[
\left\Vert Y_{\mathbf{n}}\right\Vert \leq\frac{c}{N}\left\Vert \nabla\bar{f}_{\mathbf{n}}^{\prime}\left(0\right)\right\Vert \,\:\mbox{and}\,\,\mathbb{P}\left\{ Q_{N-1}\geq N^{1-\alpha}\frac{c}{\sqrt{p}}\right\} 
\]
by
\[
\left\Vert Y_{\mathbf{n}}\right\Vert \leq\frac{c\sqrt{2s_{N}}}{N}\left\Vert \nabla\bar{f}_{\mathbf{n}}^{\prime}\left(0\right)\right\Vert \,\:\mbox{and}\,\,\mathbb{P}\left\{ Q_{N-1}\geq N^{-1/2}\sqrt{\frac{s_{N}^{\prime}}{s_{N}}}\frac{c}{\sqrt{p}}\right\} ,
\]
but other that this the proof remains the same.

Lastly, we point out what needs to be changed in the proof of \cite[Lemma 11]{pspinext}
that states that if all points in $\mathscr{C}_{N}\left(m_{N}-\kappa,m_{N}+\kappa\right)$
are good, then $\mathscr{G}_{N,\kappa}$ can be defined appropriately.
The last line in each of the first three displayed formulas in the
proof should be replaced by the following, respectively,
\begin{align*}
\sqrt{N}\bar{f}_{\boldsymbol{\sigma},apx}^{+}\left(Y_{\boldsymbol{\sigma}}\right) & =H_{N}\left(\sqrt{N}\boldsymbol{\sigma}\right)+\sqrt{\frac{2s_{N}}{N}}H_{N}^{\prime}\left(\sqrt{N}\boldsymbol{\sigma}\right)-2s_{N}C_{0}\left(1+o\left(1\right)\right),\\
\sqrt{N}\bar{f}_{\boldsymbol{\sigma}}^{+}\left(Y_{\boldsymbol{\sigma}}\right) & \leq\sqrt{N}\bar{f}_{\boldsymbol{\sigma},apx}^{+}\left(Y_{\boldsymbol{\sigma}}\right)+O\left(N^{-1/2}(s_{N}^{\prime})^{3/2}\right)\\
\inf_{x:\,\left\Vert x\right\Vert =w_{0}N^{-1/2}\sqrt{s_{N}}}\sqrt{N}\bar{f}_{\boldsymbol{\sigma}}^{+}\left(x\right) & \geq H_{N}\left(\sqrt{N}\boldsymbol{\sigma}\right)+\sqrt{\frac{2s_{N}}{N}}H_{N}^{\prime}\left(\sqrt{N}\boldsymbol{\sigma}\right)\\
 & +\left(\frac{1}{4}K_{p,\delta}w_{0}^{2}-\frac{1}{\sqrt{2}}w_{0}\right)s_{N}^{\prime}\left(1+o\left(1\right)\right).
\end{align*}
Other changes that are required are straightforward.\qed

\bibliographystyle{plain}
\bibliography{master}

\end{document}